\documentclass[11pt]{amsart}

\usepackage{subfiles}
\usepackage{lipsum}
\usepackage[font=itshape]{quoting}
\usepackage{tabu, fullpage,diagbox, booktabs}
\usepackage{placeins}
\usepackage[margin=1in]{geometry}
\usepackage[dvipsnames]{xcolor}   		             		
\usepackage{graphicx}			
\usepackage{amssymb,csquotes}
\usepackage{dsfont}
\usepackage{mathrsfs}
\usepackage{amsthm}
\usepackage{amsmath}
\usepackage{stmaryrd}
\usepackage{tikz,framed}
\usepackage{tikz-cd}
\usepackage{accents}
\usepackage{upgreek}
\usepackage{enumitem}
\usepackage{bm}
\usepackage{mathtools}
\usepackage[all]{xy}
\usepackage{caption}

\usepackage{url}
\usepackage{float}
\usepackage{todonotes}
\usepackage{bbm}
\usepackage{easyeqn}

\usepackage[full]{textcomp}
\usepackage[cal=cm]{mathalfa}
\usepackage{xparse}
\usepackage{comment}
\usepackage{cite}

\tikzset{
	commutative diagrams/.cd, 
	arrow style=tikz, 
	diagrams={>=stealth}
}

\theoremstyle{definition}

\makeatletter
\newcommand{\colim@}[2]{%
	\vtop{\m@th\ialign{##\cr
			\hfil$#1\operator@font colim$\hfil\cr
			\noalign{\nointerlineskip\kern1.5\ex@}#2\cr
			\noalign{\nointerlineskip\kern-\ex@}\cr}}%
}
\newcommand{\colim}{%
	\mathop{\mathpalette\colim@{\rightarrowfill@\textstyle}}\nmlimits@
}
\makeatother

\makeatletter
\def\@tocline#1#2#3#4#5#6#7{\relax
	\ifnum #1>\c@tocdepth % then omit
	\else
	\par \addpenalty\@secpenalty\addvspace{#2}%
	\begingroup \hyphenpenalty\@M
	\@ifempty{#4}{%
		\@tempdima\csname r@tocindent\number#1\endcsname\relax
	}{%
		\@tempdima#4\relax
	}%
	\parindent\z@ \leftskip#3\relax \advance\leftskip\@tempdima\relax
	\rightskip\@pnumwidth plus4em \parfillskip-\@pnumwidth
	#5\leavevmode\hskip-\@tempdima
	\ifcase #1
	\or\or \hskip 1em \or \hskip 2em \else \hskip 3em \fi%
	#6\nobreak\relax
	\dotfill\hbox to\@pnumwidth{\@tocpagenum{#7}}\par
	\nobreak
	\endgroup
	\fi}
\makeatother

\usetikzlibrary{calc}
\usetikzlibrary{fadings}
\usetikzlibrary{decorations.pathmorphing}
\usetikzlibrary{decorations.pathreplacing}

\newcounter{marginnote}
\DeclareMathAlphabet{\mathpzc}{OT1}{pzc}{m}{it}

\usepackage[backref=page]{hyperref}
\hypersetup{
	colorlinks   = true,          %Colors links instead of ugly boxes
	urlcolor     = blue,          %Color for external hyperlinks
	linkcolor    = purple,          %Color of internal links
	citecolor   = blue             %Color of citations
}

\theoremstyle{definition}
\newtheorem{maintheorem}{Theorem}

\newtheorem{theorem}{Theorem}[subsection]
\newtheorem{situation}[theorem]{Situation}

\newtheorem{mainconjecture}[maintheorem]{Conjecture}
\newtheorem{corollary}[theorem]{Corollary}
\newtheorem{lemma}[theorem]{Lemma}
\newtheorem{lemdefn}[theorem]{Lemma/Definition}
\newtheorem{proposition}[theorem]{Proposition}

\newtheorem{remark}[theorem]{Remark}

\newtheorem*{runningexample*}{Running example}

\newtheorem{constr}[theorem]{Construction}

\newtheorem{definition}[theorem]{Definition}

\newtheorem{example}[theorem]{Example}

\newtheorem{proposition-definition}[theorem]{Proposition-Definition}

{%
	\pushQED{\qed}\begin{constr}}
	{\popQED\end{constr}}

\newcommand{\bcd}{\begin{center}\begin{tikzcd}}
	\newcommand{\ecd}{\end{tikzcd}\end{center}}

\newcommand{\pt}{\text{pt}}

\usepackage{accents}

\DeclareMathAlphabet{\mathpzc}{OT1}{pzc}{m}{it}

\NewDocumentCommand{\compatibilitydatum}{m m m m m m O{} O{} O{}}{
	\begin{equation*} \begin{tikzcd}[ampersand replacement=\&]
	\: \arrow{r} \& {#1} \arrow{r} \arrow{d}{#7} \& {#2} \arrow{r} \arrow{d}{#8} \& {#3} \arrow{r}{[1]} \arrow{d}{#9} \& \: \\
	\: \arrow{r} \& {#4} \arrow{r} \& {#5} \arrow{r} \& {#6} \arrow{r} \& \:
	\end{tikzcd} \end{equation*}}

\NewDocumentCommand{\commutingsquare}{m m m m o O{} O{} O{} O{}}{
	\begin{equation}\begin{tikzcd}[ampersand replacement=\&] \label{#5}
	#1 \arrow{r}{#6} \arrow{d}{#7} \& #2 \arrow{d}{#8} \\
	#3 \arrow{r}{#9} \& #4
	\end{tikzcd}\IfValueTF{#5}{\label{#5}}{} \end{equation}}

\NewDocumentCommand{\cartesiansquarelabel}{m m m m m O{} O{} O{} O{}}{
	\begin{tikzcd}[ampersand replacement=\&]
	#1 \arrow{r}{#6} \arrow{d}{#7} \arrow[dr, phantom, "\square"] \& #2 \arrow{d}{#8} \\
	#3 \arrow{r}{#9} \& #4
	\end{tikzcd}\IfValueTF{#5}{\label{#5}}{}
}

\NewDocumentCommand{\triangleofspaces}{m m m O{} O{} O{}}{
	\begin{tikzcd} [ampersand replacement=\&]
	#1 \arrow{r}{#4} \arrow[bend right]{rr}{#5} \& #2 \arrow{r}{#6} \& #3
	\end{tikzcd}}

\newcommand{\stacks}[1]{\cite[Tag~\href{https://stacks.math.columbia.edu/tag/#1}{#1}]{stacks}}

\title{The logarithmic generalised Kummer locus\\ and its tropicalisation}
\author{Patrick Kennedy-Hunt}

\begin{document}

\begin{abstract}
    We construct a separated logarithmic compactification of a divisor in the moduli of polarised irreducible holomorphic symplectic (IHS) varieties of generalised Kummer type in dimension at least four. This compactification carries a universal family of logarithmic IHS varieties, and its tropicalisation carries a universal family of tropical generalised Kummer varieties. 

    The compactification is built from a more general construction of explicit, projective, and logarithmically smooth models for type~II and type~III degenerations of generalised Kummer varieties of arbitrary dimension and over base logarithmic schemes of any dimension. Over a base of any logarithmic rank, these degenerations come with a tropicalisation map to a family of integral affine manifolds with singularities, and with boundary in the type II case. This tropical family is constructed through a tropical generalised Kummer construction. 

    Our approach uses logarithmic Hilbert schemes to extend the conventional generalised Kummer construction to families of logarithmic abelian varieties. We describe fibres of our degenerations, and thus every point in the boundary of our compactification. Proper models correspond to polyhedral structures on the associated tropical generalised Kummer variety, with the resulting polyhedral complex serving as the dual intersection complex; each irreducible component is constructed by applying the generalised Kummer construction to a product of logarithmic group compactifications of semi-abelian varieties.
\end{abstract}
\maketitle
\tableofcontents

\section*{Introduction}

Beauville's \textit{generalised Kummer varieties} are, alongside Hilbert schemes of points on K3 surfaces, the central examples of irreducible holomorphic symplectic (IHS) varieties \cite{Beauville1983,Huybrechts1999}. 
Outside the K3 situation, where it was achieved in \cite{AlexeevEngel2023}, a modular compactification for the space of polarised IHS varieties remains largely open. This paper equips a compactification of a divisor in the moduli of polarised IHS varieties of generalised Kummer type with a universal family of \textit{logarithmic IHS varieties} (Definition~\ref{defn:LogIHS}), and offers an explicit description of the boundary points.

The prerequisite task of naming explicit strict normal crossing, or more generally logarithmically smooth, degenerations over a one-parameter base was not achieved outside of the $\text{K3}^{[n]}$ case \cite{GulbrandsenHalleHulek2019,GulbrandsenHalleHulekZhang2021,DegenHilbnK3,Tschanz2023,Tschanz2024} until recent work \cite{halle2026degenerationsgeneralizedkummervarieties}. \textit{Loc. cit.} achieve their strongest results for one-parameter type~II degenerations of generalised Kummer varieties of dimension at most four. The type~III generalised Kummer case has remained unexplored. Moreover, degenerations of IHS varieties over a base scheme of dimension greater than one have not been studied outside the K3 situation \cite{engel2026polyncvarietiesmultiparameterkulikov}. 

This paper uses \textit{logarithmic Hilbert schemes} \cite{LogQuot,MR20} and \textit{logarithmic abelian varieties} \cite{KajiwaraKatoNakayama2008} to extend the generalised Kummer construction to a natural class of logarithmically smooth degenerations of abelian varieties. Fixing $n$ a positive integer, our central object of study is the \emph{generalised Kummer locus}
$\mathcal{K}_{2n}^\circ$: a divisor in the moduli of polarised IHS varieties whose
closed points specify \textit{principal generalised Kummer varieties of dimension $2n$},
$$\operatorname{Kum}_{2n}(A) := \Sigma^{-1}(0) \subset A^{[n+1]}$$ equipped with
the polarisation of Definition~\ref{defn:polarisation}, where
$\Sigma \colon A^{[n+1]} \to A$ is the summation morphism induced by
the group law on a principally polarised abelian surface~$A$. Lemma \ref{lem:resolveComment} shows that the generalised Kummer locus is a connected component in a certain Noether--Lefschetz locus. 

We extend Beauville's arguments for Hodge-theoretic, topological, and smoothness properties of generalised Kummer varieties to the logarithmic setting, and thus establish that logarithmic generalised Kummer varieties satisfy a logarithmic version of the definition of IHS variety. Theorem~\ref{mainthm:Niceness properties} records a canonical and separated compactification of the generalised Kummer locus. It is obtained by combining Corollary~\ref{corr:familyIHS}, Proposition \ref{prop:NLlocus}, and Proposition~\ref{prop:KnBehaved}.

\begin{maintheorem}\label{mainthm:Niceness properties}
    Let $n \geq 2$ be an integer. The logarithmic generalised Kummer construction equips a compactification $\mathcal{K}_{2n}^\circ\hookrightarrow \mathcal{K}_{2n}$ of the generalised Kummer locus with a universal family 
     $$\pi_{2n} \colon \mathcal U_{2n} \longrightarrow \mathcal K_{2n}$$
     of logarithmic IHS varieties. Moreover, both $\mathcal{K}_{2n}$ and $\mathcal{U}_{2n}$ are proper \textit{logarithmic Deligne--Mumford (DM) stacks}, and thus admit a logarithmic \'etale cover by schemes with logarithmic structure.
\end{maintheorem}

%Subject to Conjecture \ref{conj:annoying}, we further compute the logarithmic variations of Hodge structure associated to points in the boundary of $\mathcal{K}_{2n}$, and thus establish that fibres of $\pi_{2n}$ over standard logarithmic closed points are non-isomorphic for distinct points. 
The reader is referred to Section \ref{sec:IntroLogAlgspaces} for a discussion of \textit{logarithmic spaces} (resp. \textit{logarithmic DM stacks}); these are a mild generalisation of conventional algebraic spaces (resp. DM stacks).\footnote{The functor of points of a logarithmic IHS variety on the category of logarithmic schemes is setoid valued - i.e. not stacky. Despite this fact, our logarithmic IHS varieties are sometimes constructed by equipping certain DM stacks (with non-trivial stabilisers) with logarithmic structure. The statement of Theorem \ref{mainthm:Niceness properties} therefore does not imply existence of logarithmically smooth Kulikov models in the category of schemes. Instead we will establish existence of good models in the category of DM stacks.}
After Kummer logarithmic \'etale sheafifying the functor of points of both spaces, on the level of coarse spaces\footnote{I.e. on the level of sheaves of isomorphism classes.}, the moduli space $\mathcal{K}_{2n}$ coincides with the moduli of principally polarised logarithmic abelian varieties. However, the automorphism groups differ: in $\mathcal{K}_{2n}$ the group of automorphisms is enlarged from the group $\operatorname{Aut}(A)$ of polarised group automorphisms of $A$ to $$(\mathbb{Z}/(n+1)\mathbb{Z})^4\rtimes \operatorname{Aut}(A).$$ Note that away from the logarithmic boundary, and choosing a polarisation so that all polarised automorphisms fix the exceptional divisor, the results of \cite{BOISSIERE2011553} identify the group displayed above with the automorphism group of $\operatorname{Kum}_{2n}(A)$. We expect this result extends to the logarithmic boundary (provided automorphism groups are taken in the logarithmic category).

The paper proceeds by developing three parallel versions of the generalised Kummer construction: tropical, logarithmic, and moduli theoretic. Together, these approaches establish the following outcomes:

\begin{enumerate}[leftmargin=*]
    \item \textbf{Compactifying a slice in the moduli space of IHS varieties.} Theorem~\ref{mainthm:ReadOffModel} explains that, after making a combinatorial choice, Theorem~\ref{mainthm:Niceness properties} specifies a DM stack (with logarithmic structure) which compactifies the generalised Kummer locus. This conventional (i.e. not logarithmic) DM stack is equipped with a universal family of singular generalised Kummer varieties.
    \item \textbf{Describing points in the boundary via explicit one-parameter degenerations.} We construct and describe the special fibres of one-parameter logarithmically smooth type~II and type~III degenerations of generalised Kummer varieties in every dimension. We recover each point in the boundary of $\mathcal{K}_{2n}$ as the special fibre of such a family. See Theorem~\ref{mainthm:1param}, Theorem~\ref{mainthm:Explicit}, and Section~\ref{sec:examples}. %These degenerations are logarithmic IHS varieties, even where the original abelian variety is not principal. 
    \item \textbf{Explicit presentations for integral affine manifolds with singularities.} Theorem~\ref{Thm:maintrop} and Theorem~\ref{mainthm:1param} describe the \textit{integral affine manifold with singularities} which controls the dual intersection complex of a one-parameter degeneration as described in (2). These data are described through a \textit{tropical generalised Kummer construction}.  
\end{enumerate}

In this paper, the phrase \textit{family of integral affine manifolds with singularities} includes families of structures which arise from type II degeneration, and are thus best considered integral affine manifolds with both singularities and boundary. For type III degenerations, the integral affine manifolds have no boundary: such integral affine manifolds are familiar as central objects used to compute the degeneration and gluing formulae exploited in enumerative geometry \cite{abramovich2021punctured,decompdegenGWinvar,gross2025remarksgluingpuncturedlogarithmic,MR23,Ran19}, as well as in the SYZ approach to mirror symmetry \cite{GrossSiebert2006,GrossSiebert2010,GSintrinsic,GSwall,KontsevichSoibelman2001,KontsevichSoibelman2006,SYZ1996}. A logarithmically smooth degeneration whose central fibre can be understood is a key step towards applying these degeneration and gluing techniques. Enumerative invariants are not fully understood for spaces in the generalised Kummer deformation class \cite{Oberdieck2022GWNL,Oberdieck2024HAE}. 

Whilst minimal model program (MMP) techniques ensure that any IHS variety $\mathcal{X}_\Delta\rightarrow \Delta$ over a trait admits a minimal model with mild singularities \cite{RmkDegenHK},
these \textit{Kulikov} models are far from unique, and naming examples requires running the MMP. In the K3 case, (2) was first achieved in the seminal work \cite{Kulikov1977,PerssonPinkham1981}, which has since been extended outside the complex setting in the Kummer K3 case \cite{Overkamp2021}. For Hilbert schemes of points on a K3 surface, the same outcome was obtained in the type~II case without logarithmic techniques\footnote{The author finds it helpful to imagine logarithmic geometry as a language for studying expansions. As such, the approach in works by subsets of Gulbrandsen, Halle, Hulek, and Zhang exploits the same underlying geometry as the logarithmic approach, although they do not use the logarithmic language. See Section \ref{sec:relatedlit}.} \cite{GulbrandsenHalleHulek2019,GulbrandsenHalleHulekZhang2021,Nagai2008,Nagai2018,Nagai2021}, and in the type~III case \cite{DegenHilbnK3,Tschanz2023,Tschanz2024} via logarithmic Hilbert schemes of points \cite{LogQuot,MR20}. Degenerations of generalised Kummer type are constructed in \cite{halle2026degenerationsgeneralizedkummervarieties} which we discuss further in Section~\ref{sec:relatedlit}.

\subsection{The role of logarithmic spaces}\label{sec:IntroLogAlgspaces} A central theme of this paper is that \textit{logarithmic spaces} and \textit{logarithmic DM stacks}, together with associated tropical objects, give a functorial way to organise non-canonical choices in the category of schemes. Such non-canonical choices would otherwise obstruct each of the outcomes described above, and always arise in constructing a strict normal crossing degeneration, see \cite{logcoh,FeusiMolchoOrbifoldCompactifiedJacobians,KajiwaraKatoNakayama2008,KajiwaraKatoNakayamaPartIV,KajiwaraKatoNakayamaPartVII,LogQuot,Quot2,MolchoWise} for other applications. We will now explain how, after making a combinatorial choice, a theorem about schemes can be read off from each of our theorems describing logarithmic spaces.

Studying degenerations of generalised Kummer varieties arising from a degenerating abelian variety is in some ways more subtle than the parallel $\text{K3}^{[n]}$ situation. This is because the generalised Kummer construction relies on the group structure of an abelian variety. (Non-trivial) strict normal crossing degenerations of abelian surfaces within the category of schemes never preserve the group structure: group objects over the complex numbers are smooth \stacks{047N}. 

An idea is required to extend the construction as the surface degenerates and the group law is lost.
Working with {logarithmic spaces}, it is possible to degenerate abelian varieties to proper objects, called \textit{logarithmic abelian varieties} (Section \ref{sec:log-abelian-variety}), in a way which preserves the group structure \cite{KajiwaraKatoNakayama2008,KajiwaraKatoNakayamaPartIV}. 

\textit{Logarithmic DM stacks} play a further essential role: Specifying a logarithmically smooth degeneration of an IHS variety over a trait within the category of schemes requires a combinatorial choice. This fact introduces a challenge when compactifying the generalised Kummer locus: choices must be made in a compatible way. 

We now explain the most troublesome source of combinatorial choice for this paper. The special fibres of both type~II and type~III degenerations of IHS varieties are reducible. 
Given any scheme flat over a trait with reducible and connected central fibre, one can produce a new model by blowing up the locus where two or more components meet.\footnote{Performing this move to the scheme with logarithmic structure associated to the left diagram of Figure \ref{fig:firstTropAv} corresponds to a different combinatorial choice of subdivision. In particular, blowing up the intersection of several components associated to dots $v_1,...,v_k$ in Figure \ref{fig:firstTropAv} corresponds to adding an extra vertex and associated edges into the cone spanned by $v_1,...,v_k$.} When the blowup is performed along an ideal which is locally monomial in equations of the irreducible components of the central fibre, it is the prototypical example of a \textit{logarithmic alteration}. For the type~III degenerations of generalised Kummer varieties we consider, the resulting system of models has no terminal element in the category of schemes, but there is a terminal choice in the category of logarithmic spaces. 

\begin{remark}\label{rem:promise}
    Every time we introduce a logarithmic space or logarithmic DM stack $\mathcal{X}$ which is not a DM stack with logarithmic structure, we will also specify an associated tropical object $\mathpzc{X}$.\footnote{We reserve the latter font for stacks on the tropical site (i.e. the face site of \cite{ModStckTropCurve}).} A logarithmic space specifies an equivalence class of schemes with logarithmic structure. Elements of the equivalence class are called logarithmic alterations of $\mathcal{X}$. The dual intersection complex of each logarithmic alteration of $\mathcal{X}$ admits a canonical map to the link of a cone complex subdividing $\mathpzc{X}$. This map is often an isomorphism. Figure \ref{fig:firstTropAv} illustrates the discussion of Section \ref{sec:logspaces}, which explains how logarithmic alterations can be read--off from a combinatorial choice: a cone complex structure on the attendant tropical gadget. 
\end{remark}
\begin{figure}
\begin{tikzpicture}[scale=0.85, line cap=round, line join=round]

  %======================
  % Left: subdivision Σ
  %======================
  \begin{scope}
    \coordinate (O) at (0,0);
    \coordinate (A) at (5,0.3);
    \coordinate (B) at (1.6,3.4);
    \coordinate (C) at ($(A)+(B)$);

    % boundary
    \draw[thick] (O) -- (A) -- (C) -- (B) -- cycle;

    % subdivision vertices
    \coordinate (P1) at ($(O)!0.4!(A)$);
    \coordinate (P2) at ($(O)!0.7!(B)$);
    \coordinate (P3) at ($(A)!0.6!(C)$);
    \coordinate (P4) at ($(B)!0.5!(C)$);

    % subdivision edges
    \draw[very thick, blue!70!black] (P1) -- (P2);
    \draw[very thick, blue!70!black] (P1) -- (P3);
    \draw[very thick, blue!70!black] (P2) -- (P3);
    \draw[very thick, blue!70!black] (P2) -- (P4);
    \draw[very thick, blue!70!black] (P3) -- (P4);

    % dots
    \foreach \pt in {P1, P2, P3, P4} {
        \fill (\pt) circle (2pt);
    }

    % edge identifications
%    \draw[->, very thick, blue] ($(B)+(-0.15,0.28)$) -- ($(C)+(-0.15,0.28)$);
%    \draw[->, very thick, blue] ($(A)+(0.15,-0.28)$) -- ($(O)+(0.15,-0.28)$);

    \draw[->, very thick, red] ($(O)+(-0.35,0.15)$) -- ($(B)+(-0.35,0.15)$);
    \draw[->, very thick, red] ($(A)+(0.35,0.15)$) -- ($(C)+(0.35,0.15)$);

    \node[blue] at ($(B)!0.5!(C)+(0,0.55)$) {\small identify};
    \node[blue] at ($(O)!0.5!(A)+(0,-0.55)$) {\small identify};
    \node[red, rotate=65] at ($(O)!0.5!(B)+(-0.65,0)$) {\small identify};
    \node[red, rotate=65] at ($(A)!0.5!(C)+(0.65,0)$) {\small identify};

  \end{scope}

  % arrow
  \draw[->, thick] (7.2,2.0) -- (9.0,2.0);

  %======================
  % Right: tropical abelian variety
  %======================
  \begin{scope}[xshift=10.5cm]
    \coordinate (O2) at (0,0);
    \coordinate (A2) at (5,0.3);
    \coordinate (B2) at (1.6,3.4);
    \coordinate (C2) at ($(A2)+(B2)$);

    % boundary
    \draw[thick] (O2) -- (A2) -- (C2) -- (B2) -- cycle;

    % edge identifications
    \draw[->, very thick, blue] ($(B)+(-0.15,0.28)$) -- ($(C)+(-0.15,0.28)$);
    \draw[->, very thick, blue] ($(O)+(0.15,-0.28)$) -- ($(A)+(0.15,-0.28)$);
    % horizontal edge identifications: torus orientation
    \draw[->, very thick, blue] ($(B2)+(-0.15,0.28)$) -- ($(C2)+(-0.15,0.28)$);
    \draw[->, very thick, blue] ($(O2)+(0.15,-0.28)$) -- ($(A2)+(0.15,-0.28)$);
    \draw[->, very thick, red] ($(O2)+(-0.35,0.15)$) -- ($(B2)+(-0.35,0.15)$);
    \draw[->, very thick, red] ($(A2)+(0.35,0.15)$) -- ($(C2)+(0.35,0.15)$);

    \node[blue] at ($(B2)!0.5!(C2)+(0,0.55)$) {\small identify};
    \node[blue] at ($(O2)!0.5!(A2)+(0,-0.55)$) {\small identify};
    \node[red, rotate=65] at ($(O2)!0.5!(B2)+(-0.65,0)$) {\small identify};
    \node[red, rotate=65] at ($(A2)!0.5!(C2)+(0.65,0)$) {\small identify};

  \end{scope}

\end{tikzpicture}
\caption{The reader can imagine that a family of tropical abelian surfaces is a family over a base cone of real two-tori equipped with an integral affine atlas. One such two-torus is depicted above right. Above left is the polyhedral subdivision induced by a choice of cone complex structure on the total space of such a family of tropical abelian varieties. The logarithmic alteration associated to this subdivision has two components (corresponding to the two vertices in our subdivision).}
\label{fig:firstTropAv}
\end{figure}
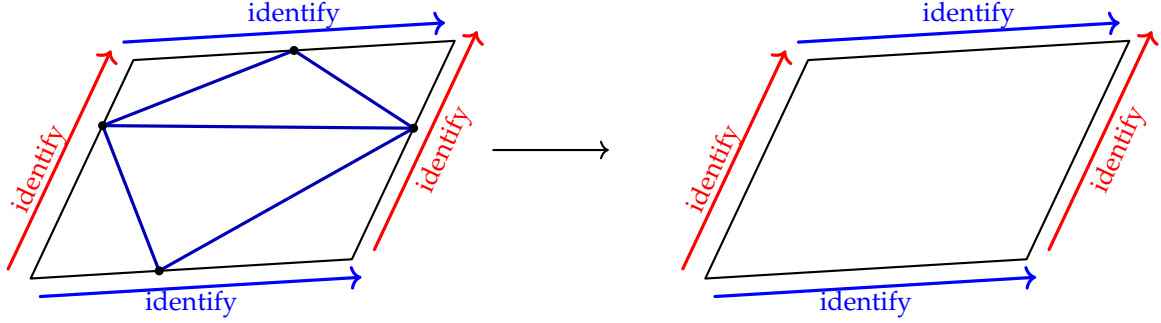

\subsection{The tropical generalised Kummer construction}\label{sec:TropGenKummerint} Section \ref{sec:TropGenKummer} carries out the first step in describing points of $\mathcal{K}_{2n}$ by naming the tropical gadget, which we call a \textit{family of integral affine manifolds with singularities} (IAMS), associated to a family of generalised Kummer varieties over a base $S$. These \textit{families of tropical generalised Kummer varieties} are obtained by applying the generalised Kummer construction in the tropical category. The role of an abelian surface is played by a \textit{tropical abelian variety}, whose underlying geometry is the integral affine manifold obtained as a quotient of $\mathbb{R}^s$ by a lattice (cf. Figure \ref{fig:firstTropAv}). Points associated to the special fibre of a type~II degeneration are characterised by $s=1$, while those from type~III occur precisely when $s=2$. The role of a Hilbert scheme of points is replaced by a tropical version of a Hilbert scheme, called the \textit{stack of tropical supports}, introduced in \cite{LogQuot}. 

Section \ref{sec:TropGenKummerlocus} mirrors the construction behind Theorem~\ref{mainthm:Niceness properties} in the tropical setting to obtain a stack on the face site $\mathpzc{K}_{2n}$ equipped with a universal family of tropical generalised Kummer varieties. Our next theorem is obtained by combining Proposition~\ref{prop:TypeIHomeotype}, Proposition~\ref{prop:TypeIIHomeotype}, and Lemma~\ref{lem:squarefromMainThm}. To parse the statement, recall that $a^\ast$ is the functor which associates to a stack on the face site a category fibred in groupoids (CFG) over the category of logarithmic schemes, see Section \ref{sec:logschemesandtrop}.

\begin{maintheorem}\label{Thm:maintrop}
Let $n\geq 2$. The tropical generalised Kummer construction  equips $\mathpzc{K}_{2n}$ with a universal family 
$$\mathpzc{U}_{2n}\rightarrow \mathpzc{K}_{2n}$$ of integral affine manifolds with singularities. The topological realisation of each fibre is homeomorphic to either $\mathbb{CP}^n$, an $n$-simplex, or a point. There is a commutative square (in the category of CFG over the category of logarithmic schemes) with strict horizontal maps $$\begin{tikzcd}
\mathcal{U}_{2n} \arrow[r,"\hat{\varpi}"] \arrow[d]
  & a^{\ast}  \mathpzc{U}_{2n} \arrow[d] \\
\mathcal K_{2n} \arrow[r, "\varpi"]
  & a^{\ast} \mathpzc{K}_{2n}.
\end{tikzcd}
$$  
\end{maintheorem}

A parallel to the square in Theorem~\ref{Thm:maintrop} is understood explicitly for moduli of logarithmically smooth curves \cite{ModStckTropCurve}, and implicitly for moduli of logarithmic abelian varieties \cite{KajiwaraKatoNakayama2008,KajiwaraKatoNakayamaPartVII}. Theorem~\ref{Thm:maintrop} is the first to extend this strategy to the IHS setting. 
Theorem~\ref{mainthm:ReadOffModel} demonstrates the relevance of such a tropical diagram when working with logarithmic DM stacks.

\begin{maintheorem}[{Theorem~\ref{thm:GettingModels}}]\label{mainthm:ReadOffModel}
    Let $n\geq 2$. Consider any of the (infinitely many) choices, indexed by $\Lambda$, of diagrams of stacks upon the tropical site, 
    $$
    \begin{tikzcd} 
    \mathpzc{U}_{2n}^\Lambda \arrow[r] \arrow[d,"p"]
  & \mathpzc{U}_{2n} \arrow[d] \\
\mathpzc K_{2n}^\Lambda \arrow[r]
  & \mathpzc{K}_{2n}.
\end{tikzcd}$$
    where the horizontal maps are tropical alterations (i.e. compositions of subdivisions and changes of lattice), and $p$ is a saturated morphism of cone stacks. The base change of the diagram of Theorem~\ref{Thm:maintrop} yields a vertical, saturated, and logarithmically smooth morphism 
    $$\mathcal{U}_{2n}^\Lambda = \mathcal{U}_{2n}\times_{a^\ast \mathpzc{U}_{2n}}a^\ast \mathpzc{U}^\Lambda_{2n}\rightarrow \mathcal{K}_{2n}^\Lambda = \mathcal{K}_{2n}\times_{a^\ast \mathpzc{K}_{2n}} a^\ast \mathpzc{K}^\Lambda_{2n}$$
    in the category of DM stacks with logarithmic structure. Moreover, the open immersion $\mathcal{K}_{2n}^\circ \hookrightarrow \mathcal{K}_{2n}^\Lambda$ is a toroidal compactification.
\end{maintheorem}

\subsection{The geometry of one-parameter degenerations: dual intersection complexes}\label{sec:explicitDescription}
We will now describe the central fibre of a one-parameter degeneration of logarithmic generalised Kummer varieties (in both the type~II and type~III situations). Applying these results to the principally polarised case offers a description of points in the boundary of $\mathcal{K}_{2n}$. 

In the sequel $\Delta$ will denote a curve smooth over $\mathbb{C}$ with divisorial logarithmic structure from a closed point $0$. Let $A\rightarrow \Delta$ be a family of logarithmic abelian surfaces. Write $\mathpzc{A}\rightarrow \mathbb{N}$ for the associated family of tropical abelian varieties. We choose tropical alterations, indexed by $\Lambda$, of the resulting family of tropical generalised Kummer varieties. There is a corresponding logarithmic alteration of the $\Delta$-family of logarithmic Kummer varieties:
$$\operatorname{Kum}_{2n}^\Lambda(A) \rightarrow \operatorname{Kum}_{2n}(A)\rightarrow \Delta, \quad \quad \Lambda \rightarrow \operatorname{Kum}_{2n}(\mathpzc{A}) \xrightarrow{p} \mathbb{N}.$$
Section~\ref{sec:TropGenKummerIAMSing} interprets $p$ as the data of an integral affine manifold with singularities $B_K^{2n}(\mathpzc{A})$ (which has boundary in the type II situation). See Section \ref{sec:examples} for examples. The choice of subdivision $\Lambda$ specifies a polyhedral complex $B_\Lambda$ equipped with a ``subdivision" morphism to $B_K^{2n}(\mathpzc{A})$.
\begin{maintheorem}[Corollary \ref{corr:DualIntCx}]\label{mainthm:1param} Let $n\geq 1$ be any positive integer.
    The dual intersection complex of central fibre to the one-parameter degeneration $$\pi \colon \operatorname{Kum}_{2n}^\Lambda(A) \rightarrow \operatorname{Kum}_{2n}(A) \rightarrow \Delta$$
    is canonically the polyhedral complex $B_\Lambda$. %Moreover, assume the degeneration is type~III and satisfies the following hypotheses,

\end{maintheorem}

%We expect, but do not prove, that (up to subdivision and saturation) $\pi$ is always dlt and minimal on coarse spaces, and thus Theorem~\ref{mainthm:1param} holds unconditionally. The difficulty is proving that the coarse space associated to a certain DM simple normal crossing pair, equipped with its boundary, is dlt. 
Since this paper works with logarithmically smooth DM stacks and not dlt degenerations, we can not apply the result of \cite{NicaiseXu2016} to reprove the fact that the essential skeleton of our degeneration coincides with $B_\Lambda$.
%Note that Theorem~\ref{mainthm:1param}, together with the homeomorphism type of fibres described in Theorem~\ref{Thm:maintrop}, is consistent with 
The essential skeleton of the one-parameter degenerations we construct is already computed in \cite[Theorem~1.7.2]{BrownMazzon}, and the essential skeleton of \textit{loc. cit.} is PL homeomorphic to $B_\Lambda$. 

We emphasise that the integral affine manifold $B_K^{2n}(\mathpzc{A})$ is strictly more data than the dual polyhedral complex $B_\Lambda$. In particular, it is more data than the PL homeomorphism type of the essential skeleton of our degeneration in the sense of \cite{NicaiseXu2016}. The latter tracks only the PL homeomorphism type, while the former retains an integral affine atlas (with singularities). For type~III degenerations this integral affine manifold is expected to coincide with the base of the SYZ fibration in the sense of \cite[Section 5 and Section 6]{NicaiseXuYu2019NonArchimedeanSYZ}. 

\subsection{Irreducible components of the special fibre}\label{sec:IrredCompSpecialFib} It remains to describe the isomorphism type of each irreducible component in the special fibre of a type $s+1\in \{2 = \text{II},3= \text{III}\}$ degeneration $\operatorname{Kum}_{2n}^\Lambda(A)/\Delta$. Following Theorem~\ref{mainthm:1param}, points in the dual intersection complex to the central fibre of $\operatorname{Kum}_{2n}^\Lambda(A)$ specify a point of $B_K^{2n}(\mathpzc{A})$, which is homeomorphic to a subset of $\operatorname{Sym}^{n+1}(\mathbb{R}^s/\mathbb{Z}^s)$. To each vertex of $B_\Lambda$, and thus to each irreducible component of $\operatorname{Kum}_{2n}^\Lambda(A)$, we have associated a tuple of positive integers which sum to $n+1$ called the \textit{collision type}. The collision type tracks how the $n+1$ points parameterised by $\operatorname{Sym}^{n+1}(\mathbb{R}^s/\mathbb{Z}^s)$ meet. 

The logarithmic abelian variety $A_0 = A\times_\Delta 0$ contains a canonical open semi-abelian variety which we denote $A^\circ$, defined as the kernel of the map from $A$ to its tropicalisation.
In fact, $A_0$ is an $A^\circ$ bundle over $a^\ast \mathpzc{A}$ where $\mathpzc{A}$ is a tropical abelian variety. 
We construct in Definition~\ref{defn:Alog} a proper group object $A^{\operatorname{log}}$ which contains $A^\circ$ as an open subgroup.
The space $A^{\operatorname{log}}$ is obtained by compactifying the torus directions of $A$ using the \textit{logarithmic multiplicative group $\mathbb{G}_m^{\log}$.}
The associated tropical object is thus $\mathbb{R}^s$, and its logarithmic alterations are compactifications of $A^\circ$ by toric (stack) bundles.

In Section \ref{sec:LogHilbCompactSAV} we make sense of the logarithmic Hilbert scheme $(A^{\log})^{[m]}$ of length $m$ subschemes in $A^{\log}$. Given positive integers $n_1,...,n_k$ we denote the unordered tuple $(n_1,...,n_k) = \textbf{n}$ and write $$\Sigma\colon (A^{\log})^{[n_1]}\times(A^{\log})^{[n_2]}\times ...\times (A^{\log})^{[n_k]}\rightarrow A^{\log}\quad \quad \operatorname{Kum}_{\mathbf{n}}(A^{\operatorname{log}}) = \Sigma^{-1}(0)$$ where $\Sigma$ is the summation morphism induced by the group law of $A^{\log}$. Applying the same construction with $A^\circ$ in place of $A^{\log}$ yields a dense open subset $\operatorname{Kum}_{\mathbf{n}}(A^{\circ})$. A logarithmically smooth morphism is said to be semistable if it is saturated, vertical, and both source and target have locally free logarithmic structure.

\begin{maintheorem}[Theorem~\ref{prop:ExplicitDescription}]\label{mainthm:Explicit}
    Let $n\geq 1$ be any positive integer. Choosing $\Lambda$ such that $$\operatorname{Kum}_{2n}^\Lambda(A)\rightarrow \Delta$$ is semistable (and in particular saturated), the irreducible components of the central fibre admit the following explicit description: 
    \begin{enumerate}[leftmargin = *]
        \item \textbf{(Scheme-theoretic, interior only.)} The underlying DM stack of the interior of each irreducible component of the central fibre is isomorphic to $\operatorname{Kum}_{\mathbf{n}}(A^{\circ})$ where the tuple $\textbf{n}=(n_1,...,n_k)$ is the collision type associated to this irreducible component.
        \item \textbf{(Logarithmic category, whole irreducible component.)} The underlying DM stack of each irreducible component of the central fibre is a logarithmic alteration of $\operatorname{Kum}_{\mathbf{n}}(A^{\operatorname{log}})$ where the tuple $\textbf{n}= (n_1,...,n_k)$ is the associated collision type. 
    \end{enumerate}
\end{maintheorem}

\subsection{Outlook: towards logarithmic IHS geometry}\label{sec:outlook}
The purpose of this section is to explain how the results of this paper fit into a broader strategy for studying moduli and geometry of singular IHS varieties. We begin by explaining how we expect our results to generalise outside the generalised Kummer locus.

The phrase \textit{descending logarithmic IHS variety} refers to a logarithmic IHS variety whose functor of points satisfies descent in the full logarithmic \'etale topology. These are examples of the structure captured by \textit{descending logarithmic spaces}, as introduced in \cite[Section 4]{logcoh}. A theme of this paper is the expectation that moduli of descending logarithmic spaces offers a useful paradigm for understanding modular compactifications to the space of polarised IHS varieties. 

Given $X$ any algebraic space with logarithmic structure, there is a morphism from $X$ to the descending space $\operatorname{Sh}(X)$ obtained by sheafifying the functor of points of $X$ in the full log \'etale topology. This assignment is far from injective, and in particular identifies any two algebraic spaces with logarithmic structure which admit a common logarithmic alteration. Whilst the spaces (e.g. the universal family) studied in this paper are not descending logarithmic spaces, we expect that points in the boundary are in natural bijection with a suitable class of descending logarithmic spaces. Similarly, we expect that there is a natural bijective correspondence between the logarithmic Quot spaces of \cite{LogQuot} (resp. the logarithmic Picard groups of \cite{MolchoWise}, the logarithmic abelian varieties of \cite{KajiwaraKatoNakayama2008}) and a natural class of descending logarithmic spaces. The correspondence between descending logarithmic spaces and the models studied in the respective \textit{loc. cit.} will be developed elsewhere. 

The motivation for studying descending logarithmic spaces is two-fold: first, moduli of logarithmic coherent sheaves, including logarithmic Hilbert and Quot schemes, are native to the category of descending logarithmic spaces i.e. the Hilbert space associated to a logarithmic algebraic space is a logarithmic algebraic space: see for example \cite[Lemma 4.1.3]{logcoh}. For this reason the logarithmic Hilbert space of a scheme with logarithmic structure need not be a scheme with logarithmic structure, see \cite[Section 9.4]{LogQuot}; by contrast the logarithmic Hilbert scheme of a descending logarithmic space is (expected but not yet proved to be) a descending logarithmic space. More generally, \cite{logcoh} assigns a theory of coherent sheaves to any descending logarithmic space from which one can construct associated moduli spaces and name a derived category. This paper exploits the following consequences: in working with moduli of descending logarithmic spaces, constructions involving coherent sheaves on IHS varieties can be extended to the boundary in a functorial way. This strategy is at the heart of Theorem~\ref{mainthm:Niceness properties}.

Second, Theorem~\ref{mainthm:Niceness properties} offers evidence that working with moduli of descending logarithmic IHS varieties simplifies compactification problems by removing one of the sources of choice in naming the special fibre in a one-parameter degeneration. Under the correspondence sketched in the third paragraph of this subsection, this observation offers an alternative perspective on the moduli of logarithmic abelian varieties \cite{KajiwaraKatoNakayamaPartVII}. 

Olsson studied moduli of logarithmic K3 surfaces \cite{Olsson2004}, but he did not work with descending logarithmic spaces. This is one of the reasons why the moduli space of \textit{loc. cit.} is non-separated. Another reason is that in constructing a Kulikov model for a K3 surface one makes choices in addition to a choice of logarithmic alteration. One such choice is related to the possibility of a floppable curve in the central fibre. In conversations with Robert Crumplin, we came to expect that two degenerations related by the simplest examples of floppable curves are $\ell$-equivalent in the sense of \cite[Section 4]{logcoh}. This observation leads to the following conjecture, which admits an obvious generalisation to the higher dimensional IHS situation, and can be compared to \cite[Theorem G]{logcoh}.

\begin{mainconjecture}\label{Conj}
    Let $\operatorname{M}_{\operatorname{K3}}^{\log}$ denote the moduli stack of polarised descending logarithmic IHS surfaces, and let $\operatorname{M}_{\operatorname{tor}}$ denote the space obtained by log \'etale sheafifying the functor of points of any toroidal compactification of the moduli of polarised K3 surfaces. There is a commutative square 
    $$% https://tikzcd.yichuanshen.de/#N4Igdg9gJgpgziAXAbVABwnAlgFyxMJZABgBpiBdUkANwEMAbAVxiRAB12I0YAnOnBF5g6AWxjAAsgF8A+sADSAZmkA9YJwYQA5tJDTS6TLnyEUARnJVajFm07c+AoSPFS5Grj36DhYib7SegZG2HgERGTm1vTMrIgcXk6+rhIy8spqno4+Lv7AOLzcQfqGIBhhpkSW0dSxdgmcAPLi2nTqDt7Ofm6FxcHWMFDa8ESgAGZFokhkIIJIljZx9uxoWKUTUzPU84gATCEgkxDTiIu7KmXHp3s7EEgqFNJAA
\begin{tikzcd}
\operatorname{M}_{K3}^{\log} \arrow[r, "\pi"] \arrow[d] & \operatorname{M}_{\operatorname{tor}} \arrow[d] \\
a^\ast \operatorname{M}_{K3}^{\operatorname{trop}} \arrow[r,"\pi^{\operatorname{trop}}"]   & a^\ast \Omega^{\operatorname{trop}}                   
\end{tikzcd}$$
where $\pi$ is a categorical moduli space for the category of descending logarithmic spaces; $\operatorname{M}_{K3}^{\operatorname{trop}}$ is an (as yet undefined) moduli space of integral affine manifolds with singularities; and $\Omega^{\operatorname{trop}}$ is an (as of yet undefined) appropriate tropical period domain parallel to the constructions of \cite{ChanEtAlTropSymm}.
\end{mainconjecture}

The morphism $\pi$ (resp. $\pi^{\operatorname{trop}}$) can be thought of as a logarithmic (resp. tropical) period map, and (despite being of relative dimension zero) is very far from being injective. Indeed, the analytifications of toroidal models are log alterations of the Kato--Usui compactified period domain \cite{KatoUsui2009}.
It is possible that compactifications of the moduli of IHS varieties equipped with a universal family of descending logarithmic IHS varieties can be obtained by naming any section of $\pi^{\operatorname{trop}}$ together with a fan structure upon the tropical period domain, but further work is required to verify this.

Conjecture \ref{Conj} proposes that (up to logarithmic \'etale sheafification) the logarithmic approach does not offer a new compactification of the moduli of IHS varieties. 
We explain how Theorem~\ref{mainthm:Niceness properties} fits with Conjecture \ref{Conj}. Let $\mathsf{K}$ be any toroidal compactification of the connected component containing $\mathcal{K}_{2n}^\circ$ in the moduli of IHS varieties such that the closure of $\mathcal{K}_{2n}^\circ$ intersects the boundary in a transverse fashion (formally, we ask that the closure of $\mathcal{K}_{2n}^\circ$ equipped with its strict logarithmic structure is logarithmic flat). We expect, but do not prove here, that $\mathsf{K}$ admits a canonical logarithmic \'etale morphism to the closure of $\mathcal{K}_{2n}^\circ$ in $\mathcal{K}_{2n}$, and thus Theorem~\ref{mainthm:Niceness properties} equips (a divisor in) certain toroidal compactifications with a universal family of logarithmic IHS varieties.

\subsection{Outline}\label{sec:outline}
The technical heart of this paper involves describing and establishing geometric properties of several logarithmic spaces. The following table summarises the logarithmic spaces we study, and the associated tropical gadgets promised by Remark \ref{rem:promise}.
\begin{center}
\begin{tabular}{p{7cm}|c|p{6cm}|c}
\textbf{Logarithmic space} & \textbf{Section} 
& \textbf{Tropical object} & \textbf{Section} \\
\hline

Log abelian variety $A/S$
& \ref{sec:log-abelian-variety}
& Family of tropical abelian varieties $\mathpzc{A}/\sigma$
& \ref{sec:Tropical AV} \\

Moduli of log abelian varieties with level structure $\mathcal{A}_g(n)$ 
& \ref{sec:ModulilogAVLevelStr}
& Moduli of tropical abelian varieties with level structure $\mathfrak{A}_g(n)$
& \ref{sec:modTropAVwithLevel} \\

Universal log abelian variety $\mathcal{Y}_g(n) \rightarrow \mathcal{A}_g(n)$ 
& \ref{sec:ModulilogAVLevelStr}
& Universal tropical abelian variety $\mathpzc{Y}_g(n)\rightarrow \mathfrak{A}_g(n)$
& \ref{sec:modTropAVwithLevel} \\

Log Hilbert scheme of points $A^{[n+1]}/S$
& \ref{sec:LogHilbPts}
& Stack of tropical supports $\mathpzc{A}^{[n+1]}/\sigma$
& \ref{sec:TropHilbSch} \\

Log generalised Kummer variety $\operatorname{Kum}_{2n}(A)/S$
& \ref{sec:LogKummer}
& Tropical generalised Kummer variety $\operatorname{Kum}_{2n}(\mathpzc{A})/\sigma$
& \ref{sec:TropGenKummer} \\

Log generalised Kummer locus $\mathcal{K}_{2n}$
& \ref{sec:LogGenKummerLocus}
& Tropical generalised Kummer locus $\mathpzc{K}_{2n}$
& \ref{sec:TropGenKummerlocus} \\

Universal log generalised Kummer family $\mathcal{U}_{2n} \to \mathcal{K}_{2n}$
& \ref{sec:LogGenKummerLocus}
& Universal tropical generalised Kummer family $\mathpzc{U}_{2n} \to \mathpzc{K}_{2n}$
& \ref{sec:TropGenKummerlocus}
\end{tabular}
\end{center}

\subsection{Related literature}\label{sec:relatedlit}
In \cite{halle2026degenerationsgeneralizedkummervarieties}, the authors adopt a different approach to constructing degenerations of generalised Kummer varieties. Over the next three paragraphs we summarise how the approach of \textit{op. cit.} relates to the logarithmic perspective. Starting with a type~II degeneration of abelian surfaces $A\rightarrow C$ over a curve $C$ which is smooth away from $0\in C$, they pass to a good degeneration $I^n_{A/C}/C$ of its relative Hilbert scheme of points as constructed in \cite{GulbrandsenHalleHulek2019}. The central object of study in \cite{halle2026degenerationsgeneralizedkummervarieties} is the closure of the generalised Kummer variety associated to the family of conventional abelian varieties $A\times_{C}(C\backslash 0)$ inside $I^n_{A/C}/C$. The degeneration $I^n_{A/C}/C$ is constructed as the GIT quotient of a certain relative Hilbert scheme of points $\operatorname{Hilb}^n(A[n]/C[n])\rightarrow C[n]$ by the action of a torus. The space $I^n_{A/C}/C$ is roughly the moduli space of points on expansions where points are not permitted to meet the singular locus. The GIT quotient is required because $C[n]$ is not simply the moduli space of expansions: instead it also has torus factors.

On a heuristic level, the logarithmic approach adopted by Maulik and Ranganathan \cite{MR20} is to take the stack quotient of $A[n]/C[n]$ by the above extraneous torus actions \textbf{before} passing to the relative Hilbert scheme of points. The authors then form the relative Hilbert scheme of (logarithmically flat, i.e. disjoint from the singular locus) points of the resulting stack quotient - this is an open in the usual relative Hilbert scheme. Finally, a closed subset is removed to leave a separated DM stack. This approach has been developed for moduli of subschemes of dimension at most one.

The approach of \cite{LogQuot}, which we adopt and which works in greater generality, is a shift in perspective on the approach of Maulik and Ranganathan. Instead of starting with a space of expansions, one first names a moduli functor parameterising equivalence classes of subschemes on expansions. The geometry of this functor is controlled by building a \textit{tropical support} map to the \textit{stack of tropical supports}. The stack of tropical supports is simply the moduli space of (minimal) expansions (i.e. the map forgets the embedded subscheme, and any parts of the expansion which were not needed to house that subscheme), and this moduli space of expansions can be entirely understood in terms of tropical geometry: the stack of tropical supports is obtained by applying $a^\ast$ to a tropical version of the Hilbert scheme. More recently, the functor of points of the logarithmic Quot scheme as described in \cite{LogQuot} has been reinterpreted (up to subtleties concerning logarithmic \'etale descent) as a moduli space of logarithmic flat quotients of the structure sheaf in the category of logarithmic coherent sheaves \cite{logcoh}.

We now summarise the differences between our results and those of \cite{halle2026degenerationsgeneralizedkummervarieties}. At a high level our results extend those of \cite{halle2026degenerationsgeneralizedkummervarieties} on three axes: we work over a base of any dimension, not just a one-parameter degeneration, allowing us to compactify a moduli space; we handle the type~II and type~III situations, whereas op.~cit.\ treats only type~II; and our degenerations have good properties (including logarithmic smoothness) in every dimension. In op.~cit. the relevant smoothness property (being dlt) is established only for $n=1$ (surfaces), conjectured for $n=2$ (fourfolds), and known to fail for $n= 4$ (see \cite[Example 6.21]{halle2026degenerationsgeneralizedkummervarieties}). 

Moreover, we obtain a complete description of the dual intersection complex for all $n$. At $n=2$ the dual intersection complex presented here is a simpler polyhedral structure on the same topological space: we recover the standard $2$-simplex without subdivision (see Figure~\ref{fig:DualIntCx}), whereas op.~cit. obtains a non-trivial triangulation of it. We remark also that our degenerations are canonical in a slightly stronger sense than \textit{op. cit.}, whose constructions depended on a parameter $N$ which tracked a chosen type~II degeneration of the underlying abelian surface. The translation to a scheme with logarithmic structure is only required in the type III case, which is outside of the scope of \textit{op. cit.}
We thus address several of the open problems posed in \cite[Section 1.6]{halle2026degenerationsgeneralizedkummervarieties}. In particular, the tropical generalised Kummer construction clarifies the relationship between the dual intersection complexes of the logarithmic Hilbert scheme and the logarithmic generalised Kummer construction.

\subsection*{Acknowledgements} The origin of this paper is a conversation between the author and Calla Tschanz which took place at the Oberwolfach workshop \textit{Real and Logarithmic Enumerative Geometry}. The author wishes to thank Calla for teaching him the (conventional) generalised Kummer construction, and the workshop organisers for the invitation. 

The project began whilst the author was supported as a research fellow by Peterhouse, Cambridge, and completed while a guest supported by the Max Planck Institute for Mathematics in Bonn. For a long time the author was confused by the homeomorphism type of the tropical generalised Kummer variety associated to a tropical abelian variety. This confusion was resolved by an elegant proof suggested to him by his summer project student Yuvraj Singh. Yuvraj's entertaining proof is recorded as Proposition~\ref{prop:TypeIIHomeotype}.

The author thanks Jeremy Feusi for private correspondence, including communicating the definition of level structure on a tropical abelian variety. This correspondence was of further help in preparing Section \ref{sec:Tropical AV}. The author thanks Chikara Nakayama for private correspondence which aided the author's understanding of logarithmic Hodge theory, and thanks Nakayama, Dhruv Ranganathan, and Terry Song for comments on drafts. Thanks are due to Feusi and Aitor Iribar L\'opez for sharing a draft of their forthcoming work \cite{FeusiLopezTropAV}. The author similarly thanks Thibault Poiret and Yoav Len for sharing their work in progress \cite{poiret_len_logarithmic_tropical_prym} which was helpful whilst revising Section \ref{sec:log-abelian-variety}. The author thanks Dan Abramovich, Dori Bejleri, Robert Crumplin, Mark Gross, Daniel Huybrechts, Samuel Johnston, Davesh Maulik, Thibault Poiret, Dhruv Ranganathan, Qaasim Shafi, Terry Song, Pim Spelier, Michael Temkin, Martin Ulirsch, Ajith Urundolil Kumaran, and Jonathan Wise for interest, and many conversations from which the author learned a great deal. 

\subsection*{Notation} We consider all stacks to be over the spectrum of $\mathbb{C}$. All logarithmic structures are fine and saturated. For $X$ a (logarithmic) scheme and $G$ a group acting on $X$ we write the stack quotient $[X/G]$. We reserve the notation $\Delta_0$ for $\text{Spec}(\mathbb{C})$ equipped with rank one logarithmic structure, and $\Delta$ for a curve smooth over $\mathbb{C}$ with divisorial logarithmic structure from a point $\Delta_0\in \Delta$. 

\subsection*{Use of generative AI} Whilst preparing this manuscript, the author freely used generative AI to identify and understand references (especially where no English version was available). Some ideas in the proof of Proposition \ref{prop:TangentSpaceANDUniquePointwise}(2) were developed over an extended exchange with GPT-5.4 (OpenAI). The final displayed equality in case (4) in the proof of Lemma \ref{lem:tidyup} was computed by GPT-6 (OpenAI) before being verified, and the same model extended the $\alpha=2$ version of Figure~\ref{fig:hhz} to the $\alpha = 8$ case included in this manuscript. Generative AI was also used to produce tikz code for some diagrams and to identify references for statements the author knew to be true. References identified by AI were consulted before being cited. Both Opus 5 (Anthropic) and GPT-6 (OpenAI) were used to review the manuscript. During this review the proof of Proposition~\ref{prop:TwoFormExist}(4) was dramatically shortened and now uses a strategy suggested by GPT-6. The manuscript text was written by the author, although isolated clauses suggested by generative AI have been incorporated during the review process.

\section{Preliminaries: logarithmic geometry via Artin fans}
In this section we summarise the logarithmic geometry required later. Readers familiar with logarithmic geometry may safely skim up to Section \ref{sec:logspaces}. Starting in Section \ref{Sec:LogIHS} we make sense of logarithmic Hodge theory and logarithmic fundamental groups of logarithmic spaces.

Logarithmic geometry extends the toric dictionary to connect tropical geometry with the geometry of degenerations. Standard references for logarithmic geometry include \cite{Kato1989,OgusBook2018, temkin2022introductionlogarithmicgeometry}, but we follow the tropical and stack-theoretic perspective developed in \cite{AWbirational,ModStckTropCurve,OlssonLog}. 

\subsection{Tropical functors of points} 
The \textit{toric dictionary}\cite{Fulton+2016} gives an equivalence between two categories. 
First a category arising in algebraic geometry: \textit{affine toric varieties} with torus-equivariant morphisms. 
Second, the combinatorial/tropical category $\mathsf{RPC}$ of \textit{torsion-free, fine, saturated, and sharp} monoids.

A monoid $P$ is said to be torsion-free, fine, and saturated if it can be realised as the intersection of finitely many rational half-spaces within a finitely generated free abelian group. It is sharp if $P$ contains no non-zero units. In particular, $\mathsf{RPC}$ coincides with the category of strictly convex rational polyhedral cones.

\subsubsection{The face topology} The Zariski topology on toric varieties restricts to define the \textit{face} topology on $\mathsf{RPC}$. We denote the resulting site \textbf{RPC}. Open morphisms in \textbf{RPC} are exactly the inclusions of faces of a rational polyhedral cone, and a collection of opens forms a cover if it surjects set-theoretically. Every presheaf on the site \textbf{RPC} is a sheaf \cite[Proof of Proposition 2.3]{ModStckTropCurve}.

We write $\textbf{RPC}_\sigma$ for the comma site over a cone $\sigma$. Thus objects are cones equipped with a morphism to $\sigma$, and morphisms are morphisms over $\sigma$ (i.e. commuting triangles).

Just as objects of algebraic geometry can be understood through their functor of points on the category of affine schemes, objects of tropical geometry can often be understood as functors on \textbf{RPC}. We use this language to set up moduli problems within tropical geometry.

The category \textsf{RPCC} of rational polyhedral cone complexes is the full subcategory of sheaves on \textbf{RPC} whose objects can be expressed as a colimit over a finite poset of rational polyhedral cones in which all morphisms are face inclusions, and
such that the intersection of any two cones is a single common face.

\subsubsection{Subdivisions and alterations} 
 We identify a rational polyhedral cone $\sigma$ with its functor of points, considered as a sheaf on \textbf{RPC}. More generally, a fan in the sense of toric geometry gives rise to a sheaf on \textbf{RPC}.\footnote{We warn the reader that this assignment is not injective: two fans give rise to the same sheaf on \textbf{RPC} if and only if they have the same underlying rational polyhedral cone complex.} 

\begin{definition}
    The \textit{underlying topological space} $|\sigma|$ of a cone $\sigma\subset \mathbb{Z}^n$ is the subset $\mathbb{R}_{\geq 0}\otimes \sigma$ of $\mathbb{R}^n$ equipped with the subspace topology from $\mathbb{R}^n$. The underlying topological space of a cone complex $\Sigma$, where we express $\Sigma$ as the colimit over a diagram $$J\rightarrow \mathsf{RPCC}$$ of rational polyhedral cones, is the colimit over the induced diagram in the category of topological spaces. 
\end{definition}

\begin{definition}
    A \textit{subdivision} of a rational polyhedral cone $\sigma$ is a morphism of sheaves which arises from a subdivision of $\sigma$ in the category of fans \cite{Fulton+2016}. A subdivision of a sheaf $\mathpzc{F}$ on \textbf{RPC} is a morphism of sheaves $$\mathcal{F}_\Sigma \rightarrow \mathpzc{F}$$ represented by subdivisions of rational polyhedral cones. A \textit{tropical alteration} of a rational polyhedral cone $\sigma$ is a monomorphism $\Sigma \rightarrow \sigma$ in the category RPCC which induces an isomorphism on underlying topological spaces. A morphism of sheaves on \textbf{RPC} is said to be a tropical alteration if it is represented by tropical alterations of rational polyhedral cones. A tropical model of a sheaf $\mathpzc{F}$ is a tropical alteration with domain a cone complex. The \textit{underlying topological space} $|\mathpzc{F}|$ of a sheaf $\mathpzc{F}$ of sets on \textbf{RPC} which admits a tropical model is the underlying topological space of any tropical model.
\end{definition}

Note that a morphism $f$ of cone complexes induces a morphism of topological realisations which we denote $|f|$. In our next definition we set $\rho$ to be the monoid $\mathbb{N}$, and write $1\in \rho \otimes \mathbb{R}_{\geq 0}$ for the image of $1$ in $\mathbb{N}\subset \mathbb{R}_{\geq 0}$.
See \cite[Corollary~2.7]{battistella2026vectorbundlesolssonfans} for a tropical characterisation of exact morphisms of (fine and saturated) monoids.

\begin{definition}\label{defn:Link}
    Let $\pi \colon \Lambda \rightarrow \rho$ be an exact morphism of cone complexes, inducing a morphism of topological realisations $$\pi_\mathbb{R}\colon |\Lambda| \rightarrow |\rho|.$$ The \textit{link} of $\pi$ is the topological space $\pi_\mathbb{R}^{-1}(1)$. Let $\mathpzc{F}$ be a sheaf on $\textbf{RPC}_\rho$ which admits a tropical model $\Lambda\rightarrow \mathpzc{F}$ exact over $\rho$. The \textit{link of $\mathpzc{F}$} is the link of any such tropical model.\footnote{This is well defined because any two tropical models admit a common alteration (their fibre product over $\mathpzc{F}$), and tropical alterations induce identifications of topological realisations.}
\end{definition}

The link, which is a topological space and no extra data (such as an integral affine structure), is invariant under tropical alterations of $\Lambda$. A \textit{point} of a rational polyhedral cone $\sigma$ is a morphism of rational polyhedral cones $\rho \rightarrow \sigma$ where $\rho$ is the monoid $\mathbb{N}$. Thus for any morphism of cone complexes $\Sigma \rightarrow \sigma$ Definition~\ref{defn:Link} associates a topological space to every point $\rho \rightarrow \sigma$ of $\sigma$.

\subsection{Stacks with logarithmic structure}
We now recall the definition of logarithmic scheme, and more generally algebraic stack with logarithmic structure. Logarithmic schemes offer a category in which we can describe morphisms from objects of algebraic geometry to objects built from tropical geometry. In particular we can discuss maps from a degeneration to its associated integral affine structure.

\subsubsection{Logarithmic schemes}\label{sec:logschemesandtrop} An algebraic stack with logarithmic structure is an algebraic stack $\underline{X}$, equipped with a sheaf of monoids $M_X$ and a morphism of sheaves (of multiplicative monoids) $\alpha\colon M_X\rightarrow \mathcal{O}_X$ such that $$\alpha\colon \alpha^{-1}(\mathcal{O}_X^\times) \rightarrow \mathcal{O}_X^\times$$ is an isomorphism. A morphism $X\rightarrow Y$ of stacks with logarithmic structure is said to be \textit{strict} if the logarithmic structure of $Y$ pulls back to the logarithmic structure of $X$. See \cite[Section~1.4]{Kato1989} for pullbacks of logarithmic structures. 

The connection between logarithmic and tropical geometry is via toric varieties, which carry a canonical logarithmic structure. The canonical logarithmic structure upon a toric variety $X$ is defined by declaring $M_X$ to be the subsheaf of $\mathcal{O}_X$ consisting of functions whose zero set is supported on the toric boundary.

All the algebraic stacks with logarithmic structure, often denoted $X$, in this paper are \textit{fine and saturated}, which means that $X$ admits a strict \'etale cover, and for each $U_i\rightarrow X$ in this cover there exists a toric variety $Y_i$ and a strict morphism $U_i\rightarrow Y_i$. In particular, the stalk of $M_X/\mathcal{O}_X^\times$ is torsion free, sharp, fine, and saturated. 

An algebraic stack with logarithmic structure is said to be a logarithmic scheme if the underlying stack $\underline{X}$ is a scheme. A logarithmic structure is said to be \textit{locally free} if for each scheme theoretic point $p$ there is an integer $k$ such that the stalk of $M_X/\mathcal{O}_X^\times$ at $p$ is isomorphic to $\mathbb{N}^k$. 

The site obtained by equipping the category of (fine and saturated) logarithmic schemes with the strict \'etale topology is denoted \textbf{LogSch}. 
\subsubsection{The $a^\ast$ functor - from tropical to logarithmic geometry.}
Our next goal is to recall the $a^\ast$ functor defined in \cite{ModStckTropCurve}. This functor sends a stack defined on \textbf{RPC} to a stack defined on \textbf{LogSch} - thus objects of tropical geometry are promoted to objects of logarithmic geometry.

\begin{definition}[{\cite[Section 2.2]{AWbirational}}]
    An \textit{Artin cone} is a stack with logarithmic structure obtained by taking the global quotient stack of an affine toric variety with toric logarithmic structure by its dense torus. An alteration of an Artin cone $a^\ast \sigma$ is a morphism $a^\ast \Sigma \rightarrow a^\ast \sigma$ induced by a tropical alteration $\Sigma\rightarrow \sigma$ of $\sigma$.
\end{definition} 

\begin{definition}[{\cite[Definition~2.2.4]{AWbirational}}]
    A Noetherian logarithmic scheme is said to be atomic if, when stratified by the isomorphism type of the stalk $(M_X/\mathcal{O}_X^\times)_p$, it has a unique stratum which is closed and connected, and the restriction of the sheaf $M_X/\mathcal{O}_X^\times$ to this stratum is the constant sheaf.
\end{definition}

For $X$ an atomic logarithmic scheme, the \textit{Artin cone of} $X$ is the stack quotient of the toric variety associated to the dual of the sharp monoid $M_X/\mathcal{O}_X^\times$ by its dense torus.
\'Etale locally, every Noetherian logarithmic scheme is globally charted \cite[Lemma 2.2.5]{AWbirational}. 

For \textbf{C} a site we write $\mathsf{CFG}(\textbf{C})$ for the category of small categories fibred in groupoids over $\textbf{C}$. Following \cite[Section 6]{ModStckTropCurve}, one can define a functor
$$a^\ast \colon \mathsf{CFG}(\textbf{RPC})\rightarrow \mathsf{CFG}(\textbf{LogSch}).
$$
This functor is the bridge connecting logarithmic and tropical geometry. In particular, given a sheaf (resp. stack) $\mathpzc{F}$ on \textbf{RPC}, the restriction of $a^\ast \mathpzc{F}$ to the category of Noetherian schemes with logarithmic structure is the strict \'etale sheafification (resp. stackification) of the presheaf (resp. prestack) sending a logarithmic scheme $X$ to $\mathpzc{F}(\sigma)$ whenever $X$ is atomic and has Artin cone $a^\ast \sigma$, and the empty set otherwise.

\begin{remark} Through its functor of points, a cone complex \cite[Section 2.6]{abramovich2015skeletons}, cone stack \cite[Section 2]{ModStckTropCurve}, tropical abelian variety in the sense of Definition~\ref{defn:TropicalAV}, or piecewise linear space \cite[Section 1]{LogQuot} can be upgraded to a stack on the category of logarithmic schemes. Only the first two examples give rise to sheaves which are Artin stacks equipped with logarithmic structure.
\end{remark}

\subsubsection{The Artin fan construction}
Following \cite[Section 2.3]{AWbirational}, a stack $a^\ast \mathpzc{F}$ on the category of logarithmic schemes is said to be an Artin fan if it admits a strict, representable, and \'etale cover by Artin cones $$\left\{a^\ast \sigma \rightarrow a^\ast \mathpzc{F}\right\}.$$ 
For $X$ in a broad class of schemes\footnote{Recall the data of a scheme includes a morphism to $\operatorname{Spec}(\mathbb{C})$ in this paper.} with logarithmic structure, including the case $X$ is locally Noetherian, there is a canonical way to associate an Artin fan to $X$; thus the Artin fan construction associates to each Noetherian logarithmic scheme an object of tropical geometry.
%strict morphism $$X\rightarrow \mathcal{A}_X$$ which is initial among morphisms to Artin fans which have faithful monodromy over Olsson's stack of logarithmic structures \cite[Section 2]{AWbirational}.
%
%\begin{example}\label{example:NonUniqueAfan}
%    Let $X$ be the union of two copies of $\mathsf{Spec}(\mathbb{C})$. Let $f_1\colon X\rightarrow [\mathbb{A}^1/\mathbb{G}_m]$ be the map sending both points to $0$, and let $f_2\colon X\rightarrow [\mathbb{A}^1/\mathbb{G}_m]\cup [\mathbb{A}^1/\mathbb{G}_m]$ which surjects to closed points. The logarithmic schemes induced by $f_1$ and $f_2$ are identical.
%\end{example}
%Example \ref{example:NonUniqueAfan} shows that many maps give rise to the same logarithmic scheme. 
\begin{proposition}[{\cite[Theorem~5.7]{abramovich2015skeletons},\cite{AWbirational}}]
    For $X$ a Noetherian logarithmic scheme (over the complex numbers) there is a cone stack $\mathsf{Trop}(X)$ (i.e. a stack on \textbf{RPC} satisfying representability properties described in \cite{ModStckTropCurve}) and a strict morphism
    $${X}\rightarrow \mathcal{A}_X := a^\ast \mathsf{Trop}(X)$$ which is initial among strict maps from $X$ to an Artin fan with faithful monodromy. We call $\mathcal{A}_X$ \textit{the} Artin fan of $X$. Consider any morphism of Noetherian logarithmic schemes $f\colon X\rightarrow Y$. There is an initial Artin fan $\mathcal{A}_{X/Y}$ which fits into a commutative square with strict vertical arrows of the following shape:
    $$\begin{tikzcd}
X \arrow[r] \arrow[d] & Y \arrow[d] \\
\mathcal{A}_{X/Y} \arrow[r] & \mathcal{A}_Y
\end{tikzcd}.$$ 
\end{proposition}

We call $\mathcal{A}_{X/Y}$ the relative Artin fan of $X$ over $Y$.
In general $\mathcal{A}_X\ne \mathcal{A}_{X/Y}$, and thus the Artin fan construction is not always functorial. See \cite[Section 2.10]{abramovich2015skeletons} for an example.

\begin{example}
    An effective Cartier divisor on a smooth scheme $\underline{X}$ specifies a map from $\underline{X}$ to $[\mathbb{A}^1/\mathbb{G}_m]$. Thus, a strict normal crossing pair $(X,D=\sum_iD_i)$ such that every intersect of $D_i$ is connected, together with an ordering on the $D_i$ is the same data as a morphism $$\varpi\colon X\rightarrow [\mathbb{A}^n/\mathbb{G}_m^n] = a^\ast \Sigma.$$ Note $\mathsf{Trop}(X)\subset \Sigma$ is a sub-cone complex; potentially a proper sub-cone complex as $\varpi$ need not hit the closed point of $a^\ast \Sigma$.
\end{example}
\begin{example}
    Consider a strict normal crossing degeneration: a map $\pi\colon X \rightarrow \mathbb{A}^1$ such that the central fibre is a strict normal crossing divisor $\pi^{-1}(0) = \sum_iD_i$. Since the Artin fan of $\mathbb{A}^1$ is a cone, \cite[Lemma A.21]{herr2026loggeometryliftingrational} ensures that the Artin fan construction is functorial in this case. We thus obtain a commutative square,
    $$
        % https://tikzcd.yichuanshen.de/#N4Igdg9gJgpgziAXAbVABwnAlgFyxMJZABgBpiBdUkANwEMAbAVxiRAA0QBfU9TXfIRQBGclVqMWbADrSAtnRwALAEYrgAQS4A9Yd14gM2PASKjh4+s1aIQdbbLg46AJwAEsl0oj6+xwURkFtRWUrb2js7usgrKcABmwAAqLhBoXAAU7ACU3OIwUADm8ESg8alySGQgOD6IAMw8ZRVI9dS1SABMTSDlEJWI1R2Iwj19A6I1dd0UXEA
\begin{tikzcd}
X \arrow[d] \arrow[r]              & \mathbb{A}^1 \arrow[d] \\
a^\ast \mathsf{Trop}(X) \arrow[r] & a^\ast \rho = a^\ast \mathsf{Trop}(\mathbb{A}^1).          
\end{tikzcd}
    $$
    Assuming each logarithmic stratum is connected, the underlying topological space of the dual intersection complex of $\pi^{-1}(0)$ can be obtained from the map $\mathsf{Trop}(X)\rightarrow \mathbb{N}$ as the preimage of $1$ under the map $\mathsf{Trop}(X)\otimes\mathbb{R}_{\geq 0} \rightarrow \mathbb{R}_{\geq 0}$. Thus the dual intersection complex is a choice of polyhedral structure upon the link.
\end{example}

\begin{remark} 
    In the sequel we often consider logarithmic schemes over some atomic base $S$ with tropicalisation a cone $\kappa$ (i.e. with Artin fan $S\rightarrow a^\ast \kappa$). We are often led to study sheaves on the comma site $\textbf{RPC}_\kappa$ whose objects are morphisms of cone complexes $\sigma\rightarrow \kappa$.
\end{remark}

\subsection{Toroidal compactifications: from scheme discs to tropical rays}
Let $(X,\sum_iX_i)$ be a strict normal crossing pair which is proper over $\mathbb{C}$, and let $\Delta$ be a trait as in the notation section with generic point $\Delta^\circ$. Fix a morphism of (not logarithmic) schemes $f\colon \Delta^\circ\rightarrow X$ with image disjoint from the divisor $\sum_iX_i$. Write $\Sigma$ for the cone stack such that $a^\ast \Sigma$ is the Artin fan $\mathcal{A}_X$ of the logarithmic scheme obtained by equipping $X$ with divisorial logarithmic structure from the divisor $\sum_i X_i$.

Since $X$ is proper, on the level of underlying schemes $f$ extends to a unique morphism $\overline{f}$ from $\Delta$ to $X$. Since the preimage of the boundary of $X$ under $\overline{f}$ is contained within $\Delta_0$, the morphism $\overline{f}$ can be upgraded (in a unique way) to a morphism of logarithmic schemes from $\Delta$ to $X$, where $X$ and $\Delta$ are understood to be logarithmic schemes with the obvious divisorial logarithmic structures. Since the tropicalisation of $\Delta$ is a cone, \cite[Lemma A.21]{herr2026loggeometryliftingrational} shows that tropicalisation in this situation is functorial, and we obtain a morphism
$$a^\ast f^\text{trop}\colon \mathcal{A}_{\Delta}\cong a^\ast \mathbb{N}\rightarrow \mathcal{A}_X\cong a^\ast \Sigma.$$
\begin{definition}\label{defn:RayFromDVR}
    The \textit{tropicalisation of $f$} is the unique morphism $$f^\text{trop}\colon \mathbb{N}\rightarrow \Sigma$$ such that $a^\ast f^\text{trop}$ induces the map of the previous paragraph.
\end{definition}
Note $f^\text{trop}$ is unique since $a^\ast$ induces an equivalence between the category of cone stacks and the category of Artin fans, see \cite[Theorem~6.11]{ModStckTropCurve}.
\subsection{Local properties for morphisms of logarithmic schemes}
The following definitions record the standard translation of \cite[Definition~4.1]{OlssonLog} to the language of Artin fans. They are consistent with Kato's original definitions, as recorded in \cite{Kato1989}.
\begin{definition}\label{defn:logsmoothold}
    A morphism $f\colon X\rightarrow Y$ of Noetherian logarithmic schemes is said to be \textit{logarithmic \'etale} (resp. \textit{smooth, flat}) if the morphism $$X\rightarrow {\mathcal{A}_{X/Y}}\times_{\mathcal{A}_Y} Y$$ is \'etale (resp. smooth, flat) as a morphism of stacks.
\end{definition}

Key examples of logarithmic \'etale maps include logarithmic alterations. 

\begin{definition}[{\cite{MolchoWise,MW23}}]\label{defn:logAlt}
    A \textit{logarithmic alteration} of a Noetherian DM stack with logarithmic structure $X$ is a morphism of DM stacks with logarithmic structure $X'\rightarrow X$ which, strict \'etale locally upon $X$, is the base change of an alteration of Artin cones.
\end{definition}

\begin{definition}[{\cite{hu2025grothendiecktopologieslogarithmicmodifications,MolchoWise,MW23,nak17}}]
    The \textit{logarithmic \'etale site} is the site structure on the category of logarithmic schemes whose covers are generated by strict \'etale covers (inherited from conventional algebraic geometry) and logarithmic alterations. A morphism of sheaves on \textbf{LogSch} is called a logarithmic \'etale cover if it is represented by logarithmic \'etale covers on the category of logarithmic schemes.
\end{definition}

\subsection{Logarithmic spaces}\label{sec:logspaces} Neither logarithmic abelian varieties \cite{KajiwaraKatoNakayama2008} nor logarithmic Hilbert schemes of points (when the logarithmic structure has rank at least two, see \cite[Section 9.4]{LogQuot}) are represented by schemes with logarithmic structure. This motivates the definition of logarithmic space (in the second sense) \cite[Section 10]{MR3962244}.

\begin{definition}\label{defn:logalgspaces}
    \textit{Logarithmic spaces} are sheaves of sets on \textbf{LogSch} which admit a logarithmic \'etale cover by a scheme with logarithmic structure. A \textit{logarithmic DM stack} is a CFG over \textbf{LogSch} with diagonal representable by logarithmic spaces, and which admits a logarithmic \'etale cover by a scheme with logarithmic structure. 
\end{definition}

The object which \cite{MR3962244} refers to as a \textit{logarithmic space in the first sense} will be called an algebraic space with logarithmic structure in the sequel.

\begin{example}
    The stack quotient of $\mathbb{A}^1$ (equipped with its toric logarithmic structure) by $\mu_2$ acting via the dense torus of $\mathbb{A}^1$ is a logarithmic space, even though it is obtained by equipping a Deligne--Mumford stack with non-trivial stabiliser with a logarithmic structure. A logarithmic alteration of a logarithmic space $X$ is a morphism in the category of logarithmic spaces $X'\rightarrow X$ represented by logarithmic alterations on the category of logarithmic schemes.
\end{example}

\begin{remark}
    We expect that, in order to generalise the results of this paper to compactify other moduli of IHS varieties, it will be convenient to restrict attention to only those logarithmic spaces whose functor of points satisfies logarithmic \'etale descent. These structures appear as descending logarithmic spaces studied in \cite[Section 4]{logcoh}.
\end{remark}

The logarithmic Hilbert schemes we consider do not admit a logarithmic alteration where the domain is a scheme with logarithmic structure. We are thus led to make the following definition.

\begin{definition}
    A \textit{sheafy DM stack with logarithmic structure} is a DM stack with logarithmic structure having the property that, when considered as a CFG over the logarithmic \'etale site, it is a category fibred in setoids (i.e. stabilisers are trivial).
\end{definition}

All of the logarithmic spaces we consider will admit logarithmic alterations which are sheafy DM stacks with logarithmic structure.

\subsubsection{Tropicalisation maps} For every logarithmic DM stack $Y$ we construct, we will also construct a morphism represented by algebraic stacks with logarithmic structure $$\varpi_Y\colon Y\rightarrow a^\ast \mathpzc{F},$$ where $\mathpzc{F}$ will have an explicit description. In the special case that $Y$ is a Noetherian logarithmic scheme, the role of $a^\ast \mathpzc{F}$ is played by the Artin fan of $Y$: this perspective on tropicalisation generalises the Artin fan construction.

Let $$\Sigma\rightarrow \mathpzc{F}$$ be a tropical alteration, where $\Sigma$ is a cone complex. When the domain of the map $$Y_\Sigma = a^\ast \Sigma\times _{a^\ast \mathpzc{F}} Y\rightarrow Y$$
is an algebraic stack with logarithmic structure and $Y$ is Noetherian, this map is a logarithmic alteration, and is called a \textit{proper model} of $Y$. 

\section{Logarithmic irreducible holomorphic symplectic varieties}\label{Sec:LogIHS}
Recall that an irreducible holomorphic symplectic variety $X\rightarrow S$ over a scheme $S$ of finite type is a smooth projective $S$-scheme such that for any $\mathbb{C}$-valued point $s$ of $S$, the base change $X_s = X\times_S s$ has simply connected analytic space, and $H^{2,0}(X_s,\mathbb{C})$ is generated by a nowhere-degenerate two form. 
In this section, motivated by properties of logarithmic generalised Kummer varieties and logarithmic Hilbert schemes of points of K3 surfaces, we generalise the definition of IHS variety to the logarithmic setting. 
%See https://arxiv.org/pdf/2203.10391v2 opening of section 2.3 for a reference, although this seems hard to find.

%For $X$ a scheme with logarithmic structure the literature provides a logarithmic version of being simply connected, of logarithmic Hodge theory, and of smoothness. However, this paper must work with logarithmic spaces. Worse, the logarithmic spaces we study need not admit logarithmic alterations which are schemes with logarithmic structure. On the other hand, we will always have access to a logarithmic alteration which is a sheafy DM stack with logarithmic structure.  
\subsection{Smoothness} Even the definition of logarithmic smoothness is subtle for logarithmic spaces.

\begin{definition}\label{defn:loclogsmooth}
    We say that a morphism of logarithmic spaces $X\rightarrow S$ is \textit{locally logarithmically smooth} if, possibly after replacing $S$ by a strict \'etale cover, there exist proper models $S'\rightarrow S$ and $X'\rightarrow X\times_S S'$ such that the induced morphism $X'\rightarrow S'$ is logarithmically smooth in the sense of Definition~\ref{defn:logsmoothold}.
\end{definition}

\begin{remark}
    Any morphism of schemes with log structure which is logarithmically smooth is locally logarithmically smooth. The converse is not true, see \cite[Example IV.4.3.4]{OgusBook2018} (or \cite[Example 2.6]{hu2025grothendiecktopologieslogarithmicmodifications} for an exposition). Local logarithmic smoothness is a natural generalisation of logarithmic smoothness to morphisms of descending logarithmic spaces (motivated in Section \ref{sec:outlook}). 
\end{remark}
\begin{remark}
    The author is not aware of an example of a locally logarithmically smooth morphism between static schemes (in the sense of \cite[Definition 2.2.3]{logcoh}) with logarithmic structure where both source and target are static, and which is not logarithmically smooth.
\end{remark}

Recall the conventions for $\Delta$ and $\Delta_0$ from the notation section. In proofs we will often deduce information about the central fibre of a degeneration from its general fibre. Such arguments provide information about families which satisfy the following property. 
\begin{definition}
    We say that a (polarised) saturated, vertical, and logarithmically smooth morphism $Y\rightarrow \Delta_0$ is \textit{(polarised) smoothable} if, possibly after replacing $Y$ and $\Delta_0$ by logarithmic alterations, it coincides (as a polarised logarithmic scheme) with the central fibre to a projective, vertical, logarithmically smooth, and saturated morphism $X\rightarrow \Delta$.
\end{definition}

\subsection{Homotopy groups of logarithmic spaces} 
Throughout this section all schemes (resp. DM stacks) will be of finite type over $\mathbb{C}$. Associated to such a scheme (resp. DM stack) $Y$ is a \textit{logarithmic analytic space} $Y^{an}$ (resp. logarithmic analytic orbifold). See \cite[Paragraph (1.1)]{KatoNakayama1999} for the case that $Y$ is a scheme with logarithmic structure. For $Y$ a DM stack with logarithmic structure we define $Y^{{an}}$ as a stack on the category of analytic spaces with logarithmic structure as follows. Choose a strict \'etale open cover $U\rightarrow Y$, form the groupoid $$U\times_YU\rightrightarrows U,$$ and declare $Y^{{an}}$ to be the quotient $[(U\times_{Y}U)^{an}\rightrightarrows U^{an}]$ as an analytic stack.

As a scheme degenerates, the topology of its associated analytic space changes, but topological behaviour of the \textit{Kato--Nakayama space} of the associated logarithmic analytic space is preserved in logarithmically smooth families \cite[Theorem 5.1]{NakayamaOgus2010RelativeRounding}. In this section we will make sense of the homotopy type of the Kato--Nakayama space of a logarithmic space of finite type over $\mathbb{C}$.

\begin{definition}[{\cite{KatoNakayama1999}}]\label{defn:KNSpace}
    Let $Y$ be a scheme with logarithmic structure of finite type over $\mathbb{C}$, and write $Y^{an}$ for the associated logarithmic analytic space. The \textit{topological Kato--Nakayama} space $|\operatorname{KN}(Y)|$ is obtained by equipping the set
$$
|\operatorname{KN}(Y)| := \bigl\{\,(x,h) \;\big|\; x \in Y^{an},\ h \colon M_{Y,x}^{\mathrm{gp}} \to S^{1} \text{ group hom with } h(u) = u(x)/|u(x)| \text{ for all } u \in \mathcal{O}_{Y,x}^{*} \bigr\},
$$
with the coarsest topology satisfying the following hypotheses.
\begin{enumerate}[leftmargin =*]
    \item The projection $\tau \colon |\operatorname{KN}(Y)| \to Y^{an}$, $(x,h) \mapsto x$, is continuous.
    \item For every open $U \subset Y^{an}$ and every section $m \in \Gamma(U, M_Y^{\mathrm{gp}})$, the evaluation map
    $$
    \tau^{-1}(U) \longrightarrow S^{1}, \qquad (x,h) \longmapsto h(m_x),
    $$
    is continuous.
\end{enumerate}
\end{definition}

For $Y$ a scheme with logarithmic structure of finite type over $\mathbb{C}$, we write $Y_{\text{k\'et}}^{an}$ for the Kummer logarithmic \'etale site of the analytic space $Y^{an}$, and refer to \cite[2.5]{Illusie2002} for the definition.

\begin{definition}
    The (kummer log \'etale) Kato--Nakayama space $\mathsf{KN}(Y)$ associated to $Y$ is a ringed space obtained by equipping $|\operatorname{KN}(Y)|$ with the sheaf of rings $\mathcal{O}_{Y_{\text{k\'et}}}^\text{KN}$ described in \cite[3.2]{IllusieKatoNakayama2005}.
    Following \cite[3.1]{IllusieKatoNakayama2005}, there is a morphism of topoi,
    $$\tau \colon \operatorname{KN}(Y)\rightarrow Y^{an}_{\text{k\'et}}.$$
\end{definition}

Recall from \cite[Paragraph 1.2.5]{KatoNakayama1999} that a morphism $f\colon X\rightarrow Y$ of logarithmic schemes of finite type over $\mathbb{C}$ induces a morphism $f^\text{KN}$ of Kato--Nakayama spaces. 

\begin{remark}\label{rem:Kummeret}
    The original definition of the Kato--Nakayama space \cite[Section 1]{KatoNakayama1999}, which is consistent with the use throughout \cite{KatoUsui2009}, is equipped with a different (not kummer \'etale) sheaf of rings. The underlying topological space is the same in both versions, and they differ only in their sheaf of functions. In this paper all Kato--Nakayama spaces are the Kummer logarithmic \'etale version.
\end{remark}

\begin{remark}
    Even the topological Kato--Nakayama space of $Y$ is not invariant under non-kummer logarithmic alterations of $Y$. Instead Proposition~\ref{prop:HomotopyGroups} will show that the homotopy type of the Kato--Nakayama space is well defined.
\end{remark}

Since the logarithmic generalised Kummer varieties we consider do not admit logarithmic alterations which are schemes with logarithmic structure, we will extend Definition~\ref{defn:KNSpace} to define the Kato--Nakayama space of a sheafy DM stack with logarithmic structure. Following \cite[Section 4]{noohi2005foundationstopologicalstacksi}, it will be convenient to discuss topological stacks, i.e. stacks on the site \textbf{Top} whose underlying category consists of compactly generated topological spaces and whose open covers are open covers of topological spaces.

\begin{lemdefn}\label{lemdefn:KNSpace}
    Let $X$ be a separated sheafy DM stack with locally free logarithmic structure of finite type over $\mathbb{C}$, and choose a groupoid presentation
    $$
        s,t\colon R = U\times_X U \rightrightarrows U
    $$
    for $X$, where $U$ and $R$ are schemes with logarithmic structure, and $s,t$ are strict \'etale.\footnote{Note this is the same data as a strict \'etale cover $U\rightarrow X$ by schemes with logarithmic structure, which exists by hypothesis.}
    \begin{enumerate}[leftmargin = *]
        \item The induced morphism in the category of topological spaces $$(s^{\operatorname{KN}},t^{\operatorname{KN}})\colon |\operatorname{KN}(R)|\rightarrow |\operatorname{KN}(U)| \times |{\operatorname{KN}}(U)|$$ is an embedding. Moreover the natural map $q\colon \operatorname{KN}(U) \rightarrow \operatorname{KN}(U)/\operatorname{KN}(R)$ in the category of topological spaces admits local sections. Thus by \cite[Proposition 7.9]{noohi2005foundationstopologicalstacksi} the topological stack quotient $[|\operatorname{KN}(U)|/|\operatorname{KN}(R)|]$ is representable by a topological space which we denote $|{\operatorname{KN}}(X)|$. 
        \item The homeomorphism type of $|{\operatorname{KN}}(X)|$ does not depend on the choice of groupoid presentation.
    \end{enumerate}
    The \textit{Kato--Nakayama} space of $X$, denoted ${\operatorname{KN}}(X)$, is the ringed space obtained by equipping $|{\operatorname{KN}}(X)|$ with the sheaf of rings descended from $\mathcal{O}_{U_{k\acute{e}t}}^{\operatorname{KN}}$.
    \begin{enumerate}[leftmargin = *, start = 3]
        \item The isomorphism type of ${\operatorname{KN}}(X)$ is independent of the choice of groupoid presentation, and there is a canonical morphism of ringed spaces $$\tau\colon \operatorname{KN}(X)\rightarrow X^{an}.$$
    \end{enumerate}
\end{lemdefn}

\begin{remark}
    In the logarithmic category, sheafy DM stacks with logarithmic structure behave like non-stacky objects. Lemma/Definition~\ref{lemdefn:KNSpace} captures this fact because their Kato--Nakayama space is naturally a topological space not a topological stack. We will discover that extending the analysis of Kato--Nakayama spaces to sheafy DM stacks with logarithmic structure consequently does not introduce substantial subtlety.
\end{remark}

\begin{proof}[Proof of Lemma \ref{lemdefn:KNSpace}]
    %This paragraph collects observations required for this proof. First there is an isomorphism of topological spaces $$|\operatorname{KN}(U)|\times |\operatorname{KN}(U)|\cong |\operatorname{KN}(U)\times \operatorname{KN}(U)|$$ because the Kato--Nakayama functor preserves fibre products over integral maps \cite[Proposition~3.1]{cailotto2004katonakayama}. 
    Since $X$ is a sheafy DM stack with logarithmic structure, its functor of points on the category of logarithmic schemes is set-valued. Equivalently the diagonal morphism $(s,t) \colon R \to U \times U$ is a monomorphism. The equivalence between injectivity of relations into $U\times U$ and trivial stack structure hold for presentations of Kato--Nakayama spaces in our statement. We will use this equivalence repeatedly. 
    
    We first prove that the displayed map in (1) is injective. Suppose for contradiction that there exist two points $r_1,r_2\in |\operatorname{KN}(R)|$ which are sent to the same point $p$ of $|\operatorname{KN}(U)|\times |\operatorname{KN}(U)|$. The image of $\tau(p)\in U\times U$ in the underlying Deligne--Mumford stack $\underline{X}$ of $X$ is a closed point $q$. The gerbe structure of $\underline{X}$ at $q$ is a finite group $G_q$. According to paragraph 1 we will reach a contradiction if we show that $G_q$ acts freely on points of the Kato--Nakayama space lying over $q$.

    Since the functor of points of $X$ on the category of logarithmic schemes is setoid valued, the group $G_q$ acts freely upon standard logarithmic points lying over $q$. Note that $G_q$ acts on the pullback of $M_X$ to the closed subscheme $q$. The monoid of global sections splits $$\varphi \colon M_{q}\xrightarrow{\cong} \overline{M}_q\oplus \mathbb{C}^\times$$ because $q$ is a $\mathbb{C}$ (which is algebraically closed) valued point, and thus admits a global chart. The splitting is not canonical. Given a choice of splitting and $h\in M_q$, we write $|h|$ for the modulus of the image of $h$ under projection in $\mathbb{C}^\times$. Choosing $f\in G_q$, we will say the splitting $\varphi$ is $f$-modulus split if $|f(h)| = |h|$ for all $h$ in $M_q$. Note that in any case $f$ necessarily fixes $\mathbb{C}^\times$ because it comes from a $\mathbb{C}$ morphism of schemes. 

We claim that, after fixing a choice of $f$, it is possible to choose a splitting with respect to which $M_q$ is $f$-modulus split. Indeed, since $f$ is an automorphism and thus acts as an automorphism of $\overline{M}_q\oplus \mathbb{C}^\times/\mathbb{C}^\times$, the image of any ray in $\overline{M}_q\oplus \{1\}$ is an element of the form $(\rho,\gamma)\in \overline{M}_q\oplus\mathbb{C}^\times $ where $\rho$ is a ray. Pick an ordering on the rays $\{\rho_1,...,\rho_k\}$ of $\overline{M}_q\oplus \{1\}$ such that, up to adding an element of $\mathbb{C}^\times$, we find $f(\rho_j)$ in the set $\{\rho_1,...,\rho_{j+1}\}$. This ensures (and can be achieved by ensuring) that either $f(\rho_j,1) =(\rho_{j+1},\gamma_{j+1})$ for some $\gamma_{j+1}$ in $\mathbb{C}^\times$, or $f(\rho_j)$ is (up to the $\mathbb{C}^\times$ factor) already in the set $\{\rho_1,...,\rho_{j}\}$. We further assume that $|f(\rho_s)|=1$ for $s< r$ for some integer $r> 0$. We will inductively extend this situation to replace $r$ by $r+1$. 
    
%    First note it is harmless to ensure that $|\rho_r|=1$, since the log structure is free so $\rho_r$ is not in the $\mathbb{Q}$ span of $\{\rho_i\}_{i<r}$, and thus we can freely alter our splitting to rescale $\rho_r$. 
To establish the inductive step and ensure $|f(\rho_r)|=1$ we consider cases. 

    \begin{enumerate}[leftmargin = *]
        \item If under our current splitting $f(\rho_r) = (v,\gamma)$ where $v\in \overline{M}_q$ lies in the $\mathbb{Q}$-span of $\rho_1,...,\rho_{r}$, then since also $f$ permutes rays we know that $v=\rho_j$ for some $j\leq r$. Because $G_q$ is a finite group there exists a minimal integer $k$ such that $$f^k(\rho_r) = f^{k-1}(\rho_j)+\gamma= \rho_r+ \alpha$$ for some $\alpha$ in $\mathbb{C}^\times$. By the choice of ray ordering we fixed, we know that $j\leq r$. By the inductive hypothesis we have $|f^{k-1}(\rho_j)|=1$ and so $|\alpha| = |f(\rho_r)| = |\gamma|$. Since $f$ has finite order, there is some integer $w$ such that $f^{kw} = \operatorname{id}$ and so $|f^{kw}(\rho_r)| = |\rho_r| = 1$. It follows that $|\alpha|^w=1$, but $\mathbb{R}_{> 0}$ is torsion-free and so $|\alpha|=|\gamma| = 1$.\footnote{This argument does not show that we can achieve a splitting such that the argument of $f(\rho_j)=1$. This is because the multiplicative group of $\mathbb{C}$ has torsion.}
\item The only other case is that $v$ does not lie in the $\mathbb{Q}$-span of $\rho_1,...,\rho_r$. In this case we are free to alter the splitting through the automorphism characterised by sending $\rho_j\in \overline{M}_q\oplus \mathbb{C}^\times $ to $\rho_j$, unless $\rho_j=v$ where we send $v \mapsto v/\gamma$.
%        \item The remaining situation is that $f(\rho_r)$ does not lie in the $\mathbb{Q}$-span of $\rho_1,...,\rho_{r-1}$ and $\mathbb{C}^\times$. Indeed, if it lies in the $\mathbb{Q}$ span then coefficients of each $\rho_i$ can be taken to lie in $\mathbb{Q}_{\geq 0}$ by sharpness of $\overline{M}_q$. The $\mathbb{Q}$ independence assumption allows us to change splitting so that $|f(\rho_r)|=1$ without rescaling $\rho_i$ for any $i\leq r$ or $f(\rho_i)$ for any $i< r$.
    \end{enumerate}
    
    The existence of distinct $r_1,r_2$ (fixed for contradiction in second paragraph of this proof) is equivalent to the existence of a morphism $k\colon M_q\rightarrow S^1$ extending the argument function on $\mathbb{C}^\times$, and an element $f$ of $G_q$ which fixes $k$ under precomposition. We will use these data to construct a standard log point valued point of $X$ which is invariant under $f$, thus reaching a contradiction.
    
    Choose an $f$-modulus splitting $M_q\cong \overline{M}_q\oplus \mathbb{C}^\times$ and label rays $\rho_1,...\rho_k$ of $\overline{M}_q\oplus \{1\}$. Consider a morphism $$M_q\cong \overline{M}_q\oplus \mathbb{C}^\times \rightarrow \mathbb{N}\oplus \mathbb{C}^\times$$
$$\left(\sum_i a_i \rho_i,\lambda\right) \mapsto \left(\sum_i a_i,|\lambda|k\left(\left(\sum_i a_i\rho_i,\lambda \right)\right)\right)$$
    This map is invariant under $f$ because $f$ induces a bijection from the set of rays to the set of rays, and thus the map to $\mathbb{N}$ is $f$ invariant; the argument part of the map to $\mathbb{C}^\times$ is $f$ invariant by hypothesis; the modulus part is invariant because we chose an $f$-modulus split splitting.
    We have thus specified a standard log point valued point of $X$ which is fixed by $f$. Contradiction. The displayed map in (1) is injective.

    The displayed map in (1) is proper because $|X|$ is separated. Any universally closed injection into a Hausdorff topological space is an immersion, so our map is an immersion. 
    
    We will show that the map $q\colon |\operatorname{KN}(U)|\rightarrow |\operatorname{KN}(U)|/|\operatorname{KN}(R)|$ is a local homeomorphism. The idea is to find for each point $p$ in $\operatorname{KN}(U)$ an open set $W$ containing $p$ which is mapped homeomorphically to its image - i.e. no two points distinct points in $W$ have a relation between them. Indeed, first note that the image $e$ of $U$ in $R$ is open because $U\rightarrow X$ is unramified \stacks{02GE}. Moreover, since $\operatorname{KN}(R)$ embeds in $\operatorname{KN}(U)\times \operatorname{KN}(U)$ it has the subspace topology and so there exists an open subset $T$ of $\operatorname{KN}(U)\times \operatorname{KN}(U)$ whose intersect with $\operatorname{KN}(R)$ is $e$. We can choose $W$ to be an open of $|\operatorname{KN}(U)|$ such that $W\times W$ is a subset of $T$. We claim that the restriction of $q$ to $W$ is a homeomorphism. Injectivity is immediate because we just arranged that any relation with both ends in $W$ is the identity. The map is continuous by definition of the quotient topology on $|\operatorname{KN}(U)|/|\operatorname{KN}(R)|$. 
    It remains to check that $q$ is an open map, but for $V\subset |\operatorname{KN}(U)|$ open $q^{-1}(q(V)) = t(s^{-1}(V))$, and $t$ is open because it comes from a strict \'etale map.
    
    For claim (2), it suffices to show that replacing a presentation by a refinement induces a Morita equivalent groupoid presentation, since any two presentations admit a common refinement. Concretely, let $U' \rightarrow U$ be a refined cover, and write
$$
R = U\times_X U,
\qquad
R' = U'\times_X U'.
$$
This induces a morphism of groupoid presentations
$$
\begin{tikzcd}
R' \arrow[r] \arrow[d,shift left=.5ex] \arrow[d,shift right=.5ex] & R \arrow[d,shift left=.5ex] \arrow[d,shift right=.5ex] \\
U' \arrow[r] & U.
\end{tikzcd}
$$

To check that the induced morphism
$$\left(
|\operatorname{KN}(R')| \rightrightarrows |\operatorname{KN}(U')|\right)
\longrightarrow\left(
|\operatorname{KN}(R)| \rightrightarrows |\operatorname{KN}(U)|\right)
$$
is a Morita equivalence we first verify full faithfulness: we have
$
R'
\cong
U'\times_U R\times_U U',
$
so preservation of fibre products by the Kato--Nakayama functor over strict morphisms \cite[Proposition~3.1]{cailotto2004katonakayama} gives
$$
|\operatorname{KN}(R')|
\cong
|\operatorname{KN}(U')|
\times_{|\operatorname{KN}(U)|}
|\operatorname{KN}(R)|
\times_{|\operatorname{KN}(U)|}
|\operatorname{KN}(U')|,
$$
but this is what it means to be fully faithful.

We now argue that the induced morphism is essentially surjective. Indeed the map
$
|\operatorname{KN}(U')|\to |\operatorname{KN}(U)|
$
is surjective and open because $U'\to U$ is surjective and strict étale. 
%(see \cite[Section 4.3]{MR4286366} openness of a map obtained by applying the Kato--Nakayama functor to a strict \'etale map is standard, see standard fact, see for example \cite[Section 4.3]{MR4286366}). 
Surjectivity ensures that every point $u\in |\operatorname{KN}(U)|$ lifts to $u'\in |\operatorname{KN}(U')|$ and then $(u',\mathrm{id}_u)$ is a point of
$
|\operatorname{KN}(U')|
\times_{|\operatorname{KN}(U)|}
|\operatorname{KN}(R)|
$
with image $u$.
Claim (3) is descent of sheaves and morphisms along a strict \'etale open cover. 
\end{proof}

Remark \ref{rem:NeedRelative} might aid the reader in parsing our next definition.
    
\begin{definition}\label{defn:logFG}
 Let $X$ be a separated complex logarithmic space which admits a choice $q\colon X'\rightarrow X$ of logarithmic alteration by a sheafy DM stack with locally free logarithmic structure which is of finite type over $\mathbb{C}$. The $k^\mathsf{th}$ homotopy group of $X$ is the $k^{\mathsf{th}}$ homotopy group of the Kato--Nakayama space $\mathsf{KN}(X')$.
Let $S$ be the standard logarithmic point.
The relative homotopy groups of a morphism of logarithmic spaces $X\rightarrow S$ are defined as the kernel $$\pi_k(\mathsf{KN}_0(X)) := \mathsf{Ker}(\pi_k(\mathsf{KN}(X))\rightarrow \pi_k(\mathsf{KN}(S))).$$ 
\end{definition}

The profinite completion of the logarithmic fundamental group of $X$ is studied, through an alternative perspective, in \cite{CarchediScherotzkeSibillaTalpo2025}. We now check that the above definition did not depend on choices.

\begin{proposition}\label{prop:HomotopyGroups}
    The logarithmic homotopy groups in Definition~\ref{defn:logFG} of a logarithmic space $X$ of finite type over $\mathbb{C}$ are well defined independent of the choice of logarithmic alteration $X'\rightarrow X$.
\end{proposition}

    We require a preparatory lemma, which itself uses the following notation. Let $Y$ be a separated (complex) DM stack with logarithmic structure, and let $U\rightarrow Y$ be a strict \'etale cover by a separated scheme with logarithmic structure $U$. We adopt the notation $$R_s = \underbrace{U \times_Y U \times_Y \cdots \times_Y U}_{s\ \text{factors}}.$$ Note that the nerve of the presentation $R_2\rightarrow R_1$ of $Y$ is obtained by sending $[n]\mapsto R_{n+1}$ and setting degeneracy maps to be diagonals, and projection maps face maps.

\begin{lemma}\label{lem:NerveStuff}
    Applying the Kato--Nakayama functor to the nerve of a sheafy DM stack with locally free logarithmic structure $Y$ satisfying the hypotheses above yields our simplicial space with the following properties.
    \begin{enumerate}
        \item The natural map between the two topological realisations of a simplicial space described in \cite[Appendix A]{SEGAL1974293} is a homotopy equivalence.
        \item This common homotopy type of (1) coincides with the homotopy type of $\operatorname{KN}(Y)$.
    \end{enumerate}
\end{lemma}

\begin{proof}
    By \cite[Proposition A.1(iv)]{SEGAL1974293}, (1) will follow once we show that our nerve is good in the sense of \cite[Definition A.4]{SEGAL1974293}. To show that a simplicial space is good we check that for each $s$ the inclusion of the image of $\operatorname{KN}(R_{s-1})$ in $\operatorname{KN}(R_{s})$ under each degeneracy map is a closed cofibration. We will show more: this inclusion is both closed and open, thus is the inclusion of a union of connected components.
    
    Since the map $\tau\colon \operatorname{KN}(R_{s})\rightarrow R_s^{an}$ is continuous, it will suffice to check that the image of $R_{s-1}^{an}$ in $R_{s}^{an}$ is both closed and open. Moreover, it suffices to check that the image of $R_{s-1}$ in $R_{s}$ is closed and open because the Zariski topology has fewer open sets than the analytic topology.

    Openness follows by \stacks{02GE} because the map $U\rightarrow Y$ is (strict) \'etale and thus unramified. The image is closed because $U\rightarrow Y$ is separated which is by definition the same as the map $U\rightarrow U\times_YU$ being a closed immersion, see \stacks{03HL} and \stacks{04XB}.

    It remains to prove (2). Indeed the map $U\rightarrow Y$ is strict and so the fibre products in the definitions of $R_i$ commute with the Kato--Nakayama functor by \cite[Proposition~3.1]{cailotto2004katonakayama}. Thus the nerve of $R_1\rightarrow R_2$ is sent to the nerve of $\operatorname{KN}(R_1)\rightarrow \operatorname{KN}(R_2)$. Claim (2) follows because the Cech nerve of a local homeomorphism realises the homotopy type of a space by \cite[Corollary 1.5]{MR2045835}
\end{proof}

\begin{proof}[Proof of Proposition \ref{prop:HomotopyGroups}.]
    Any two logarithmic alterations $X',X''\rightarrow X$ admit a common refinement $X'\times_XX''$. Subdividing further we may assume this logarithmic alteration has locally free logarithmic structure. In the rest of this proof all logarithmic schemes are assumed to have locally free logarithmic structure. It will suffice to verify that for any sheafy DM stack with logarithmic structure $Y$, the isomorphism classes of the logarithmic homotopy groups are invariant under logarithmic alteration $Y'\rightarrow Y$. 
    
    Note that any logarithmic alteration $Y'\rightarrow Y$ can be factored $$Y'\rightarrow \overline{Y}\rightarrow Y$$ where the first map is a change of lattice and the second is a (possibly not projective) logarithmic modification. We now argue that it will suffice to handle the case that this second logarithmic modification is projective. Indeed after replacing $\overline{Y}$ by a log modification $\overline{Y}'$ which is projective over $Y$, and $Y'$ by $\widetilde{Y} = \overline{Y}'\times_{\overline{Y}}Y'$, both of the maps $\widetilde{Y}\rightarrow Y$ and $\widetilde{Y}\rightarrow Y'$ have this projectivity property.\footnote{See for example \cite[Prop 5.5.5 (v)]{EGA2}.} We can thus deduce the general case after we establish both the case that $Y'\rightarrow Y$ is a log blowup, and the case that $Y'\rightarrow Y$ is a Kummer root stack.
    
    We first handle the case that $Y$ is a scheme with logarithmic structure.
    Following the previous paragraph, we may assume that $Y'\rightarrow Y$ is either a log blowup or a root stack. The isomorphism class of the logarithmic homotopy groups to a fine and saturated logarithmic scheme $
    Y$ is invariant under logarithmic blowups 
    \cite[Proposition 9.3.5]{binda2025logarithmicmotivichomotopytheory}. 
    It is immediate from the definition of root stack presented in \cite[Definition~2.2.6.2]{MolchoWise} that passing to the root stack does not change the isomorphism type of the topological Kato--Nakayama space (because the group $S^1$ is divisible).

    In the case that $Y$ is a DM stack with logarithmic structure which is not a scheme with logarithmic structure, choose a strict \'etale groupoid presentation $s,t\colon R_2 = U\times_{Y} U \rightrightarrows R_1 = U$ of $Y$ where $U$ and $R_2$ are schemes as in the paragraph preceding Lemma \ref{lem:NerveStuff}. Base changing this presentation along the map $Y'\rightarrow Y$ yields a strict \'etale groupoid presentation of $Y'$, say $R_2'=U'\times_{Y'} U' \rightrightarrows U'$ where $R_1'=U'$ and $R_s'$ are DM stacks with logarithmic structure. The morphisms $R_i'\rightarrow R_i$ are the base change of a logarithmic alteration, and since for all $i$ we know that $R_i$ is a scheme with logarithmic structure it follows from the previous paragraph that these maps induce a logarithmic homotopy equivalence i.e. a homotopy equivalence on the level of Kato--Nakayama spaces. By Lemma \ref{lem:NerveStuff}(2) we have a homotopy equivalence at every level of the nerve of the Kato--Nakayama spaces of $Y$ and $Y'$, and the claim follows from \cite[Proposition A.1(ii)]{SEGAL1974293} together with Lemma \ref{lem:NerveStuff}(2).
\end{proof}

Finally we record the property which ensures that the logarithmic homotopy type behaves well in families. 

\begin{lemma}\label{lem:Serre Fibration}
    Let $X\rightarrow Y$ be a saturated, proper,\footnote{We follow the algebro--geometric convention that any proper map is separated. Separatedness is stated as a separate condition in \cite[Theorem~5.1]{NakayamaOgus2010RelativeRounding} where the authors often appeal to results in the topological category where conventions differ.} and logarithmically smooth morphism of sheafy DM stacks with locally free logarithmic structure. Assume that $Y$ is separated and of finite type over $\mathbb{C}$. Then the map $$|\mathsf{KN}(X)|\rightarrow |\mathsf{KN}(Y)|$$ is a Serre fibration.
\end{lemma}
\begin{proof}
    If $X$ and $Y$ were schemes with logarithmic structure we could deduce from \cite[Theorem~5.1]{NakayamaOgus2010RelativeRounding} that the morphism $\mathsf{KN}(X)\rightarrow \mathsf{KN}(Y)$ is a fibre bundle. To extend this claim to the situation that $X,Y$ are sheafy DM stacks with logarithmic structure recall that the proof of \textit{loc. cit.} proceeds by checking that the morphism of Kato--Nakayama spaces is a proper (separated) topological submersion. Being a topological submersion can be checked locally, and the local situation is unchanged when studying the Kato--Nakayama space of a sheafy DM stack with logarithmic structure because we  saw in the proof of Lemma \ref{lemdefn:KNSpace} that the map $|\operatorname{KN}(U)|\rightarrow |\operatorname{KN}(U)|/|\operatorname{KN}(R)|$ is a local homeomorphism of topological spaces. Properness descends along finite proper maps by the argument of \stacks{03GN}.
    It remains to note that the target is paracompact and Hausdorff (follows directly from construction of \cite[Section~1]{KatoNakayama1999}) and so the result follows by \cite[Theorem, Section~7.4]{May1999Concise}.
\end{proof}

\begin{remark}\label{rem:NeedRelative}
    The topological Kato--Nakayama space of a standard logarithmic point is $ S^1$, and in particular has non-trivial fundamental group. It thus follows from Lemma \ref{lem:Serre Fibration} that any scheme which is exact, proper, and logarithmically smooth over a standard logarithmic point has non-trivial logarithmic fundamental group.
\end{remark}

\subsection{Definition of logarithmic IHS varieties}
%For a logarithmically smooth morphism of Noetherian schemes of finite type $X\rightarrow S$, the \textit{sheaf of logarithmic differentials} is the sheaf of differentials $\Omega_{X/\mathcal{A}_{X/S}\times_{\mathcal{A}_S} S}$. This sheaf descends to define the sheaf of relative differentials for a morphism of Deligne--Mumford stacks with logarithmic structure. 

\begin{definition}
    Let $X\rightarrow S$ be a morphism of Deligne--Mumford stacks with logarithmic structure of finite type over $\mathbb{C}$. 
We denote the sheaf of kummer \'etale logarithmic differential 1-forms of \cite[3.8]{IllusieKatoNakayama2005} by $\omega_{X/S}^1$ and the associated logarithmic de Rham complex by $\omega_{X/S}^\bullet$ (also defined in \cite[3.8]{IllusieKatoNakayama2005}). 
\end{definition}

If $X\rightarrow S$ is logarithmically smooth of finite type then this sheaf is locally free of finite rank.  

\begin{definition}\label{defn:cotangent}
    Let $X\rightarrow S$ be a logarithmically smooth morphism of Noetherian Deligne--Mumford stacks with logarithmic structure. The \textit{relative logarithmic tangent bundle} $T_{X/S}^{\operatorname{log}}$ is the vector bundle dual to $\omega^1_{X/S}$.
\end{definition}

\begin{definition}\label{defn:Perfectpair}
    Let $X/\Delta_0$ be a smooth morphism of Noetherian DM stacks with logarithmic structure. An element $\eta$ of $H^0(X,\omega_{X/\Delta_0}^2)$ is nowhere degenerate if it induces an isomorphism from the logarithmic tangent bundle to the logarithmic cotangent bundle $$T_{X/\Delta_0}^{\operatorname{log}}\rightarrow \omega_{X/\Delta_0}^1.$$
\end{definition}

Logarithmic IHS varieties are families of logarithmic spaces, and are typically not represented by DM stacks with logarithmic structure. We extend Definition~\ref{defn:cotangent} and Definition~\ref{defn:Perfectpair} to the setting of logarithmic spaces.

\begin{definition}\label{defn:PerfectPair}
    Let $\Delta_0$ be the standard logarithmic point.\footnote{It suffices that $\Delta_0$ is static in the sense of \cite[Section 2]{logcoh}.} Let $X\rightarrow \Delta_0$ be a morphism of logarithmic spaces which is locally logarithmically smooth, and assume that $X$ admits a logarithmic alteration which is a Noetherian DM stack with logarithmic structure. We define a \textit{logarithmic k form}, or equivalently a \textit{global section of $\omega_{X'/\Delta_0}^k$}, to be an equivalence class of pairs $(q\colon X'\rightarrow X, \eta')$ where $X'\rightarrow X$ is any logarithmic alteration with domain a Noetherian DM stack with a logarithmically smooth logarithmic structure, and $\eta' \in H^0(X',\omega^k_{X/\Delta_0})$. The equivalence relation is pullback along refinements of logarithmic alterations. We say that a logarithmic 2-form $[(q\colon X'\rightarrow X, \eta')]$ is \textit{nowhere degenerate} if for some choice of $q$ the form $\eta'$ is nowhere degenerate in the sense of Definition \ref{defn:Perfectpair}.
\end{definition}

\begin{remark}
    The logarithmic cotangent bundle is same data as a locally free coherent sheaf in the logarithmic \'etale topology. A logarithmic two form is a global section of this sheaf, and nowhere degenerate means that the induced morphism from the logarithmic tangent bundle to the logarithmic cotangent bundle is an isomorphism of logarithmic coherent sheaves. More precisely, one defines the logarithmic cotangent bundle by applying the up functor of \cite[Section 1]{logcoh} to $\omega^2_{X'}$ for any choice of $X'$. 
\end{remark}

\begin{lemma}
    Definition \ref{defn:PerfectPair} is well defined independent of the logarithmic alteration ${X}\rightarrow X$.
\end{lemma}

The lemma amounts to checking that the logarithmic (co)-tangent bundle satisfies logarithmic \'etale descent whenever $X'\rightarrow S$ is a logarithmically smooth morphism of DM stacks with logarithmic structure. A variant of this result was originally proved in \cite[Section 5.1, especially Proposition 5.2]{FujisawaNakayama2020}, although \textit{loc.cit.} only consider descent under logarithmic blowups.

\begin{proof}
    Since each element in the equivalence class $X'\rightarrow S$ is logarithmically smooth, we know each possible representative of $\omega_{X/S}$ is locally free. Moreover, $S$ is static and thus for every representative $X'\rightarrow S$ which is logarithmically smooth, we know that $X'$ is static. Locally free sheaves on static logarithmic schemes satisfy descent in the full logarithmic \'etale topology by \cite[Section 2]{logcoh} and the result is proved. 
\end{proof}
Finally we record the definition of logarithmic IHS variety. 
\begin{definition}\label{defn:LogIHS}
    We say that a proper and locally logarithmically smooth morphism  $X\rightarrow S$ of logarithmic spaces is a family of logarithmic IHS varieties if the following hypotheses are satisfied whenever $\mathcal{X}_p$ is the base change of $X\rightarrow S$ along any standard logarithmic point $p$ of $S$.
     \begin{enumerate}[leftmargin = *]
         \item \textbf{Logarithmic simply connected.} The morphism of logarithmic fundamental groups $$\pi_1^{\operatorname{log}}(\mathcal{X}_p)\rightarrow \pi_1^{\operatorname{log}}(\Delta_0)$$ has trivial kernel and the morphism $X\rightarrow S$ has geometrically connected fibres.
         \item \textbf{Logarithmic symplectic structure.} The cohomology of the space $\mathcal{X}_p$ has the property $$H^{2,0}_{\log}(\mathcal X_p,\mathbb{C}) := H^0(\mathcal X_p,\omega_{\mathcal X_p/p}^2) = \mathbb{C}\eta$$ where $\eta$ is nowhere degenerate.
         \item \textbf{Projective up to logarithmic alteration.} There exists a logarithmic alteration of $\mathcal{X}_p$ which admits an ample polarisation (and is thus a DM stack with projective coarse moduli space).
     \end{enumerate}
\end{definition}

\section{The tropical generalised Kummer construction}\label{Bigsec:TropGenKummer}
This section constructs the tropical gadgets which will later allow us to work with logarithmic spaces. In this way, we will understand the right hand column of the table in Section \ref{sec:outline}.

We proceed by performing a tropical version of the generalised Kummer construction: Given a family of tropical abelian varieties over a rational polyhedral cone $\sigma$ we construct a family of tropical generalised Kummer varieties over $\sigma$. 

\subsection{Families of integral affine manifolds with singularities}\label{sec:TropAVTrait}
We recall the definition of integral affine manifold. Let $M_m \cong \mathbb{Z}^m$ be a finite rank free abelian group, and write $$
\operatorname{Aff}(M_m) = M_m \rtimes \mathrm{GL}_m(\mathbb{Z}).
$$
The following definitions coincide with those of Gross and Siebert \cite[Section 1]{GrossSiebert2006}.
\begin{definition}
    An integral affine manifold $B$ of dimension $m$ is an $m$-dimensional topological manifold $\underline{B}$ together with an equivalence class of atlases $$U_i\rightarrow \underline{B} \text{ and charts } \psi_i\colon U_i \xrightarrow{\cong} V_i$$
    where $U_i\rightarrow \underline{B}$ forms an open cover, each $V_i\subset M_m\otimes \mathbb{R}$ is a connected open subset, $\psi_i$ is a homeomorphism, and $$\psi_j \circ \psi_i^{-1}\colon \psi_i(U_i \cap U_j) \rightarrow V_j \text{ is the restriction of some } h \in \operatorname{Aff}(M_m).$$ The equivalence relation identifies atlases whose union is also an atlas. A morphism of integral affine manifolds $B\rightarrow B'$ is a morphism of topological manifolds $\underline{f}\colon \underline{B}\rightarrow \underline{B}'$ such that the induced map $$\psi_i'\circ \underline{f} \circ\psi_j^{-1}\colon V_i\rightarrow V_j'$$ is locally the restriction of an element in $M_{\operatorname{dim}(B')} \rtimes \operatorname{Hom}(M_{\operatorname{dim}(B)},M_{\operatorname{dim}(B')})$.
\end{definition}

The \textit{sheaf of integral affine functions} on $B$, denoted $\mathcal{O}_B$, is the sheaf of maps from $B$ to the integral affine manifold $\mathbb{R}$. We write $C\subset \mathcal{O}_B$ for the subsheaf of locally constant functions - note $C$ is isomorphic to the constant sheaf $\mathbb{Z}$. Whilst tropical abelian varieties can be thought of as families of integral affine manifolds, type III tropical generalised Kummer varieties fit the following, more general, definition which is also due to Gross and Siebert.

\begin{definition}[{\cite[Definition~1.15]{GrossSiebert2006}}]\label{defn:IASsing}
    An \textit{integral affine manifold with singularities (IAMS) of dimension $m$} $(B,D)$ is an $m$-dimensional topological manifold $B$, a closed subset $D\subset B$ called the \textit{singular locus} which locally can be written as a finite union of manifolds of codimension at least two, and an integral affine manifold $B_0$ with underlying topological manifold $\underline{B}_0 = B\backslash D.$ An isomorphism of IAMS is an isomorphism of topological manifolds $B\rightarrow B'$ which identifies singular loci, and which induces an isomorphism of integral affine manifolds $B_0\rightarrow B_0'$.
\end{definition}

A single IAMS captures the geometry of a one-parameter degeneration. In this paper we must work on the level of moduli spaces, and are thus led to study families of integral affine manifolds with singularities over a cone $\sigma$. 
To orient the reader we make the following definition now, before the examples listed are defined.

\begin{definition}\label{defn:IAMWithSing}
    The phrase \textit{a family of integral affine manifolds with singularities over the rational polyhedral cone $\sigma$} refers to any of the following constructions, which we introduce in the following sections: 
    \begin{enumerate}
        \item Families of tropical abelian varieties over $\sigma$ (Definition~\ref{defn:TropicalAV}).
        \item The (relative) tropical Hilbert scheme of points on a tropical abelian variety over $\sigma$ (Definition~\ref{defn:TropHilb}).
        \item Families of tropical generalised Kummer varieties over $\sigma$ (Definition~\ref{defn:TropGenKummer}).
    \end{enumerate}
\end{definition}

\begin{remark}\label{rem:IAMSB}
    With two exceptions, we will explain how points in the base of a family of integral affine manifolds with singularities specify an IAMS in the sense of Definition \ref{defn:IASsing}. The exceptions are the structures which type II degenerations of Hilbert schemes of points and generalised Kummer varieties give rise to. This structure is not an integral affine manifold with singularities - in the generalised Kummer case it is homeomorphic to a simplex, and satisfies the obvious definition of an \textit{integral affine manifold with boundary}. We do not suggest such a definition because it is not required to prove any of our theorems.
\end{remark}

\begin{remark}
    We do not attempt to formulate general axioms for our perspective on families of integral affine manifolds. Although we anticipate such language will eventually become desirable, it is not required for the arguments of this paper. We note that this speculative definition would generalise the category of piecewise linear spaces introduced in \cite[Section 1]{LogQuot}: subdivisions of a cone complex with domain a piecewise linear space biject with those monomorphisms with codomain a cone complex which induce a bijection on points.
\end{remark}

\begin{remark}
    Whilst moduli of tropical abelian varieties have been extensively studied, a significant section of the moduli of tropical abelian variety literature does not characterise a family of tropical abelian varieties as a family of integral affine manifolds. Exceptions include an implicit definition in the work of Kajiwara, Kato, and Nakayama \cite{KajiwaraKatoNakayama2008}, which is made explicit in recent work, see \cite[Definition~3.4]{botero2026pureextensionthetadivisor}, \cite{FeusiLopezTropAV}, and \cite[Definition~3.5]{FeusiMolchoOrbifoldCompactifiedJacobians}.
\end{remark}

\begin{remark}
    In this paper we are led to study families of integral affine manifolds with singularities because, in constructing $\mathcal{K}_{2n}$ and its universal family, we consider generalised Kummer varieties over a base fine and saturated logarithmic scheme $S$ of any dimension, and with logarithmic structure of any rank. In the case that $S=\Delta$, and thus the Artin fan of $S$ is $a^\ast \sigma = a^\ast \mathbb{N}$, it is not necessary to work with families of integral affine manifolds.
\end{remark}

Lemma \ref{lem:DataIAMS} will be used to establish the descriptions promised in Remark~\ref{rem:IAMSB}.
We will write $C^0(\underline{B})$ for the sheaf of continuous real valued functions on an integral affine manifold.
\begin{lemma}\label{lem:DataIAMS}
    An IAMS $(B,D)$ of dimension $n$ is determined by its underlying topological manifold $B$, the singular locus $D\subset B$, and the sheaf $\mathcal{O}_{B_0}\subset C^0(\underline{B}_0)$ of integral affine functions on $\underline{B}_0 = B \backslash D$. 
\end{lemma}
Note that the subsheaf $C\subset \mathcal{O}_{B_0}$ of constant functions is determined by the inclusion $\mathcal{O}_{B_0}\subset C^0(\underline{B}_0)$.
\begin{proof}
    Choose an open cover of $\underline{B}_0$ by simply connected sets $U_i\rightarrow \underline{B}_0$. Refining the cover if necessary, we may further assume that sections of the quotient sheaf $\mathcal{O}_{B_0}/C$ on each $U_i$ are isomorphic to $\mathbb{Z}^{n}$. 
    Since each $U_i$ is simply connected $H^1(U_i,\mathbb{Z})=0$, and thus any basis of $\Gamma(U_i,\mathcal{O}_{B_0}/C)$ admits a lift to a tuple of sections of $\mathcal{O}_{B_0}(U_i)$. This lift determines a map from $U_i$ to $\mathbb{R}^n$, and we choose an identification $\mathbb{R}^n \cong M_n\otimes \mathbb{R}$. Working chart locally on $U_i$ (with charts on our favourite atlas) we can apply the inverse function theorem to deduce that this map is a local homeomorphism. After replacing $U_i$ with a smaller open subset we have defined an immersion of integral affine manifolds $\psi_i\colon U_i\rightarrow M_n\otimes \mathbb{R}$. In particular, we have constructed an atlas, and thus the IAM structure on $B_0$ is determined.
\end{proof}

\subsection{Tropical abelian varieties}\label{sec:Tropical AV} Tropical abelian varieties \cite{MikhalkinZharkov08} are integral affine manifolds obtained as quotients of $\mathbb{Z}^n\otimes \mathbb{R}$ by a lattice. These structures have been studied in enumerative geometry \cite{Blomme2022TropicalCI,Blo23,Blo25,HalleRose2017,nishinou2020realizationtropicalcurvesabelian}, and also because they, and their moduli \cite{ChanEtAlTropSymm,BrannettiMeloViviani2011, Chan13}, exhibit many of the same structural properties as classical abelian varieties. 
\subsubsection{Core construction for families of tropical abelian varieties}\label{sec:TropAVDefn}
Following any of \cite{botero2026pureextensionthetadivisor}, \cite{FeusiLopezTropAV}, or \cite{FeusiMolchoOrbifoldCompactifiedJacobians}, we now record the definition of tropical abelian variety. The perspective presented here is implicit in work of Kajiwara, Kato, and Nakayama \cite[Definition~2.8]{KajiwaraKatoNakayama2008}.

Fix $\sigma$ a rational polyhedral cone, and let $\textbf{X},\textbf{Y}$ be finitely generated free abelian groups of rank $g$. We write $\varphi$ for a group homomorphism $$\varphi \colon \textbf{Y}\rightarrow \mathsf{Hom}(\textbf{X},\mathsf{Hom}(\sigma,\mathbb{Z})).$$ 
Note $\varphi$ is the same data as a bilinear map $$Q_{\varphi}\colon \textbf{X} \times \textbf{Y} \rightarrow \mathsf{Hom}(\sigma,\mathbb{Z}).$$

\begin{definition}
    An element $\phi$ of $\mathsf{Hom}(\textbf{X},\mathsf{Hom}(\sigma,\mathbb{Z}))$ is said to have \textit{bounded monodromy} (with respect to $\varphi$) if for each $x$ in $\textbf{X}$ there exist $y,y'$ in $\textbf{Y}$ such that, $$Q_{\varphi}(x,y) \leq \phi(x) \leq Q_{\varphi}(x,y').$$ We write $\mathsf{Hom}(\textbf{X},\mathsf{Hom}(\sigma,\mathbb{Z}))^{(\textbf{Y})}$ for the subgroup of bounded monodromy elements.
\end{definition}

\begin{definition}\label{defn:TropicalAV}
    Let $\mathpzc{A}(\varphi)$ be the group sheaf on $\textbf{RPC}_\sigma$ whose $\kappa/\sigma$ valued points are elements of $$\mathsf{Hom}(\textbf{X},\mathsf{Hom}(\kappa,\mathbb{Z}))^{(\textbf{Y})}/\textbf{Y},$$ and where pullback along $\kappa\rightarrow \kappa'$ is induced in the obvious way. A \textit{tropical abelian variety over} $\sigma$ is any sheaf on $\textbf{RPC}_\sigma$ isomorphic to $\mathpzc{A}(\varphi)/\sigma$ where $\varphi$ satisfies the following \textit{polarisability hypothesis}: There exists a morphism $\Phi\colon \textbf{Y} \rightarrow \textbf{X}$ with finite cokernel such that the induced bilinear form $Q_{\varphi}(\Phi(-),-)$ is symmetric and positive semi-definite (i.e. the image function is non-negative on $\sigma$).
\end{definition}

\begin{lemma}\label{lem:FunctorPointsTropAV}

There exists a tropical model $$j\colon \Lambda \rightarrow \mathpzc{A}(\varphi).$$
\end{lemma}

\begin{proof}
    This is \cite[Theorem~5.2.1]{KajiwaraKatoNakayama2008}.
\end{proof}
\subsubsection{Principal polarisations on tropical abelian varieties} 
Feusi and Iribar L\'opez define the dual $\mathpzc{A}^\vee$ to a tropical abelian variety $\mathpzc{A}$ as the bounded monodromy subsheaf of $\mathcal{E}xt(\mathpzc{A},\operatorname{Hom}(-,\mathbb{Z}))$, and note that the dual to $\mathpzc{A}(\varphi)$ is isomorphic to $\mathpzc{A}(\varphi^\vee)$ where $\varphi^\vee \colon \textbf{X}\rightarrow \mathsf{Hom}(\textbf{Y},\mathsf{Hom}(\sigma,\mathbb{Z}))$. 
See \cite[Section 3]{FeusiLopezTropAV} for details. We recall their definition of principally polarised tropical abelian varieties.

\begin{definition}\label{defn:TropAVPolarised}
    Let $\sigma$ be a rational polyhedral cone. A principally polarised tropical abelian variety over $\sigma$ is the data of a tropical abelian variety equipped with an isomorphism $\lambda\colon \mathpzc{A}\rightarrow \mathpzc{A}^\vee$ which is symmetric (i.e. equal to its dual under the identification $(\mathpzc{A}^\vee)^\vee \cong \mathpzc{A}$) and satisfies a positivity condition (see \cite[Definition 4.3]{FeusiLopezTropAV} for details). 
\end{definition}

The details of the definition of principal polarisation will not play an essential role in the sequel; it is defined only so that we may discuss moduli of principally polarised tropical abelian varieties with level structure.

\subsubsection{Integral affine manifolds from families of integral affine manifolds - the tropical abelian variety case} Let $\sigma$ be a rational polyhedral cone, set $\rho\rightarrow \sigma$ a point of $\sigma$, and let $\mathpzc{A}/\sigma$ be a family of tropical abelian varieties.

\begin{definition}\label{defn:ModuliIAMAbVar}
    Let $$\varphi \colon \textbf{Y}\rightarrow \mathsf{Hom}(\textbf{X},\mathsf{Hom}(\rho,\mathbb{Z}))$$ be a group homomorphism such that there is an isomorphism of group sheaves over $\rho$ $\mathpzc{A}(\varphi) \cong \mathpzc{A}\times_\sigma \rho$. The \textit{integral affine manifold $B(\mathpzc{A}\times_\sigma \rho/\rho)$ associated to the pair} $(\rho\rightarrow \sigma, \mathpzc{A}/\sigma)$ is defined to be $\mathsf{Hom}(\textbf{X},\mathsf{Hom}(\rho,\mathbb{Z}))^{(\textbf{Y})}\otimes \mathbb{R}/\varphi(\textbf{Y})$. Whenever we have fixed a presentation $\mathpzc{A}(\varphi) = \mathpzc{A}$, we can abbreviate $B(\mathpzc{A}\times_\sigma \rho) = B(\varphi)$.
\end{definition}

\begin{lemma}
     The integral affine manifold $B(\mathpzc{A}\times_\sigma \rho/\rho)$ is independent of the choice of $\varphi$. 
\end{lemma}

Note that $\mathbb{N}\cong \operatorname{Hom}(\mathbb{N},\mathbb{N})$ acts on the set of maps from $\mathbb{N}$ to a sheaf $\mathpzc{F}$ on \textbf{RPC} by pre-composition.
We say that a morphism $\mathpzc{G}\rightarrow \mathpzc{F}$ of sheaves on \textbf{RPC} is \textit{Kummer free} if it satisfies the property that whenever there is a positive integer $k$ and map $g$ from $\mathbb{N}$ to $\mathpzc{F}$ such that $k(g)$  factors through $\mathpzc{G}$, then also $g$ factors through $\mathpzc{G}$.
\begin{proof}
    The underlying topological space of $B(\mathpzc{A}\times_\sigma \rho/\rho)$ is the link to the morphism $\mathpzc{A}\times_\sigma\rho \rightarrow \rho$. By Lemma \ref{lem:DataIAMS} it remains to identify its sheaf of integral affine functions. A base $\{\iota\colon U_\iota\hookrightarrow B(\varphi)\}$ for the topology of the link is indexed by injective kummer-free maps over $\rho$, $$\iota \colon \tau \hookrightarrow \mathpzc{A}(\varphi)$$ from cones $\tau/\rho$ of dimension $\mathsf{dim}(|\mathpzc{A}\times_\sigma \rho|)\leq g+1$.
    The open set $U_\iota$ is the intersection of the preimage of $1\in \rho$ with the image of the interior of $|\tau \otimes \mathbb{R}_{\geq 0}|$ under $|\iota\otimes \mathbb{R}_{\geq 0}|$.
    The sheaf of integral affine functions is defined by declaring its sections on the open set $U_\iota$ to be $\operatorname{Hom}(\tau,\mathbb{Z})$. There is a natural inclusion of $\operatorname{Hom}(\tau,\mathbb{Z})$ into the set of continuous functions on $U_\iota$: indeed $U_\iota$ is a subset of the preimage of $1\in \rho$ and thus pulls back to a subset of $\tau$.
\end{proof}

Continuing notation from the proof, the dual to the morphism $\tau\rightarrow \rho$, which we will denote $f$, is a map $$f^\vee\colon \operatorname{Hom}(\rho,\mathbb{Z})\hookrightarrow \operatorname{Hom}(\tau,\mathbb{Z}).$$ The subsheaf of constant functions is then $\operatorname{Hom}(\rho,\mathbb{Z})$.

\begin{example}
    Suppose $\textbf{X},\textbf{Y} \cong \mathbb{Z}$ have rank $1$ and $\sigma=\mathbb{N}^2$. We specify a family of tropical abelian varieties through the map $$\begin{aligned}
\varphi\colon \textbf{Y} = \mathbb{Z} 
&\longrightarrow \mathbb{Z}^2 
  = \mathsf{Hom}\bigl(\mathbb{Z}, \mathsf{Hom}(\mathbb{N}^2,\mathbb{Z})\bigr) \\
u 
&\longmapsto (3u, 0).
\end{aligned}$$ 
Given a ray $\rho = (\alpha, \beta)$ in $\sigma$ write $\varphi_\rho$ for the base changed tropical abelian variety. Then for $\alpha \ne 0$, $B(\varphi_\rho)$ is the quotient of $\mathbb{R}$ by $3 \alpha\mathbb{Z}$, but $B(\varphi_\rho)$ is zero in the special case that $\alpha=0$.
\end{example}

\subsubsection{Moduli of tropical abelian varieties with level structure}\label{sec:modTropAVwithLevel} Let $V$ be the $\mathbb{Z}$ module $(\mathbb{Z}/m\mathbb{Z})^{2g}$ equipped with its standard symplectic form and set $\textbf{V}$ the constant sheaf associated to $V$. For a tropical abelian variety $\mathpzc{A}$ we write $$\mathpzc{A}[m] = \operatorname{Ker}(\mathpzc{A}\xrightarrow{\times m} \mathpzc{A})$$ for its $m$-torsion.

For a cone $\tau$ in \textbf{RPC} and positive integer $n$, we will write $[n]\colon \tau \rightarrow \tau$ for the multiplication by $n$ map. The following definition originates in \cite{FeusiLopezTropAV}.

\begin{definition}\label{defn:TropAVLevel}(Feusi--L\'opez \cite{FeusiLopezTropAV})
    A level $m$ structure on a tropical abelian variety $\mathpzc{A}\rightarrow \sigma$ is a surjection of group objects on $\textbf{RPC}_\sigma$ $$\pi_m\colon  [m]^\ast \textbf{V}\twoheadrightarrow [m]^\ast \mathpzc{A}[m].$$ For every test object $\tau/\sigma \in \textbf{RPC}_\sigma$ we require that $\ker(\pi_m)(\tau) \subseteq V$ is a coisotropic subgroup isomorphic as an abstract group to $(\mathbb{Z}/m\mathbb{Z})^{r_\tau}$ for some $g \leq r_\tau \leq 2g$.
\end{definition}

\begin{definition}
    The moduli space of principally polarised tropical abelian varieties with level $m$ structure $\mathfrak{A}_g(m)$ is the stack on \textbf{RPC} whose $\sigma$ valued points are principally polarised tropical abelian varieties over $\sigma$ with level $m$ structure. We write $\mathpzc{Y}(m)\rightarrow \mathfrak{A}_g(m)$ for its universal tropical abelian variety.
\end{definition}

After identifying the moduli space of polarised tropical abelian varieties with level structure with a space of quadratic forms as defined in \cite[Definition 3.1]{assaf2026levelstructurestropicalabelian}, see \cite{FeusiLopezTropAV}, our next result follows as in \cite[Remark 1.1.1]{KajiwaraKatoNakayamaPartVII}.

\begin{proposition}\label{prop:TropAVModel}
    For each integer $m\geq 3$ there are (infinitely many) choices of diagram with horizontal arrows tropical models 
    $$ 
\begin{tikzcd}
\mathpzc{Y}_{\Lambda}(m) \arrow[r] \arrow[d] 
  & \mathpzc{Y}(m) \arrow[d] \\
\mathfrak{A}_{g,\Lambda}(m) \arrow[r] 
  & \mathfrak{A}_g(m),
\end{tikzcd}
    $$
    and the left hand vertical map is an integral morphism of cone complexes.
\end{proposition}
\begin{proof}
    See \cite{FeusiLopezTropAV}.
\end{proof}

\subsection{Tropical Hilbert schemes via the stack of tropical supports}\label{sec:TropHilbSch}
Fix a family of tropical abelian varieties $\mathpzc{A}$ over a cone $\sigma$. 

\begin{definition}
    Let $\mathpzc{T}$ be a sheaf on \textbf{RPC} which admits a tropical model. A locally closed stratification $\mathcal{P}$ of $|\mathscr{T}|$ is \textit{conical} if $\mathscr{T}$ admits a tropical model such that each stratum of $\mathcal{P}$ is a union of interiors of cones.
\end{definition}

\begin{definition}\label{defn:weightedPLSUBDiv}
An $m$-\textit{weighted zero-dimensional piecewise linear subdivision} of a family $\mathpzc{A}\rightarrow \sigma$ of tropical abelian varieties (resp. of the projection $\mathbb{R}^k\times \sigma\rightarrow \sigma$) is a pair $$(\mathcal{P},w\colon \mathcal{P}\rightarrow \mathbb{N})$$ where,
\begin{enumerate}[leftmargin = *]
    \item The first element of our tuple $\mathcal{P}$ is a conical stratification of $|\mathpzc{A}|$ (resp of $|\mathbb{R}^k \times \sigma|$) such that the image of each stratum of $\mathcal{P}$ in $|\sigma|$ is the entire interior of a face.
    \item There exist $m$ sections $(s_1,..,s_m)$ of $\mathpzc{A}\rightarrow \sigma$ (resp. $\mathbb{R}^k\times \sigma \rightarrow \sigma$) such that:
    \begin{enumerate}[label=(\alph*)]
      \item The stratification $\mathcal{P}$ is the common refinement of the stratification pulled back from $\sigma$, and the stratification defined by the image of the $s_i$.
      \item The function $w$ is defined by $$w(\alpha)=\operatorname{card}\left(\{i:\alpha\subseteq |s_i(\sigma)|\}\right).$$ 
    \end{enumerate}
\end{enumerate}
\end{definition}

Note that the family of weighted zero-dimensional piecewise linear subdivisions determines the sections $s_i$ up to permutation, but not every collection of sections gives rise to such a family (as the images of sections may intersect over the interior of $\sigma$). Tracking weights is necessary if we are to define a summation map (Section \ref{sec:TropGenKummer}).

\begin{definition}\label{defn:TropHilb}
    Let $\mathpzc{A}/\sigma$ be a tropical abelian variety. The \textit{tropical Hilbert scheme} $\mathpzc{A}^{[n+1]}/\sigma$ of $\mathpzc{A}/\sigma$ is the sheaf on \textbf{RPC}$/\sigma$ whose $\kappa\rightarrow \sigma$ valued points are families of $(n+1)$-weighted zero-dimensional piecewise linear subdivisions of $\mathpzc{A}\times_\sigma \kappa$ over $\kappa$. The \textit{universal tropical subscheme} $\mathpzc{Z}_{n+1}(\mathpzc{A})$ is defined through its functor of points, whose value on a test object $\tau$ is the subset of $\operatorname{Hom}(\tau,\mathpzc{A}\times_\sigma \mathpzc{A}^{[n+1]})$ with the following property. First note that a map to $\mathpzc{A}^{[n+1]}$ definitionally determines a conical locally closed stratification of $\tau\times_\sigma \mathpzc{A}$. Our condition is that the image of $\tau$ in $\tau\times_\sigma \mathpzc{A}$ is contained within the union of strata of positive weight. 
\end{definition}
In particular, by construction there is a canonical monomorphism $$\mathpzc{Z}_{n+1}(\mathpzc{A})\rightarrow \mathpzc{A}\times \mathpzc{A}^{[n+1]}.$$
\begin{remark}
    One can define the stack of tropical supports of $\mathpzc{A}\rightarrow \sigma$ as the CFG over $\textbf{RPC}_\sigma$ whose $\tau\rightarrow \sigma$ valued points are monomorphisms in the category of sheaves on \textbf{RPC} $$\iota \colon \mathscr{T}\rightarrow \tau \times_{\sigma} \mathpzc{A}$$ satisfying:
    \begin{enumerate}[leftmargin=*]
        \item A flatness condition (parallel to \cite[Definition~1.3.3]{LogQuot})
        \item The morphism $\iota$ is represented in the category of piecewise linear spaces, see \cite[Section~1]{LogQuot}. 
    \end{enumerate}
    There is a canonical map from $\mathpzc{A}^{[n+1]}$ as defined here to the above stack of tropical supports. The current perspective differs in two, essentially cosmetic, ways. First, since this paper works with Hilbert schemes of points, it suffices to concern ourselves with $\mathcal{P}$ arising from a collection of sections. Second, strata track a weight $w$ which is closely related to K-tropicalisation as discussed in \cite{Quot2}.
\end{remark}

\subsubsection{Subdividing tropical Hilbert schemes} Note that the tropical moduli space $\mathpzc{A}^{[n+1]}\rightarrow \sigma$ carries a universal diagram (over $\sigma$)
\begin{equation}\label{eq:An-universal-square}
\begin{tikzcd}
\mathpzc{Z}_{n+1}(\mathpzc{A})\arrow[r]&\mathpzc{A}^{[n+1]} \times \mathpzc{A}
  \arrow[r] 
  \arrow[d, "\pi"'] 
& 
\mathpzc{A} \\
&\mathpzc{A}^{[n+1]} &
\end{tikzcd}
\end{equation}

\begin{proposition}\label{prop:ModelOfAn}
    There exist (infinitely many) choices of tropical model $$\mathpzc{A}^{[n+1]}(\Lambda)\rightarrow \mathpzc{A}^{[n+1]}$$ and for each such choice the fibre product $$\mathpzc{Z}_{n+1}^\Lambda(\mathpzc{A}):=\mathpzc{A}^{[n+1]}(\Lambda)\times_{\mathpzc{A}^{[n+1]}}\mathpzc{Z}_{n+1}(\mathpzc{A}) \rightarrow  \mathpzc{Z}_{n+1}(\mathpzc{A})$$ is a tropical model which is exact over $\mathpzc{A}^{[n+1]}$.
\end{proposition}

\begin{proof}
    We will first argue that $\mathpzc{A}^{[n+1]}$ admits a tropical model. Let $\overline{\mathpzc{A}}\rightarrow \mathpzc{A}$ be a tropical model as in Lemma~\ref{lem:FunctorPointsTropAV}.
    The fibre product $\overline{\mathpzc{A}}^{n+1}$ of $n+1$ copies of $\overline{\mathpzc{A}}/\sigma$ taken over $\sigma$ remains a cone complex. The $\kappa\rightarrow \sigma$ valued points are ordered tuples of $n+1$ sections $\{(s_1,...,s_{n+1})\}$ of $\overline{\mathpzc{A}}\times_\sigma \kappa\rightarrow \kappa$. Let $\Sigma'$ be a subdivision of $\mathpzc{A}^{n+1}$ such that on the interior of each cone the set of pairs $(i,j)$ such that $s_i = s_j$ is constant - such a subdivision exists because the stratification by locus where $s_i= s_j$ is conical in the sense of \cite[Section~1.2.1]{LogQuot}. By \cite[Proposition~1.2.5, Lemma~1.5.4]{LogQuot} and \cite{MR2248429}, $\Sigma'$ admits a $\operatorname{Sym}(n+1)$ equivariant subdivision. This subdivision descends to a subdivision $\mathpzc{A}^{[n+1]}(\Lambda)$ of the coarse space to the stack quotient $[\mathpzc{A}^{n+1}/\operatorname{Sym}(n+1)]$. To see that $\mathpzc{Z}_{n+1}^\Lambda(\mathpzc{A})$ is a cone complex note that the cone complex structure is induced by the sections used to define the locally closed stratification $\mathcal{P}^{[n+1]}$. Exactness requires showing that the image of each cone is an entire cone, but the projection of the image of a section is certainly the entire cone. 
\end{proof}

\begin{figure}
\begin{tikzpicture}[scale=1.0, line cap=round, line join=round]

  % Fundamental parallelogram
  \coordinate (O) at (0,0);
  \coordinate (A) at (5,0);
  \coordinate (B) at (1.8,3.2);
  \coordinate (C) at ($(A)+(B)$);

  \draw[thick] (O) -- (A) -- (C) -- (B) -- cycle;

  % grid
  \foreach \t in {0.2,0.4,0.6,0.8} {
    \draw[gray!45] ($(O)!\t!(A)$) -- ($(B)!\t!(C)$);
    \draw[gray!45] ($(O)!\t!(B)$) -- ($(A)!\t!(C)$);
  }

  % points
  \coordinate (P1) at ($(O)!1.20!(A)!0.90!(B)$);
  \coordinate (P2) at ($(O)!0.58!(A)!0.18!(B)$);
  \coordinate (P3) at ($(O)!0.30!(A)!0.72!(B)$);

  \fill[red] (P1) circle (2.2pt)
      node[below left=3pt]{};

  \fill[red] (P2) circle (2.2pt)
      node[below right=3pt] {};

  \fill[red] (P3) circle (2.2pt)
      node[above left=3pt] {};

  % projections showing coordinates
  \draw[dashed] (P1) -- (P1 |- O);
  \draw[dashed] (P1) -- (P1 -| O);

  \draw[dashed] (P2) -- (P2 |- O);
  \draw[dashed] (P2) -- (P2 -| O);

  \draw[dashed] (P3) -- (P3 |- O);
  \draw[dashed] (P3) -- (P3 -| O);

  % coordinate labels
  \node[below] at (P1 |- O) {$x_1$};
  \node[left]  at (P1 -| O) {$h_1$};

  \node[below] at (P2 |- O) {$x_2$};
  \node[left]  at (P2 -| O) {$h_2$};

  \node[below] at (P3 |- O) {$x_3$};
  \node[left]  at (P3 -| O) {$h_3$};

  \node at ($(O)!0.5!(C)+(0,0.15)$) {$\mathpzc{A}=\mathbb{R}^2/\Lambda$};
\end{tikzpicture}
\caption{Points of the stack of tropical supports $\mathpzc{A}^{[3]}$ correspond to tuples of three (unordered) points on a tropical abelian variety as shown. For $\sigma$ an arbitrary cone, $\sigma$ valued points correspond to subdivisions where the distances $x_1,x_2,x_3$ and $h_1,h_2,h_3$ adopt values in $\operatorname{Hom}(\sigma,\mathbb{N})$ instead of in $\mathbb{R}_{\geq 0}$. The flatness hypothesis ensures that two points cannot coincide on a proper subset of the interior of $\sigma$.}\label{fig:TropSupp}
\end{figure}
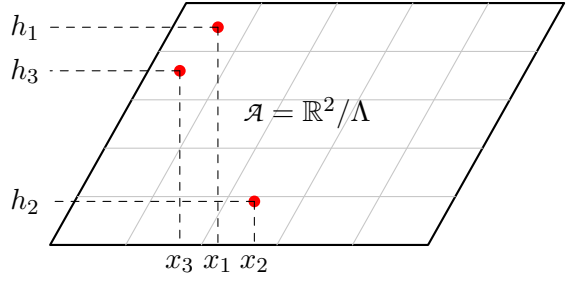

\subsubsection{Subdivisions over the moduli of tropical abelian varieties}
Passing to the stack of tropical supports is functorial in $\sigma$, and thus specifies a morphism of sheaves $$\mathpzc{Y}^{[n+1]}(n+1)\rightarrow \mathfrak{A}_g(n+1).$$

\begin{lemma}\label{lem:HilbPtsOverA2}
    Let $n\geq 2$. There is a diagram with horizontal arrows tropical models 
    $$ 
\begin{tikzcd}
\mathpzc{Y}_{\Lambda}^{[n+1]}(n+1) \arrow[r] \arrow[d] 
  & \mathpzc{Y}^{[n+1]}(n+1) \arrow[d] \\
\mathfrak{A}_{g,\Lambda}(n+1) \arrow[r] 
  & \mathfrak{A}_g(n+1).
\end{tikzcd}
    $$
    and the left hand vertical map is an integral morphism of cone complexes.
\end{lemma}
\begin{proof}
    To begin, choose any tropical model $\Lambda \rightarrow \mathfrak{A}_g(n+1)$ which exists by Proposition~\ref{prop:TropAVModel}. For each maximal cone $\sigma$ of $\Lambda$ choose a subdivision of $\mathpzc{Y}^{[n+1]}(n+1)\times_{\mathfrak{A}_g(n+1)}\sigma$, which exists by Proposition~\ref{prop:ModelOfAn}. We obtain a piecewise linear space $\mathscr{T}$ subdividing $\mathpzc{Y}^{[n+1]}(n+1)\times_{\mathfrak{A}_g(n+1)}\Lambda$ as follows. It will suffice to specify compatible piecewise linear subdivisions of $$\mathpzc{Y}^{[n+1]}(n+1)\times_{\mathfrak{A}_g(n+1)}\sigma$$ induced by pullback along the map $\sigma \rightarrow \Lambda$. Such a piecewise linear subdivision is induced by a conical locally closed stratification: we take the common refinement over all choices of maximal cone $\kappa \rightarrow \Lambda$ of the locally closed stratification of $$\mathpzc{Y}^{[n+1]}(n+1)\times_{\mathfrak{A}_g(n+1)}\sigma\times_{\mathfrak{A}_g(n+1)} \kappa$$ induced by the choice made in the second sentence of this paragraph.

    The resulting piecewise linear space $\mathscr{T}$ admits a tropical model $\Lambda'$ by \cite[Proposition 1.2.5]{LogQuot}. We thus have a diagram with horizontal maps tropical models,$$ 
\begin{tikzcd}
\mathpzc{Y}_{\Lambda}^{[n+1]}(n+1) \arrow[r] \arrow[d] 
  & \mathpzc{Y}^{[n+1]}(n+1) \arrow[d] \\
\mathfrak{A}_{g,\Lambda}(n+1) \arrow[r] 
  & \mathfrak{A}_g(n+1).
\end{tikzcd}
    $$
    Applying toric semistable reduction \cite{molcho2019universal} to the morphism on the left yields the required diagram.
\end{proof}
\subsection{The tropical generalised Kummer construction}\label{sec:TropGenKummer}
\subsubsection{Central definition} Since $\mathpzc{A}/\sigma$ is a group object in the category of sheaves on $\textbf{RPC}_\sigma$, there is a summing morphism $$\Sigma\colon \mathpzc{A}^{[n+1]}\rightarrow \mathpzc{A},$$
defined by taking a weighted piecewise linear polyhedral subdivision with associated sections $(s_1,...,s_{n+1})$ to $\sum_i s_i\in \mathsf{Hom}(\sigma, \mathpzc{A}).$

\begin{definition}\label{defn:TropGenKummer} The \textit{tropical generalised Kummer variety} $\operatorname{Kum}_{2n}(\mathpzc{A})$ associated to a tropical abelian variety $\mathpzc{A}$ over a cone $\sigma$ is the fibre product in the category of sheaves on $\textbf{RPC}_\sigma$,
$$\begin{tikzcd}
\operatorname{Kum}_{2n}(\mathpzc{A}) \arrow[d] \arrow[r] & \sigma \arrow[d, "0", hook] \\
{\mathpzc{A}^{[n+1]}} \arrow[r]                           & \mathpzc{A}
\end{tikzcd}.$$
\end{definition}
The definition captures, in the tropical category, the collection of $n+1$ points which sum to zero. By base changing the universal data associated to the tropical Hilbert scheme, there is a universal diagram,
\begin{equation}\label{eq:Kum-universal-square}
\begin{tikzcd}
\mathpzc{Z}_{K}(\mathpzc{A}) := Z_{n+1}(\mathpzc{A})\times_{\mathpzc{A}^{[n+1]}}\operatorname{Kum}_{2n}(\mathpzc{A})\arrow[r]\arrow[rd]&\mathpzc{X}(\mathpzc{A}):=\operatorname{Kum}_{2n}(\mathpzc{A}) \times \mathpzc{A}
  \arrow[r] 
  \arrow[d, "\pi"'] 
& 
\mathpzc{A} \\
&\operatorname{Kum}_{2n}(\mathpzc{A}) &
\end{tikzcd}
\end{equation}
Note that $\mathpzc{X}(\mathpzc{A})$ inherits from $\mathpzc{A}^{[n+1]}\times \mathpzc{A}$ a weighted subdivision $(\mathcal{P}^{[n+1]},w)$ of $|\mathpzc{X}(\mathpzc{A})|$.

\begin{proposition}\label{prop:KummerSubdivision}
    There are (infinitely many) choices of tropical model
    $$\operatorname{Kum}_{2n}^\Lambda(\mathpzc{A})\rightarrow \operatorname{Kum}_{2n}(\mathpzc{A})$$ and for each such choice the fibre product $$\mathpzc{Z}_{K}^\Lambda(\mathpzc{A}):=\operatorname{Kum}_{2n}^\Lambda(\mathpzc{A})\times_{\operatorname{Kum}_{2n}(\mathpzc{A})}\mathpzc{Z}_{K}(\mathpzc{A}) \rightarrow  \mathpzc{Z}_{K}(\mathpzc{A})$$
    is a tropical model.
\end{proposition}
\begin{proof}
    We will show existence of such a tropical model. The second clause follows identically to the proof of Proposition~\ref{prop:ModelOfAn}. 
    
    Fix a choice of tropical model $\mathpzc{A}^{[n+1]}(\Lambda)\rightarrow \mathpzc{A}^{[n+1]}$ whose existence is guaranteed by Proposition~\ref{prop:ModelOfAn}, and choose also a tropical model $\overline{\mathpzc{A}}\rightarrow\mathpzc{A}$.
    Note that the summation morphism on the sections defining $\mathcal{P}^{[n+1]}$ defines a map $\Sigma_{\operatorname{trop}}\colon \mathpzc{A}^{[n+1]}\rightarrow \mathpzc{A}$.
    Let $|e|$ be the image of the identity section of $\mathpzc{A}/\sigma$ in $|\mathpzc{A}|$, and note that $\{\Sigma_{\operatorname{trop}}^{-1}(|e|), |\mathpzc{A}^{[n+1]}| \backslash \Sigma_{\operatorname{trop}}^{-1}(|e|)\}$ pulls back along the first projection to a conical locally closed stratification of $|\mathpzc{A}^{[n+1]}\times\mathpzc{A}|$, and thus of $|\mathpzc{A}^{[n+1]}(\Lambda)\times\overline{\mathpzc{A}}|$.
    By \cite[Proposition 1.2.5]{LogQuot} such a conical locally closed subdivision of a cone complex defines a subdivision of piecewise linear spaces $\mathpzc{T}\rightarrow {\mathpzc{A}}^{[n+1]}(\Lambda)\times\overline{\mathpzc{A}}$. By \cite[Lemma~1.5.4]{LogQuot} there exists a tropical model $\mathpzc{A}^{[n+1]}\times\mathpzc{A}(\Lambda)\rightarrow \mathpzc{T}$. After replacing $\mathpzc{A}^{[n+1]}(\Lambda)$ by a subdivision, we may assume that the morphism $\mathpzc{A}^{[n+1]}\times\mathpzc{A}(\Lambda) \rightarrow \mathpzc{A}^{[n+1]}(\Lambda)$ is integral. Indeed the new subdivision is obtained by replacing our non-integral map with its stacky toric semistable reduction \cite{molcho2019universal}. Having updated $\mathpzc{A}^{[n+1]}(\Lambda)$ in this way we obtain the required tropical model $$\operatorname{Kum}_{2n}^\Lambda(\mathpzc{A}):= \operatorname{Kum}_{2n}(\mathpzc{A})\times_{\mathpzc{A}^{[n+1]}}\mathpzc{A}^{[n+1]}(\Lambda) \rightarrow \operatorname{Kum}_{2n}(\mathpzc{A})$$ 
    because the canonical morphism $\operatorname{Kum}_{2n}^\Lambda(\mathpzc{A})\hookrightarrow \mathpzc{A}^{[n+1]}(\Lambda)$ is the inclusion of a sub-cone complex.
\end{proof}

\begin{lemma}
    Let $n\geq 2$. There is a diagram with horizontal arrows tropical models 
    $$ 
\begin{tikzcd}
\operatorname{Kum}_{2n}^\Lambda(\mathpzc{Y}(n+1)) \arrow[r] \arrow[d] 
  & \operatorname{Kum}_{2n}(\mathpzc{Y}(n+1)) \arrow[d] \\
\mathfrak{A}_{g,\Lambda}(n+1) \arrow[r] 
  & \mathfrak{A}_g(n+1).
\end{tikzcd}
    $$
    and the left hand vertical map is an integral morphism of cone complexes.
\end{lemma}
\begin{proof}
    Follows from Proposition~\ref{prop:KummerSubdivision} by the argument used to deduce Lemma~\ref{lem:HilbPtsOverA2} from Proposition~\ref{prop:ModelOfAn}
\end{proof}

\subsubsection{The tropical generalised Kummer variety as a family of integral affine manifolds with singularities}\label{sec:TropGenKummerIAMSing}
Let $\rho$ be a ray and let $\mathpzc{A}\rightarrow \rho$ be a family of tropical abelian surfaces.
\begin{proposition}\label{prop:description link}
    There is a homeomorphism between the link of the map $$\operatorname{Kum}_{2n}(\mathpzc{A})\rightarrow \rho$$ and the closed subspace of $\operatorname{Sym}^{n+1}(B(\mathpzc{A}/\rho))$ specifying unordered tuples of $n+1$ points which sum to zero in the group law of $B(\mathpzc{A}/\rho)$.
\end{proposition}

\begin{proof}
    Apply the topological realisation functor to the square of Definition~\ref{defn:TropGenKummer} (setting $\sigma =\rho$), and note it remains cartesian as a diagram in the category of topological spaces.
\end{proof}

\begin{definition}\label{defn:IAMSGenKum}
    Let $\mathpzc{A}\rightarrow \rho$ be a tropical abelian variety of relative dimension two over a ray. There is an IAMS (possibly with boundary), denoted $B_K^{2n}(\mathpzc{A})$, associated to $$\pi \colon  \operatorname{Kum}_{2n}(\mathpzc{A})\rightarrow \rho$$ which is defined as follows.
\begin{enumerate}[leftmargin = *]
    \item The underlying topological space is the link of $\pi$.

    \item The codimension two singular locus $D\subset B_K^{2n}(\mathpzc{A})$ consists of points where at least two of $(s_1,...,s_{n+1})$ coincide. 

    \item The integral affine structure on $B_K^{2n}(\mathpzc{A})\backslash D$ is specified by the affine functions on a base for the topology. A base is indexed by kummer-free injective maps with domain a cone of dimension $2n+1$, $$\iota \colon \tau \hookrightarrow \operatorname{Kum}_{2n}(\mathpzc{A})$$ and the open set $U_\iota$ associated to this map is the intersection of the image of the interior of $\tau$ with the link of $\pi$. The sheaf of integral affine functions has sections $\operatorname{Hom}(\tau,\mathbb{Z})$ on this open set.
\end{enumerate}
\end{definition}

\begin{lemma}
    Definition \ref{defn:IAMSGenKum} specifies a (unique by Lemma~\ref{lem:DataIAMS}) integral affine manifold with singularities $B_K^{2n}(\mathpzc{A})$.
\end{lemma}
\begin{proof}
    The underlying topological space of $B_K^{2n}(\mathpzc{A})$ is canonically the moduli space of $n+1$ points on $B(\mathpzc{A})$ which sum to zero in the group law of $B(\mathpzc{A})$. 
    The singular locus consists of tuples where two or more points meet.
    The underlying topological space of $B^{2n}_K(\mathpzc{A})$ is a manifold (with boundary) by Proposition~\ref{prop:TypeIIHomeotype}. 
    The complement of the singular locus is the closed submanifold of the space of $n+1$ distinct points cut out by the closed condition that points sum to zero. This space inherits an integral affine manifold structure because the differential of the summation map is surjective.
\end{proof}
 Note that Definition~\ref{defn:IAMSGenKum} and Definition~\ref{defn:ModuliIAMAbVar} reflect the same construction of an integral affine manifold from a point in the base of a family of integral affine manifolds. 

\begin{remark}
    Making a similar definition for a tropical abelian variety of relative dimension one, yields a space where singularities occur in codimension one. The result is a gadget which could be sensibly described as an integral affine manifold with boundary. 
\end{remark}
\subsubsection{IAMS for type~II degeneration}
Let $\mathpzc{A}/\rho$ be a tropical abelian variety of dimension one; this case arises from the degeneration of abelian surfaces into a cycle of ruled surfaces. Later we realise the polyhedral complexes we now describe arise from algebro--geometric logarithmically smooth degenerations of IHS varieties. Consequently the results in this section are expected by \cite{BrownMazzon}, who computed the PL homeomorphism type of dlt degenerations of generalised Kummer varieties arising from degenerations of the abelian variety.
\begin{proposition}\label{prop:TypeIHomeotype}
    The link $B_K^{2n}(\mathpzc{A})$ of the generalised Kummer variety is homeomorphic to a simplex of dimension $n$.
\end{proposition}
\begin{proof}
    Write $\mathpzc{A} = \mathpzc{A}(\varphi)$. The link of $\mathpzc{A}^{[n+1]}\rightarrow \rho$ inherits the topology of $\operatorname{Sym}^{n+1}(B(\varphi))$. Points in the link of $\mathpzc{A}^{[n+1]}\rightarrow \rho$ which specify points summing to $0$ in the group law are precisely the link of $\operatorname{Kum}_{2n}(\mathpzc{A})$. It follows that the underlying topological space of $B_K^{2n}(\mathpzc{A})$ is homeomorphic to the moduli space of $n+1$ points in $S^1$ which sum to zero in the group law on $S^1$. This space is homeomorphic to $\Delta^{n}$ by \cite{Morton_1967} as noted in \cite[Proposition 6.3.4]{BrownMazzon}.
\end{proof}
\subsubsection{IAMS for type~III degeneration}\label{sec:IASTypeIII} Let $\mathpzc{A}/\rho$ be a tropical abelian variety of dimension two. This case arises from degenerations of an abelian surface into a union of toric varieties glued along their toric boundary. Once again, we are explaining how the prediction of \cite{BrownMazzon} arises in the logarithmic structure; we give a more elementary proof of \cite[Proposition 6.3.4]{BrownMazzon} due to Yuvraj Singh.
\begin{proposition}\label{prop:TypeIIHomeotype}
    The IAMS $B_K^{2n}(\mathpzc{A})$ is homeomorphic to $\mathbb{CP}^{n}$.
\end{proposition}

\begin{proof}
    The argument of Proposition~\ref{prop:TypeIHomeotype} establishes that the link is homeomorphic to the moduli space of $n+1$ points on $(S^1)^2$ which sum to zero. As a topological group, $(S^1)^2$ is isomorphic to the complex points of an (in fact any) elliptic curve $E$. Thus we are required to determine the homeomorphism type of the moduli space of effective degree $n+1$ divisors on an elliptic curve which sum to zero in the group law. But these are precisely the complex sections of a line bundle on $E$ of degree $n+1$ up to rescaling, and thus their moduli is $\mathbb{CP}^n$.
\end{proof}

\subsection{Describing vertex stars}\label{sec:DescribeVertexStar}
Throughout this section $\mathpzc{A}\rightarrow \rho$ is a tropical abelian variety over a ray, and we assume that $B(\mathpzc{A})$ has dimension $s\in \{1,2\}$. Fix a tropical model $\operatorname{Kum}_{2n}^\Lambda(\mathpzc{A})\rightarrow \operatorname{Kum}_{2n}(\mathpzc{A})$, whose existence is guaranteed by Proposition~\ref{prop:KummerSubdivision}, and write $B_\Lambda$ for the associated polyhedral complex structure subdividing the link $B_K^{2n}(\mathpzc{A})$ of $\operatorname{Kum}_{2n}(\mathpzc{A})/\rho$. Applying toric semistable reduction if necessary, we will assume that the morphism $\operatorname{Kum}_{2n}^\Lambda(\mathpzc{A})\rightarrow \rho$ is saturated.

Let $\rho$ be a ray in a cone complex $\Sigma$. Recall that the \textit{star} $\operatorname{Star}_\rho(\Sigma)$ is the cone complex whose cones are cones of $\Sigma$ containing $\rho$ modulo the linear span of $\rho$. We will now describe the star of a ray $\rho_v$ in $\operatorname{Kum}_{2n}^\Lambda(\mathpzc{A})$. Note $\rho_v$ is the same data as a vertex $v$ in $B_\Lambda$.
\begin{definition}
    The \textit{collision type} of a point in $B_K^{2n}(\mathpzc{A})$ is the unordered tuple of integers $(a_1,...,a_k)$ recording the partition of $n+1$ determined by the multiplicities in the corresponding point of $\text{Sym}^{n+1}(B(\mathpzc{A}/\rho))$.
\end{definition}

\begin{definition}\label{defn:TropHilbRn}
    The \textit{tropical Hilbert scheme} $(\mathbb{Z}^s)^{[n+1]}$ of $n+1$ points on $\mathbb{Z}^s$ is the sheaf on $\textbf{RPC}$ whose $\kappa$ valued points are families of $(n+1)$-weighted zero-dimensional piecewise linear subdivisions of $\mathbb{Z}^s\times \kappa$ over $\kappa$. The \textit{universal tropical subscheme} $\mathpzc{Z}_{n+1}(\mathbb{Z}^s)\hookrightarrow \mathbb{Z}^s\times (\mathbb{Z}^s)^{[n+1]}$ is defined similarly to $\mathpzc{Z}_{n+1}(\mathpzc{A})$ in Definition \ref{defn:TropHilb}.
\end{definition}

As in Section \ref{sec:TropGenKummer}, there is a summation map to $\mathbb{Z}^s$,
$$(\mathbb{Z}^s)^{[m]}\rightarrow \mathbb{Z}^s,$$
and thus given a tuple $\textbf{n}(v) = (n_1,...,n_k)$ of positive integers, we define a map
$$\Sigma^s_{n_1,...,n_k}\colon \prod_{i} (\mathbb{Z}^s)^{[n_i]} \rightarrow \mathbb{Z}^s$$
by summing the resulting elements in the functor of points of $\mathbb{Z}^s$. 

We write $\operatorname{Kum}_{\textbf{n}(v)}(\mathbb{Z}^s)$ for $(\Sigma^s_{n_1,...,n_k})^{-1}(0)$, and adopt the notation 
$$\mathpzc{Z}^{\textbf{n}(v)} = \bigsqcup_i \left(\mathpzc{Z}_{n_i}(\mathbb{Z}^s)\times_{(\mathbb{Z}^s)^{[n_i]}}\operatorname{Kum}_{\textbf{n}(v)}(\mathbb{Z}^s) \right).$$

\begin{lemma}\label{lem:TropicalByComponent}
    Let $v$ be a vertex of $B_\Lambda$ of collision type $\textbf{n}(v) = (n_1,...,n_k)$ specifying a ray $\rho_v$ of the cone complex $\operatorname{Kum}_{2n}^\Lambda(\mathpzc{A})$. For some choice of tropical model $$q\colon \operatorname{Kum}_{\textbf{n}(v)}^\Gamma(\mathbb{Z}^s) \rightarrow \operatorname{Kum}_{\textbf{n}(v)}(\mathbb{Z}^s),$$  there is an isomorphism between the star of $\rho_v$ and $\operatorname{Kum}_{\textbf{n}(v)}^\Gamma(\mathbb{Z}^s)$.
\end{lemma}

We will prove Lemma \ref{lem:TropicalByComponent} at the same time as Lemma \ref{lem:UnivOverComponent}, which we now state. 
Note that the fibre of $\mathpzc{Z}_{K}(\mathpzc{A})$ over $\rho_v\rightarrow \operatorname{Kum}^\Lambda_{2n}(\mathpzc{A})$ consists of several rays, which we label $\rho_1,...,\rho_k$. For each $1\leq j \leq k$ the star of $\rho_j$ maps to the star of $\rho_v$, giving $k$ morphisms of cone complexes $$\Theta_j:= \operatorname{Star}_{\rho_j}\!\left(Z_K^\Lambda(A)\right) \rightarrow \operatorname{Kum}_{\textbf{n}(v)}^\Gamma(\mathbb{Z}^s)$$ where we have used the identification of Lemma \ref{lem:TropicalByComponent} to specify the isomorphism type of the star of $\rho_v$.

\begin{lemma}\label{lem:UnivOverComponent}
    The pullback of $\mathpzc{Z}^{\textbf{n}(v)}$ along $q$ is the disjoint union $\bigsqcup_j \Theta_j$ (with the obvious projection map to $\operatorname{Kum}_{\textbf{n}(v)}^\Gamma(\mathbb{Z}^s)$).
\end{lemma}

We define a subsite $\textbf{RPC}_{\text{vert},\rho}$ of $\textbf{RPC}_\rho$ whose objects are cones $\tau\rightarrow \rho$ such that no ray of $\tau$ is mapped to zero, and this site inherits the face topology. For $\mathcal{F}$ a sheaf on \textbf{RPC} (resp. $\textbf{RPC}_\rho$), we write $\underline{\mathcal{F}}$ for the sheaf obtained by restricting the functor of ${\mathcal{F}}\times \rho$ (resp $\mathcal{F}$) to $\textbf{RPC}_{\text{vert},\rho}$. Note that restricting to this subfunctor is harmless for computing stars about rays which map surjectively to $\rho$ because a dense open subset of every cone that contains $\rho$ is not mapped to zero in $\rho$.

Choose a presentation of $\mathpzc{A}$ as the quotient of the bounded monodromy part of $\operatorname{Hom}(\mathbb{Z}^s,\operatorname{Hom}(\tau,\mathbb{Z}))$ by a sublattice. Since the bounded monodromy condition is trivial on the subsite $\textbf{RPC}_{\text{vert},\rho}$, the quotient map defines a "quotient by $Y$" morphism of sheaves,
 $$w\colon \underline{\mathbb{Z}^s}\cong \underline{\operatorname{Hom}(\mathbb{Z}^s,\operatorname{Hom}(-,\mathbb{Z}))}\rightarrow \underline{\mathpzc{A}}. 
$$

\begin{proof}[Proof of Lemma \ref{lem:TropicalByComponent} and Lemma \ref{lem:UnivOverComponent}]
    We begin by constructing a diagram of sheaves on $\textbf{RPC}_{\text{vert},\rho}$,
    $$
\begin{tikzcd}
\prod_i \underline{(\mathbb{Z}^s)^{[n_i]}} \arrow[r, dashed] \arrow[d,"\Sigma_\mathbb{Z}"] & {\underline{\mathpzc{A}^{[n+1]}}} \arrow[d] \\
\underline{\mathbb{Z}}^s \arrow[r] & \underline{\mathpzc{A}}
\end{tikzcd}
$$
The vertical maps are summation in the respective group laws as discussed at the start of Section~\ref{sec:TropGenKummer}. The bottom horizontal map is $w$. Write \textbf{0} for the vertex (i.e. ray but it is a vertex relative base ray) of $\prod_i \underline{(\mathbb{Z}^s)^{[n_i]}}$ which is zero in each factor.

The dashed arrow, which depends on $v$, is defined as follows. 
Note $v$ specifies for each $n_i$ in the collision type a point $x_i$ of $|\mathpzc{A}|$; since the morphism $\operatorname{Kum}_{2n}^\Lambda(\mathpzc{A})\rightarrow \rho$ is saturated, this point comes from a section of $\mathpzc{A}/\rho$.
For each $i$ there is a morphism from a subfunctor of $\underline{(\mathbb{Z}^s)^{[n_i]}}$ (the subfunctor is where the quotient introduces no new collisions) to $\underline{\mathpzc{A}^{[n_i]}}$ defined by $w$ followed by translation by adding $x_i$. 
There is a morphism from a subfunctor of $\prod_i\underline{(\mathbb{Z}^s)^{[n_i]}}$ to $\underline{\mathpzc{A}^{[n+1]}}$. 
This is the subfunctor where the quotient introduces no new collisions between points in each factor of $(\mathbb{Z}^s)^{[n_i]}$, and the images of the points associated to different productands $(\mathbb{Z}^s)^{[n_i]}$ do not meet. 
The dashed arrow represents the map defined on this subfunctor. 

The cone complex structure $\operatorname{Kum}_{2n}^\Lambda(\mathpzc{A})$ on $\Sigma_{\operatorname{trop}}^{-1}(\text{id})$ induces a cone complex structure on an open neighbourhood\footnote{This means it induces an equivalence class of cone complex structures about the point which all agree upon restriction to an open set.} of $\textbf{0}$ in $\Sigma_\mathbb{Z}^{-1}({0})$. The dashed arrow is well defined on this open neighbourhood (and the cone complex structure is induced on every open neighbourhood where the dashed arrow is well defined), and induces an isomorphism from the star of the vertex $\textbf{0}$ to the star of the vertex $v$. But the star at the vertex $\textbf{0}$ is a tropical model which we may write $$\operatorname{Kum}_{\textbf{n}(v)}^\Gamma(\mathbb{Z}^s) \rightarrow \operatorname{Kum}_{\textbf{n}(v)}(\mathbb{Z}^s).$$ This proves Lemma \ref{lem:TropicalByComponent}. We deduce Lemma \ref{lem:UnivOverComponent} by applying the same chain of reasoning to the universal tropical subscheme over $\underline{\mathpzc{A}^{[n+1]}}$ and $\prod_i \underline{(\mathbb{Z}^s)^{[n_i]}}$ respectively.
\end{proof}

\section{The logarithmic generalised Kummer construction}
In this section we establish that the generalised Kummer construction can be performed for families of logarithmic abelian varieties. We show that the resulting logarithmic generalised Kummer varieties are logarithmic IHS varieties in the sense of Definition \ref{defn:LogIHS}. 

\subsection{Degenerations of abelian varieties}
Recall that an abelian variety ${A}\rightarrow S$ over a scheme $S$ is a smooth and proper group object in the category of $S$-schemes of finite presentation whose geometric fibres are connected \cite[Definition 1.1]{FaltingsChai1990}. We write $A^\vee$ for $\mathsf{Pic}^0(A)$. Recall also that a \textit{principally polarised abelian variety} is a pair $({A},\varphi)$ consisting of an abelian variety ${A}$ and an isomorphism $\varphi \colon {A}\rightarrow {A}^\vee$ which is equal to its dual (under the identification of $A$ with its double dual) and which arises from an ample line bundle.

\subsubsection{Compactifying moduli of abelian varieties}\label{sec:log-abelian-variety} The moduli space of principally polarised abelian varieties with level-$n$ structure $\mathcal{A}_g^\circ(n)$ is not compact. Concretely, there exist families of abelian varieties over the generic point $\Delta^\circ$ of a trait which do not extend to abelian varieties over $\Delta$. Such families do, however, admit extensions by semi-abelian schemes \cite{FaltingsChai1990}.

\begin{definition}[{\cite[Definition 2.3]{FaltingsChai1990}}]
    A \emph{semi-abelian scheme} over $S$ is a smooth, separated, commutative group object $G \to S$ in the category of schemes with geometrically connected fibres, such that for each geometric point $s$ in $S$ there exists a short exact sequence of $s$-group schemes
$$
0 \rightarrow T_s \rightarrow G_s = G\times_S s \rightarrow A_s \rightarrow 0
$$
where $T_s \rightarrow s$ is a power of $\mathbb{G}_{m,s}$, and $A_s \to s$ is an abelian scheme.
\end{definition}

To construct a moduli space compactifying $\mathcal{A}_g^\circ(n)$, Kajiwara, Kato, and Nakayama introduce logarithmic abelian varieties \cite{KajiwaraKatoNakayama2008a,KajiwaraKatoNakayama2008}. The full definition of logarithmic abelian variety \cite[Definition~4.1]{KajiwaraKatoNakayama2008a} over a fine and saturated logarithmic scheme is lengthy, and we record only properties which are required later. Although not required here, the reader can find a brief exposition of the theory in \cite[Section~4.15]{MolchoWise}.

\begin{proposition}\label{prop:Summarise Log AV} Let $S$ be an atomic logarithmic scheme with logarithmic structure arising from a map to the Artin cone $a^\ast \sigma$. 
    \begin{enumerate}[leftmargin=*]
        \item Every logarithmic abelian variety $A\rightarrow S$ over $S$ is a proper group object in the category of logarithmic S-spaces. Possibly after replacing $S$ by a strict \'etale cover and $A$ by its pullback, there is a canonical commutative square, 
        $$\begin{tikzcd}
A \arrow[r, "\varpi_A"] \arrow[d] & a^\ast \mathsf{Trop}(A)= a^\ast \mathpzc{A} \arrow[d] \\
S \arrow[r] & a^\ast \sigma
\end{tikzcd}$$
where horizontal maps are strict, and $\mathsf{Trop}(A)$ is a family of tropical abelian varieties over $\sigma$ (Definition~\ref{defn:TropicalAV}). 
\item The morphism $\varpi_A$ is representable in the category of schemes with logarithmic structure, and the fibre over any map from the standard logarithmic point to $S\times_{a^\ast \sigma} a^\ast \mathsf{Trop}(A)$ is a semi-abelian variety. 
        \item Given any tropical model $\Lambda \rightarrow \mathsf{Trop}(A) = \mathpzc{A}$ the composed morphism $$a^\ast \Lambda \times_{a^\ast \mathpzc{A}} A\rightarrow A \rightarrow  S$$
        is a logarithmically smooth morphism of DM stacks with logarithmic structure. 
        \item  Let $A\rightarrow \Delta^\circ$ be a family of principally polarised abelian varieties over the generic point of the trait $\Delta$ (see notation section). Then, possibly after replacing $A$ by its base change along a finite map $\Delta\rightarrow \Delta$, there exists a unique polarised logarithmic abelian variety over $\Delta$ extending $A\rightarrow \Delta^\circ$.
    \end{enumerate}
\end{proposition}

We denote the closed logarithmic stratum in $S$ by $S_Z$. The proof of Proposition~\ref{prop:Summarise Log AV} is an exercise in translating \cite{KajiwaraKatoNakayama2008a} into the language of this paper. 

\begin{proof}
Logarithmic abelian varieties are group objects by \cite{KajiwaraKatoNakayama2008}, and proper by \cite[Section~11]{KajiwaraKatoNakayamaPartIV}.
The pullback of $A$ to $S_Z$ is a logarithmic abelian variety of constant degeneration \cite[Definition~3.3(1)]{KajiwaraKatoNakayama2008a}, and is thus the data of a polarisable logarithmic 1-motif\cite[Theorem~3.4]{KajiwaraKatoNakayama2008a}. 

A polarisable logarithmic 1-motif determines a local system $\textbf{Y}$ of finite rank free $\mathbb{Z}$ modules \cite[Definition~2.2(a)]{KajiwaraKatoNakayama2008} on $S$, and a semi-abelian variety $G$ over $S$ \cite[Definition~2.2(b)]{KajiwaraKatoNakayama2008} obtained by extending some abelian variety by a torus $T$. Writing $\textbf{X} = \operatorname{\mathcal{H}om}(T,\mathbb{G}_m)$ the logarithmic 1-motif specifies a pairing \cite[Section 2.3]{KajiwaraKatoNakayama2008} $$\textbf{Y} \times \textbf{X}
\rightarrow \operatorname{Hom}(\sigma,\mathbb{Z}) \text{ or equivalently } \varphi \colon \textbf{Y} \rightarrow \operatorname{Hom}(\textbf{X},\operatorname{Hom}(\sigma,\mathbb{Z})).$$

Pullback to a point $s$ in $S_Z$ specifies a morphism of $\mathbb{Z}$ modules $$\varphi \colon \textbf{Y} \rightarrow \operatorname{Hom}(\textbf{X},\operatorname{Hom}(\sigma,\mathbb{Z})).$$
The isomorphism type is locally constant in $S_Z$ and thus constant because $S$ is atomic which ensures that $S_Z$ is connected. Thus $\varphi$ is well defined independent of the choice of $s$ in $S$. We define $\operatorname{Trop}{A}$ to be $\mathpzc{A}(\varphi)$ in the sense of Definition~\ref{defn:TropicalAV}.

Next we construct the morphism $\varpi_A$.
By definition of logarithmic abelian variety there is a strict \'etale cover $\{S_i\rightarrow S\}$, which we may refine to be a cover by atomic logarithmic schemes, and upon each $S_i$ the $\varpi_A|_{S_i}$ is part of the definition of logarithmic abelian variety: in \cite[Definition~4.1.2]{KajiwaraKatoNakayama2008a} it is the final non-zero morphism in the displayed sequence of (2).

We have specified morphisms $\varpi_A|_{S_i}$ on the strict \'etale cover $S_i$ of $S$, and it remains to show these morphisms descend to a map from $S$.
Indeed, the map from $A$ to $a^\ast \mathsf{Trop}(A)$ is the stack quotient of $A$ by a subgroup $G$. 
The subgroup $G$ can be characterised as the maximal scheme with logarithmic structure which admits a monomorphism to $A$, has connected fibres over $S$, and is strict over $S$.
Being a subgroup with these properties is stable under strict \'etale pullback upon $S$. The map is thus well defined, and claim (1) is proved.
The short exact sequence of \cite[Definition~4.1.2]{KajiwaraKatoNakayama2008a} insists that fibres over any logarithmic point are isomorphic to the semi-abelian variety $G$. Claim (2) is thus proved.
Claim (3) is \cite[Theorem~8.1, especially Remark(3)]{KajiwaraKatoNakayamaPartIV}, and claim (4) is \cite[Theorem~1.6]{KajiwaraKatoNakayamaPartVII}.
\end{proof}

\subsubsection{Moduli of logarithmic abelian varieties with level structure}\label{sec:ModulilogAVLevelStr} Let $g$ and $m$ be positive integers, and fix $V$ a $\mathbb{Z}/m\mathbb{Z}$ module isomorphic to $(\mathbb{Z}/m\mathbb{Z})^{2g}$ equipped with the standard symplectic form. 

\begin{definition}[{\cite[Definition~1.4]{KajiwaraKatoNakayamaPartVII}}] A level $m$ structure on a logarithmic abelian variety $A\rightarrow S$ of dimension $g$ is an isomorphism $$V_S\cong \operatorname{ker}(A\xrightarrow{m} A).$$% where the pairing upon $\operatorname{ker}(A\xrightarrow{m} A)$ is the Weil pairing, and $V_S$ is the constant group sheaf over $S$.
\end{definition}

We will write $A^\circ$ for the kernel of the morphism $\varpi_A$: by Proposition~\ref{prop:Summarise Log AV}(1) this is a semi-abelian variety.
Let $\mathcal A_g(m)$ be the moduli space of principally polarised logarithmic abelian varieties with level $m$ structure, a moduli functor on the site $\mathsf{fs}$ of fine and saturated logarithmic schemes. The universal abelian variety is denoted $\pi\colon  Y_g(m)\rightarrow \mathcal{A}_g(m)$.

\begin{theorem}\label{thm:commuting square}
    Let $m\geq 3$. There is a commutative diagram with strict horizontal maps which are represented in the category of Noetherian DM stacks with logarithmic structure, $$\begin{tikzcd}
 Y_g(m) \arrow[r, "\widetilde{\varpi}"] \arrow[d, "\pi"'] 
& a^\ast \mathpzc{Y}_g(m) \arrow[d, "a^\ast \pi^{\mathsf{trop}}"] \\
\mathcal A_g(m) \arrow[r, "\varpi"] 
& a^\ast  \mathfrak{A}_g(m)
\end{tikzcd}.$$
Moreover, $\mathcal A_g(m)$ is a proper logarithmic space and the base change of $\mathcal{A}_g(m)$ along any tropical model $\Lambda\rightarrow \mathfrak{A}_g(m)$ is a toroidal compactification of the moduli space of (not logarithmic) abelian varieties.
\end{theorem}

\begin{proof}
    Note that $\mathcal A_g(m)$ has a model proper over $\mathbb{Z}[1/m]$ which is of finite type over $\mathbb{Z}$ \cite[Theorem~1.7]{KajiwaraKatoNakayamaPartVII}.
    Thus to specify a map from $\mathcal{A}_g(m)$ to $a^\ast \mathfrak{A}_g(m)$ it will suffice to specify the map on functors of points for Noetherian logarithmic schemes. Noetherian logarithmic schemes admit covers by atomic logarithmic schemes \cite[Lemma 2.2.5]{AWbirational}, so it will suffice to specify $\varpi$ on $S$ valued points for $S$ an atomic Noetherian logarithmic scheme. Since $ Y_g(m)$ is proper and thus finite type over $\mathcal A_g(m)$, the same construction can be used to build $\widetilde{\varpi}$.

    We first construct $\varpi$. An $S$ valued point of $\mathcal A_g(m)$ is a logarithmic abelian variety over $S$ equipped with level structure. For $S$ Noetherian and atomic with Artin cone $\varpi_S\colon S\rightarrow a^\ast \sigma$, Proposition~\ref{prop:Summarise Log AV}(1) constructs a family of tropical abelian varieties over $\sigma$. Writing $A^\circ$ for the kernel of $\pi$, $V_S$ for the constant sheaf associated to the module $V\cong (\mathbb{Z}/m\mathbb{Z})^{2g}$ equipped with its standard symplectic form, and composing the level structure with the surjection of Proposition~\ref{prop:Summarise Log AV}(1) yields a surjection $$ \varpi_S^\ast a^\ast V_\sigma \cong V_S\xrightarrow{\sim} A[m]\rightarrow A[m]/A^\circ[m] \cong  \varpi_S^\ast a^\ast \mathpzc{A}[m].$$
    The kernel is coisotropic by \cite[Theorem~9.3(3)]{FeusiLopezTropAV}.
    We thus obtain a morphism from $\sigma$ to $\mathfrak{A}_g(m)$. The map $\varpi$ is the composition $$S \rightarrow a^\ast \sigma \rightarrow a^\ast \mathfrak{A}_g(m).$$ 
    Similarly we obtain the morphism $\widetilde{\varpi}$ from the morphisms of Proposition~\ref{prop:Summarise Log AV}(1).
    
    To see that $\varpi$ is represented in the category of Noetherian DM stacks with logarithmic structure it will suffice to check that for any map from an atomic logarithmic scheme $S\rightarrow a^\ast \mathfrak{A}_g(m) $ the fibre product is an algebraic space with logarithmic structure. Since when $S$ is atomic, the map on tropicalisations corresponds to the map from a cone to $\mathfrak{A}_g$, the claim is demonstrated by the argument of \cite[Section 7.2]{KajiwaraKatoNakayamaPartVII}. The remainder of our theorem concerning representability of the moduli stack and its base changes are restatements of \cite[Theorem~1.6 and Theorem~1.7]{KajiwaraKatoNakayamaPartVII} in the language of this paper.
\end{proof}

\begin{proposition}\label{prop:LogAVSmooth}
    Any principally polarised logarithmic abelian surface $g\colon A\rightarrow \Delta_0$ over the standard logarithmic point is polarised smoothable.
\end{proposition}

Principality is not required, but we never defined non-principal polarisations.

\begin{proof}
    If $A$ is obtained by equipping a conventional abelian surface with strict logarithmic structure over $p$ then the result is clear. Otherwise replace $\Delta_0$ by a logarithmic alteration such that it is possible to choose a basis for the $3$ torsion of $A$. Having done this, we can specify a morphism from $\Delta_0$ to the boundary of $\mathcal{A}_2(3)$. Since the logarithmic structure upon $\Delta_0$ is valuative, this morphism persists after replacing $\mathcal{A}_2(3)$ with a logarithmic modification. We may thus assume that the image of $\Delta_0$ is a point $p$ in the interior (i.e. in a unique boundary divisor) of a boundary divisor in a (representable) logarithmic alteration of $\mathcal{A}_2(3)$. Recall from \cite[Theorem~1.7]{KajiwaraKatoNakayamaPartVII} that the boundary of a (representable) logarithmic modification of $\mathcal{A}_2(3)$ is toroidal, and thus there exists a map from a trait to $\mathcal{A}_2(3)$ whose generic point does not hit the toroidal boundary, but whose special point is $p$. The universal abelian variety pulls back to the required smoothing $\widetilde{A}$.
\end{proof}

\subsection{Constructing logarithmic generalised Kummer varieties}
Throughout this section we fix $S$ an atomic and globally charted logarithmic scheme, with global Artin cone $S\rightarrow a^\ast \sigma$. We will often choose a morphism $T\rightarrow S$, and whenever we do we assume that $T$ is atomic with Artin cone $a^\ast \tau$. 

\subsubsection{The logarithmic Hilbert scheme of points and its tropicalisation map}\label{sec:LogHilbPts}
Fix a logarithmic abelian surface ${A}\rightarrow S$. The tropical abelian surface prescribed by Proposition~\ref{prop:Summarise Log AV} is denoted $\mathsf{Trop}(A) = \mathpzc{A}\rightarrow \sigma$. We will write $A_T= A\times_S T$ and $\mathpzc{A}_\tau =\mathpzc{A}\times_\sigma \tau$.

\begin{definition}
A \textit{tropically framed length $n$ closed subscheme of $A_T/T$} is a pair $$(h \colon \Lambda\rightarrow \mathpzc{A}_\tau, \iota\colon Z\hookrightarrow A_T\times_{a^\ast \mathpzc{A}_\tau} a^\ast \Lambda )$$ where:
\begin{enumerate}[leftmargin = *]
    \item The morphism $h$ is a tropical model such that the composed morphism of cone complexes $ \Lambda\rightarrow \tau$ is saturated.
    \item The map $\iota$ is a strict closed immersion of a length $n$ closed subscheme such that the composed map $Z\rightarrow T$ is logarithmically flat.
\end{enumerate}
\end{definition}
\begin{definition}
    A \textit{length $n$ closed subscheme of $A_T/T$} is an equivalence class of tropically framed length $n$ closed subscheme of $A_T/T$. The equivalence relation identifies pairs $(h\colon \Lambda \rightarrow \mathpzc{A}_\tau,\iota)$ with $(h'\colon \Lambda'\rightarrow \mathpzc{A}_{\tau},\iota')$ if taking the base change of $\iota$ and $\iota'$ respectively with the projection maps $$A\times_{a^\ast \mathpzc{A}}a^\ast \Lambda \leftarrow A\times_{a^\ast \mathpzc{A}}a^\ast \Lambda\times_{a^\ast \mathpzc{A}}a^\ast \Lambda'\rightarrow A\times_{a^\ast \mathpzc{A}} a^\ast \Lambda'$$ specifies the same closed subscheme. 
    The \textit{logarithmic Hilbert scheme} $A^{[n]}/S$ of length $n$ associated to $A/S$ is the sheafification of the presheaf which assigns to a map $T\rightarrow S$ the set of length $n$ closed subschemes of $A_T/T$.
\end{definition}

We now describe the canonical morphism $$A^{[n+1]}\rightarrow a^\ast \mathpzc{A}^{[n+1]}.$$ By Noetherian approximation, it will suffice to specify this morphism on $T$ valued points for $T/S$ an atomic logarithmic scheme. The data of $h$ specifies a subdivision $\Lambda \rightarrow \mathpzc{A}$ such that the composed morphism $h\colon \Lambda\rightarrow \tau$ is integral. 

\begin{definition}
    A \textit{vertex} of $\Lambda/\tau$ is a cone $v$ of $\Lambda$ such that $|h|$ maps $|v|$ isomorphically to $|\tau|$.
\end{definition}
Choose a geometric point $t$ of $T$ whose image in $a^\ast \tau$ lies in the closed point. Irreducible components $W_v$ of the $t$-stack $$(A\times_{a^\ast \mathpzc{A}}a^\ast \Lambda)\times_T t$$ biject with vertices $v$ of $\Lambda$. Logarithmic flatness ensures that the closed subscheme $Z\times_T t$ is contained in the union of the interiors of the irreducible components. 

We write $\mathscr{T}\rightarrow \mathpzc{A}^{[n+1]}$ for the subdivision of piecewise linear spaces associated to the locally closed stratification whose strata are vertices $v$ of $\Lambda$ such that $W_v\cap Z$ is not empty. The length of subscheme on each component defines a function from the stratification $\mathcal{P}_\mathscr{T}$ specifying $\mathcal{T}$, $\mathcal{P}_\mathscr{T}\rightarrow \mathbb{N}$. We have thus specified a map $$\sigma \rightarrow \mathpzc{A}^{[n+1]} \text{ and thus a composition } S\rightarrow a^\ast \sigma \rightarrow a^\ast \mathpzc{A}^{[n+1]}.$$
By the argument of \cite[Proposition 6.3.1]{LogQuot}, this assignment descends to define a functor $$\varpi_n\colon {A}^{[n+1]}\rightarrow a^\ast \mathpzc{A}^{[n+1]}.$$

\subsubsection{Properties of the logarithmic Hilbert scheme}

\begin{theorem}\label{thm:LogHilbProper}
    Continue notation from Proposition~\ref{prop:ModelOfAn}, and fix a choice of $\Lambda$ from the same proposition. Then the morphism $$A^{[m]}_\Lambda := A^{[m]}\times_{a^\ast \mathpzc{A}^{[m]}}a^\ast {\mathpzc{A}^{[m]}(\Lambda)}\rightarrow S$$ is a vertical and proper morphism of DM stacks with logarithmic structure. %This morphism fits into a commutative square with horizontal arrows logarithmic modifications,
%    $$
%\begin{tikzcd}
%A^{[m]}_\Lambda \arrow[r] \arrow[d] & A^{[m]} \arrow[d] \\
%S_\Lambda \arrow[r] & S.
%\end{tikzcd}
%$$
Moreover, $A^{[m]}_\Lambda$ admits an open immersion into the conventional relative Hilbert scheme of points,
$$H^{[m]}_\Lambda := \operatorname{Hilb}^{[m]}\left({A}_{\operatorname{univ}} = \left( A\times_{a^\ast \mathpzc{A}}a^\ast\left(\mathpzc{Z}^{[m]}(\Lambda)\right)\right)/S\times_{a^\ast \sigma}a^\ast\mathpzc{A}^{[m]}(\Lambda)\right).$$
\end{theorem}

\begin{proof}
    All of the arguments for this claim are recorded in \cite{LogQuot}, but \textit{op. cit.} is written for a regular crossings pair, so we must note they remain valid in the present situation. The open immersion in the final claim follows from the argument of \cite[Section 6.2]{LogQuot}, and as in \textit{op. cit.} it then follows from \stacks{0D01} that $A^{[m]}_\Lambda$ is an algebraic stack. This algebraic stack is proper over $S$ by the argument of \cite[Theorem~7.0.1]{LogQuot}. The algebraic stack has finite stabilisers by the argument of \cite[Theorem~8.1.4]{LogQuot}. Vertical holds by construction.
\end{proof}

\begin{corollary}\label{corr:logsmooth}
    The proper model ${A}^{[m]}_\Lambda$  of $A^{[m]}$ is logarithmically smooth over $S$. In particular, $A^{[m]}$ is locally logarithmically smooth over $S$.
\end{corollary}

\begin{proof}
Theorem~\ref{thm:LogHilbProper} expresses $A^{[m]}_\Lambda$ as an open subset of the relative Hilbert scheme of points $H^{[m]}_\Lambda$ to a family of surfaces.
Since the map ${A}_{\operatorname{univ}}\rightarrow \operatorname{Exp}:=S\times_{a^\ast \sigma}a^\ast\mathpzc{A}^{[m]}_\Lambda$ 
is logarithmically smooth and strict, it is smooth.\footnote{Note in the definition of $A_{\operatorname{univ}}$ we discarded the non-strict locus - i.e. the singular locus where components meet - when defining our universal expansion.}
By \cite[Theorem~2.9]{Fogarty1968}, $A^{[m]}_\Lambda$ is smooth over $\operatorname{Exp}$, and is thus logarithmically smooth over $S$.
\end{proof}

\subsubsection{The logarithmic generalised Kummer construction}\label{sec:LogKummer}
In the sequel we assume that $A\rightarrow S$ is an abelian surface. Write $\mathsf{Sym}^m(A)$ for the sheaf whose $T\rightarrow S$ valued points are equivalence classes of pairs $(q\colon (A\times_ST)'\rightarrow A\times_ST,D)$ where  $q$ is a logarithmic alteration saturated over $T$, and $D$ is a $T$ section of of the m-fold relative symmetric product of the map $(A\times_ST)'\rightarrow T$. This open set is characterised by requiring that points are contained within the strict locus.
\begin{lemma}\label{lem:LogHCGpMap}
    There are logarithmic Hilbert--Chow ($\mathsf{HC}$) and summation ($\Sigma$) maps
    $${A}^{[m]}\xlongrightarrow{\mathsf{HC}} \mathsf{Sym}^m(A)\xlongrightarrow{\Sigma} {A}.$$
\end{lemma}
\begin{proof}
    Since all spaces involved are noetherian it will suffice to work on $T$ valued points for $T\rightarrow a^\ast \tau$ atomic. The morphism of atomic logarithmic schemes from $T$ to $S$ induces a map $\tau \rightarrow \sigma$ where $a^\ast \sigma$ is the Artin cone of $S$. A morphism from $T$ to $A^{[m]}$ is the data of a closed subscheme of a logarithmic alteration $(T\times_S A)'\rightarrow T\times_S A$ which is logarithmic flat and of length $m$ over $T$. Since the subscheme has relative dimension zero, it must lie in the locus of $(T\times_S A)'$ which is strict over $T$, and is thus flat over $T$. Applying \cite[Paragraph II.3.2]{Iversen1970} we obtain a morphism from the relative Hilbert scheme of $(T\times_S A)'/ T$ to the relative symmetric product of the same morphism. But all sections in the image of the logarithmic Hilbert scheme under this map are strict, so admit a canonical upgrade to sections of the morphism of logarithmic schemes. Thus we have specified a morphism $\mathsf{HC}$ from $T$ to $\operatorname{Sym}^m(A)$ taken in the logarithmic category. 
    The morphism $\Sigma$ is induced from the group structure upon $A$ in the usual way.
\end{proof}

\begin{definition}\label{defn:logKummer}
    The logarithmic Kummer variety $\operatorname{Kum}_{2n}(A/S)$ is the fibre product 
$$
\begin{tikzcd}
\operatorname{Kum}_{2n}({A}/S) \arrow[r] \arrow[d] & S \arrow[d,"e"] \\
{{A}^{[n+1]}} \arrow[r,"h"]                   & {A}
\end{tikzcd},  
$$
where $e$ is the identity section and $h$ is obtained by composing the morphisms of Lemma \ref{lem:LogHCGpMap}.
\end{definition}

\begin{proposition}\label{prop:ProperLogSmooth}
    Continue notation from Proposition~\ref{prop:KummerSubdivision}, and make a choice of $\Lambda$. The resulting morphism $$\mathrm{Kum}^\Lambda_{2n}({A}/S)\rightarrow S$$
    is a vertical, proper, and logarithmically smooth morphism of DM stacks with logarithmic structure. 
\end{proposition}
\begin{proof}
    Note that there is a cartesian square,
$$
\begin{tikzcd}
\operatorname{Kum}_{2n}^\Lambda(A/S) \arrow[r] \arrow[d] & S \arrow[d] \\
{A^{[n+1]}}\times_{a^\ast \mathpzc{A}^{[n+1]}}{a^\ast\operatorname{Kum}_{2n}^\Lambda(\mathpzc{A})} \arrow[r]                   & A\times_{a^\ast \mathpzc{A}} a^\ast \sigma
\end{tikzcd},
$$
which we obtain by base change of the diagram in Definition~\ref{defn:logKummer} along compatible morphisms from $a^\ast \operatorname{Kum}_{2n}^\Lambda(\mathpzc{A})\rightarrow a^\ast \sigma$ to left and right columns of the associated tropical diagram (see Definition~\ref{defn:TropGenKummer}).
It will thus suffice to establish that the bottom horizontal map, which is the base change of the summation map $h$ of Definition~\ref{defn:logKummer}, is logarithmically smooth. 

The required logarithmic smoothness is equivalent to conventional smoothness of the first morphism, denoted $\Sigma$, in the composition $${A^{[n+1]}}\times_{a^\ast \mathpzc{A}^{[n+1]}}{a^\ast \operatorname{Kum}_{2n}^\Lambda(\mathpzc{A})} \xrightarrow{\Sigma}                    A\times_{a^\ast \mathpzc{A}}{a^\ast \operatorname{Kum}_{2n}^\Lambda(\mathpzc{A})}\rightarrow S\times_{a^\ast \sigma}a^\ast\operatorname{Kum}_{2n}(\mathpzc{A}).$$

We begin by establishing flatness of $\Sigma$. Corollary \ref{corr:logsmooth} asserts that the above composition is smooth, and in particular flat. By critère de platitude par fibres \stacks{039C}(in the reference using (1) implies (2) direction), flatness of $\Sigma$ may be checked after pulling back along maps from the spectrum of an arbitrary field to $S\times_{a^\ast \sigma}a^\ast\operatorname{Kum}_{2n}^\Lambda(\mathpzc{A})$. After such a base change, both domain and codomain of $\Sigma$ are regular since they are (the base change of) spaces which admit a smooth map to a perfect field. Miracle flatness applies to the restriction of $\Sigma$, where the constant fibre has dimension $2n$. 

To complete our proof that $\Sigma$ is smooth, we are left to prove that the fibre of $\Sigma$ over any $\mathbb{C}$ point of $A\times_{a^\ast \mathpzc{A}}{a^\ast \operatorname{Kum}_{2n}^\Lambda(\mathpzc{A})}$ is smooth\footnote{Since after pullback to a smooth cover which is a scheme we may check upon closed points.}. But for any such fibre there is a semi-abelian variety $A^\circ$ over $\mathbb{C}$ and a tuple of positive integers $n_1,...,n_k$ with sum $n+1$ such that the fibre is isomorphic to the collection of points in $$(A^\circ)^{[n_1]}\times...\times (A^\circ)^{[n_k]}$$ whose sum is $0\in A^\circ$. This locus is smooth by Beauville's original argument for smoothness of generalised Kummer varieties \cite[Section 7, page 769]{Beauville1983}, applied to a semi-abelian variety in place of an abelian variety.
\end{proof}

\subsection{Relatively ample line bundles}
Let $A\rightarrow S$ be a logarithmic abelian variety over $S$ an atomic Noetherian logarithmic scheme with Artin cone $a^\ast \sigma$, and recall that by \cite[Theorem~1.11]{MR3962244} there is a model ${A}_{\Lambda}\rightarrow S$ projective over $S$. We write $\mathpzc{A}_{\Lambda}\rightarrow \sigma$ for the associated tropical abelian variety.

Our next definition is the local version of \cite[Definition 2.9]{MR1844577} which reflects the idea of \cite[Definition 5.5]{Kresch2009} locally in families.

\begin{definition}
    We say that a morphism of DM stacks with log structure $g\colon X\rightarrow S$ is \textit{locally quasi-projective} if there is a smooth cover $c_i\colon U_i\rightarrow S$ such that $U_i \times_S X\rightarrow U_i$ is a \textit{quasi-projective DM stack} in the sense that it has the following properties.
    \begin{enumerate}[leftmargin = *]
        \item $U_i\times_SX$ can be expressed as a global quotient stack of an algebraic space $V_i$ by the action of a subgroup $G$ of $\operatorname{GL}_n/S$ for some integer $n$ where $G$ is flat over $S$. 
        \item $U_i\times_S X$ admits a quasi-projective coarse moduli space.
    \end{enumerate}
    A locally quasi-projective morphism is said to be locally projective if it is also proper.
\end{definition}
\begin{lemma}\label{lem:LocProjBaseChange}
    The property of being a locally (quasi)-projective morphism of DM stacks is preserved under base change. 
\end{lemma}
\begin{proof}
    We establish the quasi-projective claim: the projective case follows because properness is preserved under base change. The claim follows from two facts. First, by \cite[Example 8.1, Main Properties (3) and Theorem 6.6 (the latter is for uniqueness)]{AlperGoodModuli} passing to a coarse space commutes with base change and the base change of a quasi-projective morphism remains quasi-projective by \stacks{0B3G}. Second, being a global quotient stack over $U_i$ is preserved by base change along any map to $U_i$. Flatness of the quotient group scheme, as well as inclusion into the general linear group, are likewise preserved under base change.
\end{proof}
\begin{proposition}\label{prop:Projectivity}
    Let $S$ be a globally charted affine Noetherian logarithmic scheme and let $A\rightarrow S$ be a log abelian variety over $S$. There exist cofinal choices of logarithmic alteration $S'\rightarrow S$ and proper model $A^{[n]}_\Xi\rightarrow A^{[n]}$ which are both saturated and locally projective over $S'$. 
\end{proposition}

We require a preparatory lemma.
\begin{lemma}\label{lem:ProjectiveRefinement}
    Let $X$ be a Noetherian log scheme and assume that $X$ has locally free log structure and an Artin fan which is a cone complex. Let $X'\rightarrow X$ be any log alteration. Then there exists a further log alteration $$X''\rightarrow X'\rightarrow X$$ such that the coarse space of $X''$ is projective over $X$.
\end{lemma}
\begin{proof}
    Since $X$ has locally free log structure and admits an Artin fan which is a cone complex, there is an integer $n$ such that the Artin fan of $X$ can be expressed as a Zariski open in the global quotient stack of $\mathbb{A}^n$ by its dense torus. By \cite[Lemma 4.4.7]{logcoh}, we can find $\hat{X}' = X\times_{\mathcal{A}_X}a^\ast \Sigma'$ a log alteration of $X$ where $q_t\colon a^\ast\Sigma'\rightarrow \mathcal{A}_X$ is an alteration.\footnote{The cited lemma is stated only for log modifications but the argument holds also for alterations.} Nothing is lost in assuming that $X' = \hat{X}'$, since we may put ourselves in this situation by replacing $X'$ with a further log alteration. Thus we can arrange that the morphism $q_t$ is obtained by passing to the global quotient stack to an alteration $q'\colon Y'\rightarrow Y$ of toric varieties.

    Every toric alteration can be factored as a root stack followed by a toric modification. By \cite[Example 8.1, Main Properties (3) and Theorem 6.6 (the latter is for uniqueness)]{AlperGoodModuli} we know that passing to the coarse space of a DM stack commutes with fibre products. Therefore Kummer logarithmic alterations are projective. We are thus free to assume that $q_t$ is a logarithmic modification. In summary we have arranged such that $q_t$ is the global quotient stack of a toric modification $q'$. 

    By \cite[Theorem 6.1.18]{CLSToric} (applied to any completion of our fan) the toric modification $q'$ admits a refinement $ Y''\xrightarrow{q''} Y'\xrightarrow{q'} Y$ such that the composition is projective, and thus $Y''$ possesses a line bundle $\mathcal{L}$ which is relatively ample for the composition. Write $\Sigma''$ for the cone complex associated to the Artin fan of $Y''$. By \cite[Section 3.4]{Fulton+2016}, any line bundle on the toric variety $Y''$ is pulled back from a line bundle on the global quotient stack of $Y''$ by its dense torus, and so $\mathcal{L}$ is pulled back from a line bundle $\mathcal{L}_t$ on $a^\ast \Sigma''$. Since the property that a given line bundle is ample may be checked smooth locally on the target \stacks{0D2P}, pulling back $\mathcal{L}_t$ to $X'' = X\times_{\mathcal{A}_X}a^\ast \Sigma''$ specifies a line bundle which is relatively ample over $X$. Since any log alteration is proper, it follows that $X''\rightarrow X$ is proper and the result is proved. 
%    
%    The morphism $q_t$ is representable in the category of DM stacks with log structure. Since the property that a given line bundle is relatively ample may be checked smooth locally in the source, see \cite{}, we are free to talk about relatively ample line bundles for the morphism $q_t$. Since ampleness can be checked smooth locally and is preserved under arbritary base change \stacks{0893}, if we can arrange for $q'$ to admit a relatively ample line bundle which descends to the quotient stack, then the proof will be complete.
%    
%    Projectivity for morphisms of DM stacks can be checked on the level of coarse spaces by \cite{}. Every toric modification admits a projective refinement by \cite{}.
\end{proof}

\begin{proof}[Proof of Proposition~\ref{prop:Projectivity}]
    In this proof we repeatedly saturate maps: when doing so we follow the procedure of \cite{molcho2019universal} (see \cite[1.3.1(2)]{logcoh} for a summary) and thus saturation of a map $Y\rightarrow S'$ proceeds by replacing $S'$ by a logarithmic alteration $S''\to S'$, and replacing $Y$ by $Y\times_{S'}S''$. We will write $\sigma'$ for the subdivision of $\sigma$ such that $S$ admits a strict map to $\sigma'$. In particular, the saturation of a locally projective morphism remains locally projective because the base change of a locally projective map remains locally projective by Lemma \ref{lem:LocProjBaseChange}. 
    We saturate, and thus use this fact, repeatedly. 
    Although this process modifies $S'$, we suppress the change 
    from the notation and continue to write $S'$ for a logarithmic alteration of $S$ 
    whose isomorphism type is repeatedly updated. In making such a modification we lose the ability to assume that $S$ is affine, but may instead assume $S$ is quasi-projective because it is projective over an affine by Lemma~\ref{lem:ProjectiveRefinement}. We similarly abuse notation and do not change notation when replacing other spaces by logarithmic alterations during saturation steps.
    In particular, it will suffice to construct $A^{[n]}_\Xi$ which is projective, but possibly not saturated, over $S'$. %FALSE: Since projectivity can be checked strict \'etale locally on the target, whenever we saturate a map it is possible to check projectivity on a strict \'etale cover of $S'$. In particular, we can assume that $S'$ is atomic with Artin cone $a^\ast \sigma'$.

    We fix notation which will be used later. Throughout this proof we will discuss quotients of CFGs by the action of a finite group. By this we always mean the stackification of the action groupoid $G \times X \rightrightarrows X$, taken in the strict étale or face topology as dictated by context.

    Since $S$ is globally charted, after saturating ${A}_{\Lambda}\to S'$, the $n$-fold (fs = non fs because our map is saturated) fibre product of ${A}_{\Lambda}$ with itself over $S'$ has Artin fan obtained as the $n$-fold product of $\mathpzc{A}_\Lambda$ over $\sigma'$. We denote this Artin fan $(\mathpzc{A}_{\Lambda})^n$. The collision type of points in $(\mathpzc{A}_{\Lambda})^n$ determines a conical (\cite[Section 1.2.1]{LogQuot}), locally 
    closed stratification of $|(\mathpzc{A}_\Lambda)^n|$. A conical locally closed stratification defines a subdivision of piecewise linear spaces $\mathscr{T}\rightarrow (\mathpzc{A}_\Lambda)^{n}$, see \cite[Proposition~1.2.5]{LogQuot}. Choose a tropical model 
    $\mathpzc{A}_\Lambda^{n} \rightarrow \mathscr{T}$ such that the morphism 
    $\mathpzc{A}_\Lambda^{n} \rightarrow  (\mathpzc{A}_\Lambda)^{n}$ is projective. This is possible: a (possibly not projective) tropical model exists by \cite[Proposition 1.2.5]{LogQuot}, and projective alterations form a co-final system among all choices of logarithmic alteration by Lemma~\ref{lem:ProjectiveRefinement}. 
    
    There is an action of $\operatorname{Sym}(n)$ on $\mathpzc{A}^n$ over $\sigma'$ by permuting terms. Without loss of generality, we may assume that $\mathpzc{A}_\Lambda^{n}\rightarrow \mathpzc{A}^n$ is equivariant for this action. Indeed, given any projective tropical model $\Sigma \to (\mathpzc{A}_\Lambda)^{n}$, take the common refinement of $w^*(\Sigma)$ 
    over $w \in \operatorname{Sym}(n)$. 
    Projectivity is preserved because the base change of a projective map remains projective. After a further saturation step, we may assume $\mathpzc{A}_\Lambda^{n}$ is saturated and projective over $S'$.

    Write $\mathpzc{A}_\Lambda^{(n)}$ for the coarse space of the CFG quotient of 
    $\mathpzc{A}_\Lambda^{n}$ by the action of $\operatorname{Sym}(n)$; this is a cone space (as defined in \cite{ModStckTropCurve}) by construction, and without loss of generality we may take it to be a cone complex (by subdividing further and using the fact that we constructed our subdivision such that collision type, and thus stabiliser group, is constant on each cone).

    The $n$-fold product $({A}_{\Lambda})^n$ of ${A}_{\Lambda} = {A}\times_{a^\ast \mathpzc{A}}a^\ast \mathpzc{A}_{\Lambda}$ with itself over $S'$ is projective over $S'$, being an iterated base change of projective morphisms. The logarithmic alteration
    $$
        {A}_{\Lambda}^n:= ({A}_{\Lambda})^n\times_{a^\ast (\mathpzc{A}_\Lambda)^{n}}a^\ast \mathpzc{A}_\Lambda^{n}   \;\rightarrow\; ({A}_{\Lambda})^n,
    $$
    is the base change of a projective morphism $\mathpzc{A}_\Lambda^{n} \to (\mathpzc{A}_\Lambda)^{n}$, and is thus projective. Since the composition of projective is projective, ${A}_\Lambda^{n}$ is projective over $S$.

    Since automorphisms act by strict morphisms, the CFG quotient of $A_\Lambda^{n}$ by the permutation action of $\operatorname{Sym}(n)$ is represented by 
    a DM stack with logarithmic structure. 
    Write ${A}_\Lambda^{(n)}$ for its coarse space, equipped with the descended 
    logarithmic structure (which descends to the stack quotient by definition, and to the coarse space because whenever the group action fixes a scheme theoretic point it also acts trivially on $M_X$). Since 
    ${A}_\Lambda^{(n)}$ is the coarse space of a finite quotient of a projective 
    DM stack, it is projective over $S'$ due to the following argument. We may assume the polarising ample is $\operatorname{Sym}(n)$ equivariant by tensoring by its pullbacks along automorphisms, and thus it descends to the quotient stack so that we can apply \stacks{0GFB}.
%This must be projective because 

    Note that ${A}_\Lambda^{(n)}$ admits a strict map to $a^\ast \mathpzc{A}_\Lambda^{(n)}$ which itself is a tropical model of $\mathpzc{A}^{[n]}$. Thus ${A}_\Lambda^{(n)}$ can be identified with a subset of $\operatorname{Sym}^n(a^\ast\mathpzc{Z}^{[n]}(\Lambda)\times_{a^\ast \mathpzc{A}}A/\mathpzc{A}^{[n]}\times_{a^\ast \sigma'}S')$. Combining this observation with Theorem \ref{thm:LogHilbProper}, the conventional Hilbert--Chow map specifies a canonical Hilbert--Chow morphism
    $$
        h \colon {A}_\Lambda^{[n]}:=A^{[n]}\times_{a^\ast \mathpzc{A}^{[n]}}a^\ast\mathpzc{A}^{(n)}_\Lambda \;\longrightarrow\; {A}_\Lambda^{(n)}.
    $$
    We already know that ${A}_\Lambda^{(n)}$ is locally projective, so will check that the above morphism is projective. By \cite[Corollaire 9.6.4]{EGA5}, 
    it is enough to exhibit a line bundle on $\overline{A}_\Lambda^{[n]}$ whose restriction to fibres over $\overline{A}_\Lambda^{(n)}$ is ample. The morphism we are required to show projectivity of is precisely the relative Hilbert--Chow morphism associated to the relative Hilbert scheme of points for the morphism $A_{\operatorname{univ}}\rightarrow \mathpzc{A}^{[n]}$. Writing $p\colon Z_n\rightarrow A^{[n]}$ for the universal subscheme and $L$ for the pullback of the ample on $A_\Lambda$ to $Z_n$, the line bundle $\operatorname{det}(p_\ast L)$ is ample on fibres of the above Hilbert Chow map, because it restricts to the Pl\"ucker bundle (with respect to any choice of ample on an alteration of $A$) on fibres of $h$.
\end{proof}

\begin{corollary}\label{corr:projectivityGenKummer}
    Let $A\rightarrow S$ be a logarithmic abelian variety over $S$ an atomic scheme with logarithmic structure. There is a locally projective, vertical, logarithmically smooth, and saturated morphism $$\operatorname{Kum}_{2n}^\Lambda(A)\rightarrow S'$$ which fits into a commutative square with two sides logarithmic alterations and the final side the morphism $\operatorname{Kum}_{2n}(A)\rightarrow S$.
\end{corollary}
\begin{proof}
     Writing $\mathpzc{A}\rightarrow \sigma$ for the associated family of tropical abelian varieties, choose a tropical alteration $\mathpzc{A}^{[n]}_\Lambda \rightarrow \mathpzc{A}^{[n]}$ such that the subsheaf $$\operatorname{Kum}_{2n}^\Lambda(\mathpzc{A}):= \mathpzc{A}^{[n]}_\Lambda\times_{\mathpzc{A}^{[n]}} \operatorname{Kum}_{2n}(\mathpzc{A}) \hookrightarrow \mathpzc{A}^{[n]}_\Lambda$$ is the inclusion of a sub-cone complex. By the cofinality statement of Proposition \ref{prop:Projectivity} we are free to assume that $\mathpzc{A}^{[n]}_\Lambda\times_{\mathpzc{A}^{[n]}} A$ is locally projective over some logarithmic alteration $S'$ of $S$. Thus $\operatorname{Kum}_{2n}^\Lambda(A)$ is a strict closed subscheme of a locally projective, vertical, and saturated logarithmic scheme over $S'$, and is thus locally projective over $S'$. Logarithmic smoothness is Proposition~\ref{prop:ProperLogSmooth}. 
\end{proof}

\subsection{Logarithmic generalised Kummer varieties are logarithmic IHS varieties}

\begin{proposition}\label{prop:TwoFormExist}
    Let $A\rightarrow \Delta_0$ be a principally polarised logarithmic abelian surface, and write $\mathsf{pr}_i\colon A^{n+1}\rightarrow A$ for the $i^\text{th}$ projection map. 
    \begin{enumerate}[leftmargin = *]
        \item There is a unique (up to scaling) section $\eta_A$ of $\omega^2_{A/\Delta_0}$.
        \item The form $\sum_i \mathsf{pr}^\ast_i \eta_A$ on $A^{n+1}$ descends to a 2-form ${\eta}_\circ^{(n+1)}$ on the complement of the large diagonal in $A^{(n+1)}$. Identifying the complement of the large diagonal of $A^{(n+1)}$ with the complement of the Hilbert--Chow exceptional divisor in $A^{[n+1]}$, this form extends uniquely to a global section ${\eta}^{[n+1]}$ of $\omega^2_{A^{[n+1]}/\Delta_0}$.
    \end{enumerate}
   We write $\eta$ for the pullback of $\eta^{[n+1]}$ to $\operatorname{Kum}_{2n}(A)$ under the inclusion $$\operatorname{Kum}_{2n}(A)\hookrightarrow A^{[n+1]}.$$ 
    \begin{enumerate}[start = 3, leftmargin = *]
        \item The form $\eta$ is \textit{nowhere degenerate}, in that it induces an isomorphism
    $$\phi_\eta\colon \mathsf{T}_{\operatorname{Kum}_{2n}(A)/\Delta_0}^\mathsf{log}\rightarrow \omega^1_{\operatorname{Kum}_{2n}(A)/\Delta_0}\quad \quad  v \mapsto \eta(v,-).$$
    \item There is an equality, $$H^0(\operatorname{Kum}_{2n}(A),\omega^2_{\operatorname{Kum}_{2n}(A)/\Delta_0}) = \mathbb{C} \eta.$$
    \end{enumerate}
\end{proposition}

We require a preparatory lemma.

\begin{lemma}\label{lem:ExtendSections}
    Let $X\rightarrow \Delta$ be a logarithmically smooth integral morphism of DM stacks with logarithmic structure.
    \begin{enumerate}[leftmargin = *]
        \item Then both $X$ and $X_0 \cong X\times_\Delta \Delta_0$ are Cohen--Macaulay (CM).
        \item Let $\mathcal{E}$ be any vector bundle on $X_0$ (resp. $X$) and let $j\colon U\hookrightarrow X_0$ be an open subscheme of $X_0$ (resp $X$) whose complement has codimension at least two in $X_0$ (resp. $X$). Then the morphism $\mathcal{E}\rightarrow j_\ast j^\ast \mathcal{E}$ is an isomorphism. In particular, any section of $\mathcal{E}|_{U}$ extends to a section of $\mathcal{E}$.
        \item Let $Y$ be any smooth scheme, and let $\mathcal{E}$ be any vector bundle on $Y$. Let $j'\colon W\subset Y$ be an open subset whose complement has codimension at least two. Then the morphism $\mathcal{E}\rightarrow j'_\ast (j')^\ast \mathcal{E}$ is an isomorphism, and in particular every section of $\mathcal{E}|_W$ extends to a section of $\mathcal{E}$.
    \end{enumerate}
\end{lemma}
\begin{proof}
    Note that $\Delta$ is CM because it is the spectrum of a regular local ring. The morphism from $X$ to $\Delta$ is logarithmically smooth and integral, and is thus a CM map. It follows that $X$ is CM. 
    We know that $X_0$ is CM by the final sentence of \stacks{02JN}. Thus (1) is proved.

    To establish (2) we would like to apply \stacks{0E9I}. Cohen-Macaulay ensures that every point of codimension at least two in $X_0$ also has depth at least two, and thus $\mathcal{E}$ has depth two at every point of the codimension at least two locus $X\backslash U$. All except the final sentence is an \'etale local claim, so the scheme hypothesis does no harm. The final sentence follows by definition of pushforward.

    Claim (3) follows similarly - using the fact that any smooth scheme is CM.
\end{proof}

\begin{proof}[Proof of Proposition \ref{prop:TwoFormExist}(1)]
%    %Throughout the proof we will use the following construction. and $\widetilde{\varpi}\colon \operatorname{Kum}_{2n}(\widetilde{A})\rightarrow \Delta$ for the associated generalised Kummer variety. Whenever we work on a logarithmic alteration of any of these spaces we denote the choice by a subscript $\Xi$ and require induced map to $\Delta$ is saturated, exact, and has domain with locally free logarithmic structure. It follows that all the logarithmic alterations we consider will be smooth over $\mathbb{C}$ in the conventional sense.
%     
     We first prove claim (1). By \cite[Theorem~2.6]{FujisawaNakayama2020} the logarithmic Hodge numbers of $A$ coincide with the Hodge numbers of the conventional abelian surface which is the general fibre of $\widetilde{A}/\Delta$. Appealing to \cite[Proposition 7.27]{EGM-AV} for these conventional Hodge numbers, we deduce that $h^{0,2}_{\log}(A/\Delta_0)=1$ completing the proof.
\end{proof}
    
    In the case that $A$ is a conventional abelian variety, Beauville proves a version of (2) by constructing a two-form on an open subset of $A^{[n+1]}$ with codimension two complement, see \cite[Start of Section 6]{Beauville1983}. He then appeals to a version of Lemma~\ref{lem:ExtendSections}(3). 
    Replicating this argument in the logarithmic setting is inconvenient: in particular requiring the required bound on the codimension of a closed set. We therefore use the existence of a smoothing together with Beauville's original result in the following proof.
    
\begin{proof}[Proof of Proposition~\ref{prop:TwoFormExist} continued.]
    The first sentence of (2) is strict \'etale descent and we will now prove the rest of (2). 
    Our proof requires the following construction. 
    By Proposition~\ref{prop:LogAVSmooth} we can find a logarithmic abelian variety $\widetilde{\pi}\colon \widetilde{A}\rightarrow \Delta$ whose special fibre is $A$. We write $\widetilde{\pi}_{[n+1]}\colon \widetilde{A}^{[n+1]}\rightarrow \Delta$ for the relative Hilbert scheme.
    
    Let $U_2^{[n+1]}$ be the open subset of some logarithmic alteration $\widetilde{A}^{[n+1]}_\Xi\rightarrow \widetilde{A}^{[n+1]}$ which is disjoint from the exceptional divisor of the Hilbert--Chow map and let $U_1^{[n+1]} = \widetilde{\pi}^{-1}(\Delta \backslash\Delta_0)$. 
    The form $\eta_\circ^{[n+1]}|_{U_1^{[n+1]}\cap U_2^{[n+1]}}$ extends from $U_1^{[n+1]}\cap U_2^{[n+1]}$ to $U_1^{[n+1]}$ by the smooth case proved by Beauville in \cite{Beauville1983}. 
    By the argument of the first sentence of the previous paragraph, $\eta_\circ^{[n+1]}|_{U_1^{[n+1]}\cap U_2^{[n+1]}}$ extends to a form on $U_2^{[n+1]}$. 
    By Zariski descent the previous two sentences give us an extension $\eta_{12}$ of $\eta_\circ^{[n+1]}|_{U_2^{[n+1]}}$ to $U_1^{[n+1]}\cup U_2^{[n+1]}$. Note $\omega^2_{A^{[n]}/\Delta_0}$ is locally free because $A^{[n]}$ is log \'etale locally logarithmically smooth by Corollary~\ref{corr:logsmooth}, so without loss of generality $\omega^2_{A^{[n]_\Xi}/\Delta_0}$ is strict \'etale locally smooth. 
    
    We claim that the codimension of the complement of $U_2^{[n+1]}\cap \widetilde{\pi}_{[n]}^{-1}(0)$ is at least one within $\widetilde{\pi}_{[n]}^{-1}(0)$. Indeed, each irreducible component of $\widetilde{\pi}_{[n]}^{-1}(0)$ contains a dense open which is a conventional Hilbert scheme of points on a smooth (but not compact) surface. The complement of the locus where no two points meet is codimension one in any such Hilbert scheme. 
    
    Since $\widetilde{\pi}_{[n]}^{-1}(0)$ is itself codimension one, we deduce that the complement of $U_1^{[n+1]}\cup U_2^{[n+1]}$ is codimension two in the total space of $\widetilde{A}^{[n+1]}_\Xi$.
    It follows from Lemma~\ref{lem:ExtendSections}(2) (total space case) that $\eta_{12}$ extends to a form on $\widetilde{A}^{[n+1]}_\Xi$ and (2) is proved.
\end{proof}
The proof of (3) is a check in coordinates, parallel to Beauville's strategy.

\begin{proof}[Proof of Proposition \ref{prop:TwoFormExist}(3)]
    Choose a tropical model $\widetilde{A}^n_\Xi\rightarrow \widetilde{A}^n$ saturated over $\Delta$ such that there exists a logarithmic alteration $\operatorname{Kum}_{2n}^\Xi(\widetilde{A})\rightarrow \operatorname{Kum}_{2n}(\widetilde{A})$ for which there is a strict map $f\colon \operatorname{Kum}_{2n}^\Xi(\widetilde{A})\rightarrow \widetilde{A}^{(n)}_\Xi$. Since the tropicalisations of $\widetilde{A}^{[n+1]}$ and $\widetilde{A}^{(n+1)}$ coincide, we obtain a logarithmic alteration of $\widetilde{A}^{[n+1]}_\Xi$ of $\widetilde{A}^{[n+1]}$ from this data.  Let $U_1,U_2$ be the preimages of $U_1^{[n+1]},U_2^{[n+1]}$ respectively under $f$.
    We will show that the restriction of $\phi_\eta$ to the union of $U_1$ and $U_2$ is an isomorphism. Since hom sheaves of vector bundles are themselves hom sheaves, the result will follow from Lemma \ref{lem:ExtendSections}(2). 
    
    Restricting to $U_1$ we are in the non-logarithmic situation, and the result is true by \cite{Beauville1983}. To establish the result on $U_2$, note that checking $\phi_\eta$ is an isomorphism can be done on stalks. By \stacks{00DV} it is enough that the morphism pulls back to an isomorphism on the fibres over closed points. It will suffice to check that the pullback to the special fibre $U_2^{0} = \operatorname{Kum}_{2n}^\Xi(A/\Delta_0)\cap U_2$ is an isomorphism.

    Fix a point $p$ in $U_2^0$ whose image in $A^{(n)}$ is $\overline{p}$. Let $q$ in $A^{n+1}$ be a preimage of $\overline{p}$ under the quotient map. Note by construction $q$ is disjoint from the large diagonal. The two form on $A$ is dual to coordinates $x,y$ on the stalk of $T_A$ at the point $\mathsf{pr}_i(q)$. Pulling back along $\mathsf{pr}_i$ for all $i$ yields a natural set of coordinates $x_i,y_i$ for $A^{n+1}$ about $q$. Since $q$ lies in $U_2$, the quotient map is \'etale and these coordinates descend to a set of coordinates $x_i,y_i$ about ${p}$.

     The form $\eta^{(n+1)}$ can now be expressed in the dual basis as $dx_i\wedge dy_i$. We highlight that $d x_i$ and $dy_i$ need not be of the form $df$ for $f$ in $\mathcal{O}_X$ or $M_X$, and will in general be a sum of such terms. Write $B_2$ for the matrix representing $dx \wedge dy$ with respect to the basis $(x,y)$. The restriction of $\eta$ to the tangent bundle of $\operatorname{Kum}_{2n}(A)$ is represented by the restriction of the matrix $\mathsf{Diag}(B_2,....,B_2)$ to the subspace where all even and all odd coordinates add to zero. By elementary linear algebra, the matrix is non-degenerate, and thus neither is $\eta$.
\end{proof}

%If Conjecture \ref{conj:annoying} was a theorem then property (4) would follow immediately by extending $\Delta$ to a curve of finite type over $\mathbb{C}$ and appealing to the following consequence of our conjecture. The logarithmic Hodge numbers would be constant in a logarithmically smooth family, thus forcing $h^{0,2}(\operatorname{Kum}_{2n}(A),\mathbb{C})$ to be constant in families. Since the conjecture is not known, we work around it.

\begin{proof}[Proof of Proposition~\ref{prop:TwoFormExist}(4)]
     We have shown that $C\eta \leq H^0(\operatorname{Kum}_{2n}(A),\omega^2_{\operatorname{Kum}_{2n}(A)/\Delta_0})$ with $\eta$ non-degenerate. It remains to show that $h^0(\operatorname{Kum}_{2n}(A),\omega^2_{\operatorname{Kum}_{2n}(A)/\Delta_0})=1$.
     %any class $\kappa \in H^0(\operatorname{Kum}_{2n}(A),\omega^2_{\operatorname{Kum}_{2n}(A)/\Delta_0})$ can be written as $\lambda \eta$ for some complex number $\lambda$. 
     Indeed, writing $K_{(2n)}(A)$ for the preimage of $0$ under the summation map $A^{(n+1)}\rightarrow A$, there is a restricted Hilbert--Chow morphism $\operatorname{HC}\colon \operatorname{Kum}_{2n}(A)\rightarrow K_{(2n)}(A)$. For every choice of model $K_{(2n)}^\Xi\rightarrow K_{(2n)}$, the morphism of logarithmic alterations $$\operatorname{HC}_{\Xi}\colon \operatorname{Kum}_{2n}^\Xi(A)\rightarrow K_{(2n)}^\Xi(A)$$ is an isomorphism on dense open subsets whose complement in $K_{(2n)}^\Xi$ we denote $W$. Note $W$ is the locus where two or more points meet, and thus has codimension two.  
     Thus we have an injection $$H^0(\operatorname{Kum}_{2n}(A),\omega^2_{\operatorname{Kum}_{2n}(A)/\Delta_0})\hookrightarrow H^0(\operatorname{K}_{(2n)}^\Xi(A)\backslash W,\omega^2_{\operatorname{K}_{(2n)}(A)/\Delta_0})$$ and it will suffice to show that $h^0(K_{(2n)}^\Xi\backslash W,\omega^2_{\operatorname{K}_{(2n)}(A)/\Delta_0})=1$ for some choice of $\Xi$.

     To complete the proof of (4) consider the morphism 
     $$R\colon A^{n}\rightarrow \operatorname{K}_{(2n)}(A)$$ 
     $$(a_1,...,a_n)\mapsto \left[\left(a_1,...,a_n,-\sum_{i=1}^n a_i\right)\right].$$ 
     Note $\operatorname{Sym}(n+1)$ acts on the domain via the embedding $A^n\rightarrow A^{n+1}$ which induces $R$. We can choose $\operatorname{Sym}(n+1)$ equivariant logarithmic alterations with domains which are DM stacks with logarithmic structure, $$A^{n+1}_{\Xi}\rightarrow A^{n+1}\text{ and }A^{n}_{\Xi}\rightarrow A^n\times_{A^{n+1}}A^{n+1}_\Xi.$$ We can assume that the latter is saturated because universal semistable reducition preserves $\operatorname{Sym}(n+1)$ symmetry. The image of the latter map lies within $K^\Xi_{(2n)}$ for the induced choice of $\Xi$. Thus we obtain a morphism of DM stacks with logarithmic structure $$R_\Xi\colon A^{n}_\Xi \rightarrow \operatorname{K}_{(2n)}^\Xi.$$
     The map $R_\Xi$ is a $\operatorname{Sym}(n+1)$ torsor away from $$R_\Xi^{-1}(W) = \left \{ (a_1,...,a_n)|a_i=a_j\text{ some } i,j\right\} \cup \left\{(a_1,...,a_n) \text{ such that }\exists j\text{ with } a_j = -\sum_{i=1}^n a_i\right\}.$$

    We deduce that the restriction of $R_\Xi$ to the complement of $R_\Xi^{-1}(W)$ is the global quotient stack of $A^n_\Xi \backslash R_\Xi^{-1}(W)$ by the free action of the symmetric group on $n+1$ letters. 
    The map $R|_{A^n_\Xi \backslash R_\Xi^{-1}(W)}$ is strict and \'etale so $\omega^2_{\operatorname{K}_{(2n)}(A)/\Delta_0}$ pulls back along the quotient map to $\omega^2_{A^n/\Delta_0}|_{A^n_\Xi \backslash R_\Xi^{-1}(W)}$. 
    The proof of (4) will be complete when we show the final equality in the chain $$h^0(\operatorname{K}_{(2n)}(A),\omega^2_{\operatorname{K}_{(2n)}(A)/\Delta_0}) = \operatorname{dim}(H^0(A^n_\Xi \backslash R_\Xi^{-1}(W),\omega^2_{A^n/\Delta_0}|_{A^n_\Xi \backslash R_\Xi^{-1}(W)})^{\operatorname{Sym}(n+1)}) = 1$$
    
    Since $W$ is codimension at least two and $R_\Xi$ is finite, the inclusion of $A^n_\Xi \backslash R_\Xi^{-1}(W)\hookrightarrow A^n_\Xi$ is an isomorphism outside a locus of codimension at least two. 
    Lemma \ref{lem:ExtendSections}(2) applies and we deduce that pullback induces an isomorphism which respects the symmetric group action, $$H^0(A^n\backslash R_\Xi^{-1}(W),\omega^2_{A^n/\Delta_0}|_{A^n_\Xi \backslash R_\Xi^{-1}(W)}) \cong H^0(A^n,\omega^2_{A^n/\Delta_0}).$$ It remains to show that the vector bundle $\omega^2_{A^n/\Delta_0}$ on $A^n$ admits a single dimension of symmetric group invariant sections. 
    
    Letting $\widetilde{A}/\Delta$ be a logarithmic abelian variety with central fibre $A/\Delta_0$, note that by the main results of \cite{KajiwaraKatoNakayamaPartIV} $\widetilde{A}$ admits a logarithmic alteration which is a scheme with logarithmic structure. It follows from \cite{FujisawaNakayama2020} that the logarithmic Hodge theory of $A$ specifies a variation of logarithmic Hodge structure on $\Delta$. Thus the logarithmic Hodge numbers of any logarithmic abelian surface coincide with the logarithmic Hodge numbers of a smooth abelian surface, and we deduce that $h^0(A,\omega^1) =2.$ Choose generators $\langle x^a,x^b\rangle \cong H^0(A,\omega^1)$ for this group. Since $\omega^1$ is additive over direct products, we obtain a basis $$H^0(\omega^1,A^n) = \langle x_1^a = p_1^\ast x^a,x_1^b = p_1^\ast x^b,..., x_n^b = p_n^\ast x^b \rangle.$$
    Note that by \cite[Chapter VI, Theorem 1.1(3)]{FaltingsChai1990} $\omega^1$ is the trivial rank two vector bundle: thus taking wedge power commutes with global sections. 
    
    Thus $H^0(\omega^2,A^n)$ has dimension $2n$ choose $2$. We now describe its structure as a $\operatorname{Sym}(n+1)$ representation. Following \cite[Proposition 3.12]{FultonHarris}, the canonical representation of $\operatorname{Sym}(n+1)$ on a vector space with basis indexed by $1,...,n+1$ splits as the direct sum $\operatorname{trivial}\oplus \operatorname{Std}$ where $\operatorname{Std}$ is irreducible by \cite[Proposition 3.12]{FultonHarris}. Certainly there is an isomorphism of representations $$H^0(\omega^2,A^n)\cong \left(\operatorname{Std}\oplus \operatorname{Std}\right) \bigwedge \left(\operatorname{Std}\oplus \operatorname{Std}\right)\cong \left(\operatorname{Std}\wedge\operatorname{Std}\right)^{\oplus 2}\oplus \operatorname{Std}\otimes \operatorname{Std}.$$ By \cite[Proposition 3.12]{FultonHarris} $\operatorname{Std}\wedge\operatorname{Std}$ is irreducible but not the trivial irreducible representation. By \cite[Comment end of section 3.2]{FultonHarris} $\operatorname{Std}\otimes \operatorname{Std}$ has at least a one dimensional space of invariants, and by \cite[Proposition 2.1]{FultonHarris} $\operatorname{Std}$ being a real irreducible representation is self dual so in fact the $G$-invariant part of $\operatorname{Std}\otimes \operatorname{Std}\cong \operatorname{Std}\otimes \operatorname{Std}^\vee \cong \operatorname{Hom}(\operatorname{Std},\operatorname{Std})$ has dimension at most one by Schur's lemma completing the proof. 
%     Moreover we know that $h^0(\operatorname{Kum}_{2n}(A),\omega^2_{\operatorname{Kum}_{2n}(A)/\Delta_0})=1$ because this dimension is constant in log smooth families and earlier in this proof we noted that $A$ admits a smoothing with the total space logarithmically smooth, from which it follows that $\operatorname{Kum}_{2n}(A)$ admits a smoothing with the total space logarithmically smooth.
\end{proof}

\begin{lemma}\label{lem:LogSimplyConnected}
    Let $A\rightarrow \Delta_0$ be a principally polarised logarithmic abelian surface. There are isomorphisms
        $$\pi_1(\mathsf{KN}_0(\operatorname{Kum}_{2n}(A/\Delta_0)))\cong \operatorname{ker}(\pi_1^{\operatorname{log}}(\operatorname{Kum}_{2n}(A/\Delta_0))\rightarrow \pi_1^{\operatorname{log}}(\Delta_0))\cong 0.$$
\end{lemma}
\begin{proof}
    The first isomorphism is a definition. For the second isomorphism note that by Proposition~\ref{prop:LogAVSmooth} and Proposition~\ref{prop:ProperLogSmooth} there is a logarithmically smooth morphism $$\operatorname{Kum}_{2n}(\widetilde{A})\rightarrow \Delta$$ which has central fibre which is a logarithmic alteration of the generalised Kummer variety in the lemma's statement. The general fibre has trivial fundamental group by \cite{Beauville1983}, and thus by Lemma~\ref{lem:Serre Fibration} we deduce that after replacing $A/\Delta_0$ by a kummer logarithmic alteration pulled back from $\Delta$ we have $\pi_1(\mathsf{KN}_0(\operatorname{Kum}_{2n}(A/\Delta_0)))\cong 0$. But the statement descends along a kummer logarithmic alteration pulled back from $\Delta$: such an alteration only replaces $\operatorname{KN}(\Delta_0)\cong S^1$ by a finite covering space. Pulling back along a covering space does not affect homotopy groups of the fibres.
\end{proof}
\begin{corollary}\label{corr:familyIHS}
    Let $A\rightarrow S$ be a principally polarised family of logarithmic abelian surfaces. The associated generalised Kummer variety $\pi\colon \operatorname{Kum}_{2n}(A)\rightarrow S$ is a family of logarithmic IHS varieties.
\end{corollary}
\begin{proof}
    We verify that the hypotheses of Definition~\ref{defn:LogIHS} hold. Proposition~\ref{prop:ProperLogSmooth} ensures that $\pi$ is logarithmically smooth and proper over $S$. Relatively simply connected is Lemma~\ref{lem:LogSimplyConnected}. The logarithmic symplectic structure is recorded in Proposition~\ref{prop:TwoFormExist}. Corollary \ref{corr:projectivityGenKummer} ensures after base change to any standard log point there is a projective logarithmic modification.
\end{proof}

\section{Irreducible components of logarithmic generalised Kummer varieties}

The purpose of this section is to study the following situation.

\begin{situation}\label{sitch}
Set $A\rightarrow \Delta$ be a family of logarithmic abelian surfaces, and write $\mathcal{X}' = \operatorname{Kum}_{2n}({A})/\Delta$ for the associated logarithmic generalised Kummer variety. Choosing a tropical model with domain a free cone complex and saturating the resulting morphism \cite{AbramovichKaru,molcho2019universal} specifies a saturated map $\pi \colon \mathcal{X}/\Delta$ of DM stacks with logarithmic structure. There is a corresponding saturated morphism of cone complexes $\operatorname{Kum}_{2n}^\Lambda(\mathpzc{A}) \rightarrow \rho$ where $\rho$ is a ray, $a^\ast \rho$ is the Artin fan of $\Delta$, and $\operatorname{Kum}_{2n}^\Lambda(\mathpzc{A}) \rightarrow \operatorname{Kum}_{2n}(\mathpzc{A})$ is a tropical model. 
\end{situation}

%We will describe irreducible components of $\mathcal{X}_0= \mathcal{X}\times_\Delta \Delta_0$. %Subject to Conjecture \ref{conj:annoying} we also compute the associated logarithmic variation of Hodge structure.

This section will describe the irreducible components of $\mathcal{X}_0 = \mathcal{X}\times_{\Delta}\Delta_0$ for $\mathcal{X}$ as defined in Situation~\ref{sitch}.

\subsection{Developing language to describe components of the special fibre}\label{sec:LogHilbCompactSAV} Recall that the \textit{logarithmic multiplicative group $\mathbb{G}_m^\text{log}$} is defined through its functor of points on the category of schemes with logarithmic structure. This functor of points sends a logarithmic scheme $T$ to $\Gamma(T, M_T^\text{gp}).$ The inclusion of $\mathbb{G}_m(T)$ into $M_T^\text{gp}$ induces a monomorphism $\mathbb{G}_m\hookrightarrow \mathbb{G}_m^\text{log}$.

\begin{lemma}
    The cokernel of the map $\mathbb{G}_m\hookrightarrow \mathbb{G}_m^\text{log}$ is the sheaf $\mathbb{G}_m^\text{trop}:= a^\ast \mathbb{Z}$.
\end{lemma}

\begin{proof}
    The cokernel has functor of points obtained by sheafifying the presheaf which sends a test scheme $T$ to $\Gamma(T,\overline{M}_T^{\operatorname{gp}})$. On the other hand, the functor of points of $a^\ast \mathbb{Z}$ is obtained by sheafifying the presheaf which sends a test scheme $T$ to the collection of maps from $\operatorname{Hom}(\overline{M}_T,\mathbb{N})$ to $\mathbb{Z}$. There is an isomorphism of monoids $$\overline{M}_T^{\operatorname{gp}}\rightarrow \operatorname{Hom}(\operatorname{Hom}(\overline{M}_T,\mathbb{N}),\mathbb{Z}).$$ To see that this map is an isomorphism note first that by duality for fine and saturated torsion-free monoids $\overline{M}_T\cong \operatorname{Hom}(\operatorname{Hom}(\overline{M}_T,\mathbb{N}),\mathbb{N})$; and note second that for any toric monoid $A$ there is an equality $\operatorname{Hom}(A,\mathbb{Z}) = \operatorname{Hom}(A,\mathbb{N})^{\operatorname{gp}}$.
\end{proof}

We expect, but do not prove, that our next definition coincides with the pushout construction in \cite[Section 2.1]{KajiwaraKatoNakayama2008a}. The definition presented here is a logarithmic analogue of the main object of study in \cite{MR4774210}. 

\begin{definition}\label{defn:Alog}
    Let $A^\circ$ be a semi-abelian variety, and write $\mathbb{G}_m^s$ for the 
    \textit{torus part of $A^{\circ}$}, i.e. the kernel of the map from $A^\circ$ to its Albanese. We write 
    $$A^\text{log} = [A^\circ\times (\mathbb{G}^\text{log}_m)^s/\mathbb{G}_m^s]$$
    where $\mathbb{G}_m^s$ acts diagonally as the subtorus of the second factor, and as the torus part of $A^\circ$.
\end{definition}

Note that the toric variety bundles of \cite{MR4774210} are logarithmic alterations of $A^{\operatorname{log}}$. Being the quotient of an abelian group by a subgroup, $A^\text{log}$ carries a group structure. 

\begin{definition}
    The \textit{logarithmic Hilbert scheme} $(A^{\log})^{[m]}$ of $m$ points on $A^\text{log}$ is defined through its functor of points on the category of fine and saturated logarithmic schemes, which assigns to a test scheme $S$ the set of equivalence classes of pairs $$(Y'\rightarrow S\times A^\mathsf{log}, Z'\hookrightarrow Y'),$$ 
    where:
    \begin{enumerate}[leftmargin = *]
        \item The scheme with logarithmic structure $Y'$ is a logarithmic alteration of $S\times A^\mathsf{log}$ which is saturated over $S$. 
        \item The morphism $Z'\hookrightarrow Y'$ is a strict closed immersion of Hilbert polynomial $m$ which renders $Z'$ logarithmically flat over $S$.
        \item The equivalence relation identifies $(Y',Z'\hookrightarrow Y')$ with $(Y'',Z''\hookrightarrow Y'')$ if the pullbacks of $Z'$ and $Z''$ to $Y''\times_{S\times A^\mathsf{log}} Y'$ coincide.
    \end{enumerate}
\end{definition}

Our next theorem refers to tropical Hilbert schemes of points on $\mathbb{Z}^s$. We thus continue notation from Section~\ref{sec:DescribeVertexStar}.  

\begin{theorem}\label{thm:ComponentlogHilb}
    There is a strict morphism represented by DM stacks with logarithmic structure $$(A^{\operatorname{log}})^{[m]}\rightarrow a^\ast (\mathbb{Z}^s)^{[m]}.$$
    Moreover for any choice of tropical model $$ (\mathbb{Z}^s)^{[m]}(\Lambda)\rightarrow (\mathbb{Z}^s)^{[m]},$$ there is a canonical open immersion into a (conventional) relative Hilbert scheme, $$(A^{\operatorname{log}})^{[m]}(\Lambda) \rightarrow {\operatorname{Hilb}}^m(A^{\operatorname{log}} \times_{a^\ast\mathbb{Z}^s}a^\ast \mathpzc{Z}_{m}(\mathbb{Z}^{s})(\Lambda) /a^\ast (\mathbb{Z}^s)^{[m]}(\Lambda)).$$
\end{theorem}
\begin{proof}
    The arguments of \cite{LogQuot} apply as written to this more general situation. 
\end{proof}

Let $\textbf{n} = (n_1,...,n_k)$ be a tuple of positive integers summing to $n+1$. As in Lemma \ref{lem:LogHCGpMap}, the group structure induces for each $j$ a morphism $(A^{\operatorname{log}})^{[n_j]}\rightarrow A^{\operatorname{log}},$ and thus a morphism 
$$m\colon \prod_{i=1}^k (A^{\operatorname{log}})^{[n_i]}\rightarrow A^{\operatorname{log}}.$$
\begin{definition}
    Writing $\iota_0$ for the inclusion of the identity into the group $A^\circ$ (which we identify with a subgroup of $A^{\log}$), we define two further variants of the logarithmic Kummer variety: $$
\begin{tikzcd}
\operatorname{Kum}_{\mathbf{n}(v)}(A^\circ)
  \arrow[r]
  \arrow[d]
&
\displaystyle \prod_{n_i \in \mathbf{n}(v)} (A^\circ)^{[n_i]}
  \arrow[d,"m"]
\\
\operatorname{Spec}(\mathbb{C})
  \arrow[r,"\iota_0"]
&
A^\circ
\end{tikzcd}
\quad \quad
\begin{tikzcd}
\operatorname{Kum}_{\mathbf{n}(v)}(A^{\log})
  \arrow[r]
  \arrow[d]
&
\displaystyle \prod_{n_i \in \mathbf{n}(v)} (A^{\log})^{[n_i]}
  \arrow[d,"m"]
\\
\operatorname{Spec}(\mathbb{C})
  \arrow[r,"\iota_0"]
&
A^{\log}
\end{tikzcd}.
$$
\end{definition}

\begin{definition}
    Let $\Sigma$ be a cone complex and let $\kappa$ be a cone in $\Sigma$. Consider a morphism of algebraic stacks $$X\rightarrow a^\ast \Sigma.$$ Note that there is an open immersion of the stack quotient of the toric variety associated to $\kappa$ by its dense torus into $a^\ast \Sigma$. Write $W_\kappa$ for the image of the closed torus orbit under the stack quotient followed by open immersion. The \textit{stratum of $X$ associated to $\kappa$} is $X\times_{a^\ast \Sigma}W_\kappa$.
\end{definition}

\begin{proposition}
    \label{prop:KummerStrataIrreducible}Let $$\operatorname{Kum}^\Gamma_{\textbf{n}(v)}(\mathbb{Z}^s)\rightarrow \operatorname{Kum}_{\textbf{n}(v)}(\mathbb{Z}^s)$$ be any tropical model. Define $$\operatorname{Kum}^{\Gamma}_{\textbf{n}(v)}(A^{\operatorname{log}}) = \operatorname{Kum}^\Gamma_{\textbf{n}(v)}(\mathbb{Z}^s)\times_{ \operatorname{Kum}_{\textbf{n}(v)}(\mathbb{Z}^s)} \operatorname{Kum}_{\textbf{n}(v)}(A^{\operatorname{log}}).$$
    \begin{enumerate}[leftmargin =*]
        \item The stack $\operatorname{Kum}_{\mathbf{n}(v)}(A^\circ)$ is represented by an irreducible smooth scheme of dimension $2n$.
        \item The canonical inclusion $$\operatorname{Kum}_{\mathbf{n}(v)}(A^\circ)\hookrightarrow \operatorname{Kum}^{\Gamma}_{\textbf{n}(v)}(A^{\operatorname{log}})$$ is the inclusion of a dense open into a logarithmically smooth DM stack with logarithmic structure.
        \item For any cone $\kappa$ in the cone complex $\operatorname{Kum}^\Gamma_{\textbf{n}(v)}(\mathbb{Z}^s)$, the stratum of $\operatorname{Kum}^{\Gamma}_{\textbf{n}(v)}(A^{\operatorname{log}})$ associated to $\kappa$ is irreducible.
    \end{enumerate}
\end{proposition}

\begin{proof}
    We first prove (1). 
    $(A^\circ)^{[n_i]}$ is a conventional Hilbert scheme of points on a smooth surface and is thus a smooth scheme. The morphism $m$ is smooth by the argument of \cite[Second paragraph of Section 7]{Beauville1983} - which applies to semi-abelian varieties just as well as to abelian varieties. Since smoothness is preserved under base change and since also the category of schemes is closed under fibre products, $\operatorname{Kum}_{\mathbf{n}(v)}(A^\circ)$ is a smooth scheme. 
    A connected DM stack which is smooth over a field is irreducible by \cite[Lemma 4.5]{shi2026gcompanionsalgebraicstacksapplications}. 
    It remains to argue that $\operatorname{Kum}_{\mathbf{n}(v)}(A^\circ)$ is connected. 
    Indeed, let $$K_{(\textbf{n}(v))}(A^\circ) = \left(\prod_i\operatorname{Sym}^{n_i}(A^\circ)\right)\times_{A^\circ} e$$ where the product is over the product map and the inclusion $e\rightarrow A^\circ$ of the identity.
    The Hilbert--Chow map restricts to define a universally closed morphism $\operatorname{Kum}_{\mathbf{n}(v)}(A^\circ)\rightarrow K_{\mathbf{n}(v)}(A^\circ)$, and this morphism is surjective. Since fibres of the Hilbert--Chow map are connected, by \stacks{0377} it suffices to show that $K_{\mathbf{n}(v)}$ is connected. 
    But $(A^\circ)^n$ is connected, and there is a surjective map $(A^\circ)^n\rightarrow K_{\mathbf{n}(v)}$ sending $(a_1,...,a_n)$ to $(a_1,...,a_n, -\sum_ia_i)$. 
    Since the image of a connected set is connected, the result is proved.

    We now prove (2). Note that $\operatorname{Kum}^{\Gamma}_{\textbf{n}(v)}(A^{\operatorname{log}})$ is a DM stack with logarithmic structure by argument of Proposition \ref{prop:ProperLogSmooth} - i.e. it is an open subset in the conventional relative Hilbert scheme of Theorem \ref{thm:ComponentlogHilb}. Logarithmic smoothness is also similar to Proposition \ref{prop:ProperLogSmooth}. Logarithmic smoothness ensures that the map from $\operatorname{Kum}^{\Gamma}_{\textbf{n}(v)}(A^{\operatorname{log}})$ to $a^\ast\operatorname{Kum}^{\Gamma}_{\textbf{n}(v)}(\mathbb{Z}^s)$ is smooth, and thus flat of finite presentation. A flat morphism of finite presentation is open, and thus the preimage of any dense set in $a^\ast\operatorname{Kum}^{\Gamma}_{\textbf{n}(v)}(\mathbb{Z}^s)$ is dense in $\operatorname{Kum}^{\Gamma}_{\textbf{n}(v)}(A^{\operatorname{log}})$. But $\operatorname{Kum}_{\mathbf{n}(v)}(A^\circ)$ is the preimage of the dense set $\operatorname{Spec}(\mathbb{C})\subset a^\ast\operatorname{Kum}^{\Gamma}_{\textbf{n}(v)}(\mathbb{Z}^s)$. Thus $\operatorname{Kum}_{\mathbf{n}(v)}(A^\circ)$ is dense in $\operatorname{Kum}^{\Gamma}_{\textbf{n}(v)}(A^{\operatorname{log}})$.

    Finally we prove (3). By (2) we know that $\operatorname{Kum}^{\Gamma}_{\textbf{n}(v)}(A^{\operatorname{log}})$ is logarithmically smooth. Therefore the stratum of $\operatorname{Kum}^{\Gamma}_{\textbf{n}(v)}(A^{\operatorname{log}})$ associated to $\kappa$, being a fibre over the (smooth) map to the Artin fan, is smooth. Again by \cite[Lemma 4.5]{shi2026gcompanionsalgebraicstacksapplications}, it will suffice to check that the stratum associated to $\kappa$ is connected. 
    
    Writing $\textbf{n}_\kappa(v)$ for the collision type specified by any point in the interior of $\kappa$, the stratum associated to $\kappa$ is the stack quotient of $\operatorname{Kum}_{\textbf{n}_\kappa(v)}(A^\circ)$ by the action of a group. But we already argued that schemes of the form $\operatorname{Kum}_{\textbf{n}_\kappa(v)}(A^\circ)$ are connected, and the image of a connected set is connected. The result is proved.
\end{proof}

%\begin{lemma}
%There is a strict map $$\operatorname{Kum}_{\mathbf{n}(v)}(A^{\log})\rightarrow a^\ast\left(0\times_{\mathbb{Z}^s}\prod_{n_i\in \textbf{n}}(\mathbb{Z}^s)^{[n_i]}\right)=:a^\ast\operatorname{Kum}_{\textbf{n}(v)}(\mathbb{Z}^s).$$
%\end{lemma}
%\begin{proof}
%    
%\end{proof}
\subsection{Irreducible components of the generalised Kummer locus}

We continue notation from Situation \ref{sitch}, and further declare $A^\circ$ for the semi-abelian part of the logarithmic abelian variety $A\times_\Delta \Delta_0$.
The tropical model $\operatorname{Kum}_{2n}^\Lambda(\mathpzc{A}) \rightarrow \operatorname{Kum}_{2n}(\mathpzc{A})$ specifies a polyhedral complex structure $B_\Lambda$ upon the integral affine manifold $B_K^{2n}(\mathpzc{A})$ which was associated to $\mathpzc{A}(\varphi)$ in Section \ref{sec:TropGenKummerIAMSing}.

The toric dictionary ensures that irreducible components of the relative Artin fan $$E:=a^\ast \operatorname{Kum}_{2n}^\Lambda(\mathpzc{A}) \times_{a^\ast \rho} \Delta_0 $$ biject with vertices of $B_\Lambda$. Write $Y_v$ for the irreducible component associated to vertex $v$. 

\begin{definition}
For $v$ any vertex of $B_\Lambda$ let $\mathcal{X}_v^0$ be the fibre product, 
$$
\begin{tikzcd}
\mathcal{X}_v^0 \arrow[r] \arrow[d] & Y_v \arrow[d, hook] \\
\mathcal{X}_0 \arrow[r] & E.
\end{tikzcd}
$$
\end{definition}
The stratification of $\operatorname{Sym}^{n+1}(B(\mathpzc{A}))$ by collision type assigns to $v$ an unordered tuple $\textbf{n}(v) = (n_1,...,n_k)$ of positive integers which sum to $n+1$.
To state our next theorem we adopt the following notation: for $W$ any logarithmic scheme we will denote the underlying scheme without logarithmic structure of $W$ by $\underline{W}$.

\begin{theorem}\label{prop:ExplicitDescription} We establish the following statements.
\begin{enumerate}[leftmargin = *]
    \item The morphism $$\underline{\mathcal{X}_v^0}\hookrightarrow \underline{\mathcal{X}_0}$$ is the inclusion of an irreducible component. 
    \item For some choice of tropical model 
    $\operatorname{Kum}^\Gamma_{\textbf{n}(v)}(\mathbb{Z}^s)\rightarrow \operatorname{Kum}_{\textbf{n}(v)}(\mathbb{Z}^s)$ there is an isomorphism in the category of algebraic stacks between $\underline{\mathcal{X}_v^0}$ and $\underline{\operatorname{Kum}_{\textbf{n}(v)}^\Gamma(A^{\operatorname{log}})}$.
    \item There is an isomorphism in the category of algebraic stacks between $\underline{\operatorname{Kum}_{\textbf{n}(v)}(A^\circ)}$ and the interior\footnote{That is, the open subset disjoint from all other components.} of $\underline{\mathcal{X}_v^0}$. 
\end{enumerate}
\end{theorem}

\begin{proof}[Proof assuming claim (2)]
    Note (3) follows directly from (2) by restricting attention to the open where the logarithmic structure has minimal rank. Next we show that (1) also follows once claim (2) is established. Indeed, first observe that $\underline{\mathcal{X}_v^0}$, being the preimage of an irreducible component under the tropicalisation map, is closed in $\underline{\mathcal{X}_0}$. Also $\underline{\mathcal{X}_0}$ is the union over $v$ of the $\underline{\mathcal{X}_v^0}$. By combining (2) with Proposition \ref{prop:KummerStrataIrreducible}, we learn that each $\underline{\mathcal{X}_v^0}$ is irreducible of dimension $2n$, so no one $\underline{\mathcal{X}_v^0}$ is properly contained in another. It remains to prove claim (2).
\end{proof}
    Following Lemma \ref{lem:TropicalByComponent}, the tropical model $\operatorname{Kum}_{2n}^\Lambda(\mathpzc{A}) \rightarrow \operatorname{Kum}_{2n}(\mathpzc{A})$ induces a model $$\operatorname{Kum}^\Gamma_{\textbf{n}(v)}(\mathbb{Z}^s)\rightarrow \operatorname{Kum}_{\textbf{n}(v)}(\mathbb{Z}^s)$$ where $\operatorname{Kum}^\Gamma_{\textbf{n}(v)}(\mathbb{Z}^s)$ is isomorphic to the star of $\rho_v$ in $\operatorname{Kum}_{2n}^\Lambda(\mathpzc{A})$. We will write $A_0 = A\times_\Delta \Delta_0$. 
    
    Since the cone complex $\operatorname{Kum}_{2n}^\Lambda(\mathpzc{A})$ is smooth, there is a piecewise linear function (in the sense of \cite[Remark 7.3]{ModStckTropCurve}) which adopts the value $1$ on the generator of $\rho_v$ and $0$ on all other rays. Let $\overline{Y}_v$ be the effective Cartier divisor associated to this function. Since closed immersions are preserved under base change, we have a closed immersion of algebraic stacks$$\underline{Y_v}=\underline{\Delta_0\times_{a^\ast \rho} \overline{Y}_v}\hookrightarrow \underline{\Delta_0\times_{a^\ast \rho} a^\ast \operatorname{Kum}_{2n}^\Lambda(\mathpzc{A})}.$$
    We will use the saturation hypothesis of Situation \ref{sitch} in our next lemma:
    \begin{lemma}\label{lem:StarAsComponent}
        There is an isomorphism in the category of algebraic stacks $$\underline{Y_v}\cong \underline{a^\ast \operatorname{Kum}^\Gamma_{\textbf{n}(v)}(\mathbb{Z}^s)}.$$
    \end{lemma}
    \begin{proof}
        By Lemma \ref{lem:TropicalByComponent} it will suffice to identify $\underline{Y_v}$ with $\underline{a^\ast\operatorname{Star}_{\rho_v}( \operatorname{Kum}^\Lambda_{2n}(\mathpzc{A}))}$.
        Note $\underline{Y_v}$ has a Zariski cover $U_\sigma \rightarrow \underline{Y_v}$ indexed by cones $\sigma$ of $\operatorname{Kum}_{2n}^\Lambda(\mathpzc{A})$ which contain $\sigma$ as a face. Indeed, this cover is pulled back from the open cover of $\operatorname{Kum}_{2n}^\Lambda(\mathpzc{A})$ indexed by the toric dictionary. By definition of star,  $\underline{a^\ast\operatorname{Star}_{\rho_v}( \operatorname{Kum}^\Lambda_{2n}(\mathpzc{A}))}$ has a Zariski open cover $$V_\sigma \rightarrow \underline{a^\ast\operatorname{Star}_{\rho_v}( \operatorname{Kum}^\Lambda_{2n}(\mathpzc{A}))}$$ with the same indexing set. 

        The proof will be complete if we provide for each $\sigma$ an isomorphism $V_\sigma\cong U_\sigma$ in a fashion compatible with face maps.
        We fix $\sigma$ and construct the desired isomorphism. Let $X_\sigma$ be the toric variety associated to $\sigma$ so that $a^\ast \sigma = [X_\sigma/X_\sigma^\circ]$. 
        Write $h\colon \sigma^{\operatorname{gp}}\rightarrow \mathbb{Z} = \rho^{\operatorname{gp}}$ for the lattice map induced by $\sigma \rightarrow \rho.$ Letting $u_v$ be the primitive generator of $\rho_v$, saturation means that projection induces an isomorphism
$$
g\colon \operatorname{ker}(h)\rightarrow \sigma^{\operatorname{gp}}/\mathbb Zu_v.$$

Let $D_{v,\sigma}\subset X_\sigma$ be the toric divisor associated to $\rho_v$. The restriction of $ \underline{\overline Y_v}$ to the Zariski open $a^\ast \sigma$ of $a^\ast \operatorname{Kum}_{2n}^\Lambda(\mathpzc{A})$ is $[D_{v,\sigma}/X^\circ_\sigma]$. 
Its map to $a^\ast\rho=[\mathbb A^1/\mathbb G_m]$ factors through the inclusion of the closed point $B\mathbb G_m$, and thus there is an isomorphism
$$
U_\sigma = [D_{v,\sigma}/X_\sigma^\circ]\times_{B\mathbb G_m}\operatorname{Spec}\mathbb C
\cong [D_{v,\sigma}/T_{\ker(h)}].
$$
But $D_{v,\sigma}$ is an affine toric variety whose dense torus we denote $D_{v,\sigma}^\circ$. By definition of star, $V_\sigma$ is $[D_{v,\sigma}/D_{v,\sigma}^\circ]$. The lattice isomorphism $g$ now provides the second isomorphism in the identification:
$$ V_\sigma \cong  [D_{v,\sigma}/D_{v,\sigma}^\circ] \cong [D_{v,\sigma}/T_{\ker(h)}]\cong U_\sigma.$$
\end{proof}
    Lemma \ref{lem:StarAsComponent} specifies a closed immersion
    $$\iota \colon \underline{a^\ast \operatorname{Kum}^\Gamma_{\textbf{n}(v)}(\mathbb{Z}^s)}\rightarrow \underline{\Delta_0\times_{a^\ast \rho} a^\ast \operatorname{Kum}_{2n}^\Lambda(\mathpzc{A})}.$$ 

    \begin{proof}[Proof of (2)]
    Under $\iota$ the morphism (which we think of as the universal expansion) 
    $$\pi_1\colon \underline{A_0\times_{a^\ast \mathpzc{A}} a^\ast(\mathpzc{Z}_{K}({\mathpzc{A}})\times_{\operatorname{Kum}_{2n}(\mathpzc{A})}\operatorname{Kum}_{2n}^\Lambda(\mathpzc{A}))}\rightarrow \underline{\Delta_0\times_{a^\ast \rho} a^\ast \operatorname{Kum}_{2n}^\Lambda(\mathpzc{A})}$$ pulls back to 
    $$\pi_2\colon \underline{A_0 \times_{a^\ast \mathpzc{A}}a^\ast(\mathpzc{Z}_{\mathbf{n}(v)}\times_{\operatorname{Kum}_{\textbf{n}(v)}(\mathbb{Z}^s)}\operatorname{Kum}^\Gamma_{\textbf{n}(v)}(\mathbb{Z}^s))}\rightarrow \underline{a^\ast \operatorname{Kum}^\Gamma_{\textbf{n}(v)}(\mathbb{Z}^s)}$$ by Lemma \ref{lem:UnivOverComponent}. 
    
    In particular $\operatorname{Hilb}^{[n+1]}(\pi_1)$ pulls back to $\operatorname{Hilb}^{[n+1]}(\pi_2)$ along $\iota$. Moreover, tracking the Hilbert polynomial of the subscheme in each stratum shows that the locally closed subset $A^{[n+1]}_\Lambda\times_{\mathpzc{A}^{[n+1]}\times_{a^\ast \rho} \Delta} \underline{Y_v}$ of $\operatorname{Hilb}^{[n+1]}(\pi_1)$ where a stratum associated to a cone $\kappa$ of weight $w(\kappa) = n_i$ contains a length $n_i$ subscheme (cf. Theorem~\ref{thm:LogHilbProper}) pulls back to an open subset of $\operatorname{Hilb}^{[n+1]}(\pi_2)$ isomorphic to
    $$(A^{\operatorname{log}})^{[\textbf{n}(v)]}_{a^\ast \operatorname{Kum}_\Gamma}:=\left(\prod_i(A^{\operatorname{log}})^{[n_i]}\right)\times_{a^\ast \operatorname{Kum}_{\textbf{n}(v)}(\mathbb{Z}^s)}a^\ast\operatorname{Kum}^\Gamma_{\textbf{n}(v)}(\mathbb{Z}^s)$$ 
    (cf. Theorem~\ref{thm:ComponentlogHilb}). 
    
    We have thus specified a morphism of algebraic stacks $$t\colon \underline{(A^{\operatorname{log}})^{[\textbf{n}(v)]}_{a^\ast \operatorname{Kum}_\Gamma}} \longrightarrow \underline{{A^{[n+1]}\times_{\mathpzc{A}^{[n+1]}\times_{^\ast \rho} \Delta} E}}.$$
    To complete the proof we let $$(A^{\circ})^{(\textbf{n}(v))}_{a^\ast \operatorname{Kum}_\Gamma}:=\left(\prod_i(A^\circ)^{(n_i)}\right)\times_{a^\ast \operatorname{Kum}_{\textbf{n}(v)}(\mathbb{Z}^s)}a^\ast\operatorname{Kum}^\Gamma_{\textbf{n}(v)}(\mathbb{Z}^s)$$
    and write 
    $$t^\circ \colon \underline{(A^{\circ})^{(\textbf{n}(v))}_{a^\ast \operatorname{Kum}_\Gamma}} \longrightarrow \underline{{A^{(n+1)}\times_{\mathpzc{A}^{[n+1]}\times_{^\ast \rho} \Delta} E}}$$ for the restriction of $t$ to the locus of subschemes with support inside $A^\circ$. The morphism $t^\circ$ is obtained by pushing forward under $k$ different open immersions $\underline{A^\circ}\hookrightarrow A$. Each inclusion is $A^\circ$ equivariant and thus determined by the image of the identity in $A^\circ$. Write $m_1,...,m_k$ for the image of the identity under each of these maps. Note that the image of $\sum_j n_j m_j$ in $A/A^\circ$ is zero, and thus $\sum_j n_j m_j$ lies in $A^\circ$. Indeed, this was achieved by taking a fibre product with $a^\ast \operatorname{Kum}_\Gamma$. Applying an automorphism to the first copy of $A^\circ$ - i.e. modifying $m_1$ by post-multiplying by some element of $A^\circ$ and noting that $A^\circ$ is divisible- we may assume that $\sum_j n_j m_j$ is the identity.
    %We have thus identified the open subset $A^{[n]}\times_{\mathpzc{A}^{[n]}\times_{^\ast \rho} \Delta} E$ of $Y_v\times_{E}\operatorname{Hilb}^{[n]}(\pi_1)$ (cf. Theorem~\ref{thm:LogHilbProper}) with the open subset 
    %$$(A^{\operatorname{log}})^{[\textbf{n}_i]}_{a^\ast \operatorname{Kum}_\Gamma}:=\prod_i(A^{\operatorname{log}})^{[n_i]}\times_{a^\ast \operatorname{Kum}_{\textbf{n}(v)}(\mathbb{Z}^s)}a^\ast\operatorname{Kum}^\Gamma_{\textbf{n}(v)}(\mathbb{Z}^s)$$ in the product of the conventional relative Hilbert schemes of points $\prod_i\operatorname{Hilb}^{[n_i]}(\pi_2)$ (cf. Theorem~\ref{thm:ComponentlogHilb}).

    Consider the following commutative square in the category of schemes
    $$\begin{tikzcd}
{\underline{(A^{\operatorname{log}})^{[\textbf{n}(v)]}_{a^\ast \operatorname{Kum}_\Gamma}}} \arrow[r, "t"] \arrow[d,"HC"] \arrow[dd, "m"', bend right=60] & {\underline{A^{[n+1]}\times_{\mathpzc{A}^{[n+1]}\times_{^\ast \rho} \Delta} E}} \arrow[d, "HC"] \arrow[dd, "S", bend left=90, looseness=1.8] \\
\underline{(A^{\operatorname{log}})^{(\textbf{n}(v))}_{a^\ast \operatorname{Kum}_\Gamma}} \arrow[r,"t'"] \arrow[d, "m'", dashed] & {\underline{A^{(n)}\times_{\mathpzc{A}^{[n+1]}\times_{^\ast \rho} \Delta} E}} \arrow[d, "S'", dashed] \\
\underline{A^\circ} \arrow[r] & \underline{A^\circ}
\end{tikzcd}$$
The morphism $t'$ is defined similarly to $t$ but on the level of Chow variety of points.
The morphisms labelled $HC$ are Hilbert--Chow morphisms, and the top square commutes because the Hilbert--Chow map commutes with open immersion. The morphism $m'$ (resp. $S'$) is multiplication in the group law of $A^{\operatorname{log}}$ (resp. $A$) (weighted where a point appears multiple times in the symmetric product). 

The proof of (2) will be complete if we identify the fibres of $S\circ t$ and $m$ over the identity, and thus follows once we know that $S \circ t = m$. This is immediate from our last paragraph once we check that $S'\circ t' = m'$. 
Since $(A^{\operatorname{log}})^{[\textbf{n}_i]}_{a^\ast \operatorname{Kum}_\Gamma}$ is irreducible and $A^\circ$ is separated, by \stacks{01RH} it will suffice to show that these maps agree on the (scheme theoretically dense in any model) open $\prod_i(A^{\circ})^{[n_i]}$. 
But we observed that the restriction $t^\circ$ of $t'$ is pushforward under $k$ inclusions, and thus this last claim follows from elementary group theory once we recall that $\sum_j n_j m_j$ is the identity. 
\end{proof}

\begin{corollary}\label{corr:DualIntCx}
    The dual intersection complex of $\mathcal{X}_0$ is canonically $B_\Lambda$.
\end{corollary}

\begin{proof}
    Certainly we have a morphism $$\mathcal{X}_0\rightarrow a^\ast \operatorname{Kum}_{2n}^\Lambda(\mathpzc{A})\times_{a^\ast\rho}\Delta_0$$ and points of $a^\ast \operatorname{Kum}_{2n}^\Lambda(\mathpzc{A})\times_{a^\ast\rho}\Delta_0$ are in canonical, specialisation preserving, bijection with strata of $B_\Lambda$. Moreover, Theorem \ref{prop:ExplicitDescription} identifies the stratum associated to each point of $a^\ast \operatorname{Kum}_{2n}^\Lambda(\mathpzc{A})\times_{a^\ast\rho}\Delta_0$ with a logarithmic stratum of $\operatorname{Kum}^{\Gamma}_{\textbf{n}(v)}(A^{\operatorname{log}})$ for some choice of $\Gamma$. But each of these logarithmic strata is irreducible by Proposition~\ref{prop:KummerStrataIrreducible}(3) and the claim is proved.
\end{proof}

\section{The logarithmic generalised Kummer locus and its tropicalisation}
The purpose of this section is to prove Theorem \ref{mainthm:Niceness properties}, establish the commutative square of Theorem~\ref{Thm:maintrop}, and prove Theorem~\ref{mainthm:ReadOffModel}. 
\subsection{The conventional generalised Kummer locus} 
The goal of this subsection is to define a divisor $\mathcal{K}_{2n}^\circ$ in the moduli of conventional polarised IHS varieties which we call the generalised Kummer locus, and which Theorem~\ref{mainthm:Niceness properties} compactifies. We go on to relate the generalised Kummer locus to the moduli of principally polarised abelian surfaces.

\subsubsection{Hodge theory of generalised Kummer varieties}
%Taking the first chern class, we will think of the polarisation on a conventional Kummer variety $\operatorname{Kum}_{2n}(A)$ as an element of $H^2(\operatorname{Kum}_{2n}(A),\mathbb{Z})$.
Let $X$ be an irreducible holomorphic symplectic variety and recall from \cite[Section 1.9]{Huybrechts1999} that the cohomology of an IHS variety carries a bilinear \textit{Beauville--Bogomolov--Fujiki (BBF) form}, which we denote $$q\colon H^2(X,\mathbb{Z})\times H^2(X,\mathbb{Z}) \rightarrow \mathbb{Z}.$$ Given a smooth projective morphism $f\colon \mathcal{X}\rightarrow S$, $R^2f_\ast \mathbb{Z}$ is a local system, which inherits a bilinear form from the construction of $q$. This observation is used extensively in \cite{Markman2011}. The lattice isometry type of $(H^2(X,\mathbb{Z}),q)$ is thus locally constant in smooth families.

\begin{proposition}[{\cite[Proposition 8]{Beauville1983} and \cite[Remark 5(2)]{Britze2002}}]\label{prop:HodgeTheoryRecall}
    Let $A$ be a conventional abelian surface over $\mathbb{C}$, and let $n\geq 2$ be an integer. There is an isomorphism of Hodge structures, $$H^2(\operatorname{Kum}_{2n}(A),\mathbb{Z})\cong H^2(A,\mathbb{Z})\oplus \langle \delta \rangle$$
    and using $q$ to define a bilinear form on the right hand side, $\delta$ is a class such that $q(\delta,\delta)=-2(n+1)$ and $q(\delta,x)=0$ whenever $x\in H^2(A,\mathbb{Z})$. Moreover, under this isomorphism $2\delta$ is identified with the class of the exceptional divisor of the Hilbert--Chow morphism restricted to the Kummer variety. Finally $\delta$ has Hodge type $(1,1)$.
\end{proposition}

\begin{remark}\label{rem:MoveClasses}
    We pause to explain a canonical inclusion $$\hat{\iota}\colon H^2(A,\mathbb{Z})\hookrightarrow H^2(\operatorname{Kum}_{2n}(A),\mathbb{Z}).$$
    Write $\operatorname{p}_i\colon A^{n+1}\rightarrow A$ for the $i^\text{th}$ projection map, and note that given a class $\eta$ in $H^2(A,\mathbb{Z})$ the form $\eta^{n+1} = \sum_i \operatorname{p}_i^\ast \eta$ descends to a form $\eta^{(n+1)}$ on the symmetric product $A^{(n+1)}$. Pullback of $\eta^{(n+1)}$ along the Hilbert--Chow map defines a class $\eta^{[n+1]}$ on $A^{[n+1]}$. Let $\hat{\iota}(\eta)$ be the restriction of $\eta^{[n+1]}$ to the generalised Kummer variety.
\end{remark}
Let $U$ be the rank two lattice with hyperbolic bilinear form, set $\langle k\rangle$ to be the one dimensional lattice whose generator is always denoted $\delta$ and has $q(\delta,\delta)=k$. In the case that $n\geq 2$, it follows from Proposition \ref{prop:HodgeTheoryRecall} that the lattice type of the second cohomology (with form $q$) of any generalised Kummer variety of dimension $2n$ is $$\Lambda := U\oplus U \oplus U \oplus \langle -2(n+1) \rangle.$$
Where confusion is not likely we will write $\langle -2(n+1) \rangle = \langle \delta \rangle$. %Note that $2\delta$ is the class of the exceptional divisor $E$ for the restriction of the Hilbert Chow morphism, see \cite[Proposition 8]{Beauville1983}. 
Following (for example) \cite[Chapter 14 Section 0.1]{huybrechtsK3}, note that $q$ induces a map $$\Lambda \rightarrow \Lambda^\vee:= \operatorname{Hom}(\Lambda,\mathbb{Z}),$$
and define the \textit{divisibility} of an element $v$ in $\Lambda$ to be the largest integer $k$ such that the image of $v$ in $\Lambda^\vee$ under the morphism induced by $q$ is $k$ times a primitive class. The \textit{discriminant group} is the quotient $\Lambda^\vee/\Lambda$.

\begin{definition}\label{defn:PrincipalPolarisation}
    Let $(X,L)$ be a polarised IHS variety of generalised Kummer type. We say that the polarisation $L$ is \textit{principal} if there exists an integer $N$ and a choice of lattice isometry $$H^2(X,\mathbb{Z}) \cong U\oplus U \oplus U \oplus \langle \delta \rangle$$ which identifies $L$ with a class of the form $NH -\delta$ where $H\in U^{\oplus 3}$ and $q_{U^{\oplus 3}}(H,H) = 2$. 
\end{definition}

The context for Definition \ref{defn:PrincipalPolarisation} is that under the map $\hat{\iota}$ of Remark~\ref{rem:MoveClasses} a polarisation $H^2(A,\mathbb{Z})$ is sent to a class $H$ for which $q(H,H)=2$ if and only if $H$ is a principal polarisation of the abelian variety. Thus Definition~\ref{defn:PrincipalPolarisation} is modelled on a polarisation obtained from a principal polarisation on $A$.

\begin{definition}\label{defn:K2nH}
    Fix a positive integer $N$. We define the Noether--Lefschetz locus $\operatorname{NL}(N)$ to be the locus in the moduli of principally polarised IHS varieties $(X,L)$ of generalised Kummer type such that $q(L,L) = 2N^2-2(n+1)$ and also there exists a class $\delta\in H^2(X,\mathbb{Z})$ which satisfies the following hypotheses.
    \begin{enumerate}[leftmargin = *]
        \item The class $\delta$ is of Hodge type $(1,1)$.
        \item There is an equality $q(\delta,\delta) = -2(n+1)$.
        \item The image $q(\delta,-) = \delta^\vee$ of $\delta$ in $\Lambda^\vee$ has divisibility $2(n+1)$ in $\Lambda^\vee$. %, and $\delta^\vee/2(n+1)$ is a generator of the discriminant group.
        \item There is an equality $q(\delta,L) = 2(n+1)$.
        \item The class $L+\delta$ is $N$ times a primitive class.
    \end{enumerate}
\end{definition}

\begin{remark}
    In the sequel we would like to argue that the Hodge structure on $H^2(A,\mathbb{Z})$ can be recovered from the principal polarisation class together with the Hodge structure on $H^2(\operatorname{Kum}_{2n}(A),\mathbb{Z})$. This is a key tool for relating the generalised Kummer locus to the moduli of principally polarised abelian varieties. Fortunately by Eichler's criterion, see \cite[Proposition~3.3]{GritsenkoHulekSankaran}, for any choice of $\delta$ satisfying the hypotheses of Definition \ref{defn:K2nH} there is a lattice isometry such that $$H^2(X,\mathbb{Z}) \cong U\oplus U \oplus U \oplus \langle {\delta} \rangle.$$
    However, whilst the first three summands are determined by a choice of $\delta$ (as its orthogonal complement under $q$), one can worry that there are many choices of $\delta$ satisfying our hypotheses, at least for small values of $N$. Later we will fix $N$ large and argue that this makes $\delta$ unique.
\end{remark}

\subsubsection{From abelian surfaces to generalised Kummer varieties} 
%Although the generalised Kummer locus will be a divisor in the moduli of IHS varieties, we will construct the logarithmic generalised Kummer locus from the moduli of abelian varieties with level structure. 
We now take the first steps in relating the moduli of principally polarised abelian surfaces to the moduli of IHS varieties of generalised Kummer type. This will allow us to define $\mathcal{K}_{2n}^\circ$.
\begin{proposition}\label{prop:TangentSpaceANDUniquePointwise}
    We establish the following facts. 
    \begin{enumerate}[leftmargin = *]
        \item Let $x\in \operatorname{NL}(N)$ be a point associated to a polarised IHS variety $(X,L=NH-\delta)$ such that there exists a unique class $\delta$ in $H^2(X,\mathbb{Z})$ satisfying the hypotheses listed in Definition \ref{defn:K2nH}. The period map identifies the Zariski tangent space $T_x\operatorname{NL}(N)$ with $\langle H,\delta\rangle^\perp\cap H^{1,1}(X,\mathbb{C})$.
        \item Let $A$ be any abelian surface and $H$ any principal polarisation for $A$ (i.e. $q(H,H)=2$). Let $\delta$ be half of the Hilbert--Chow exceptional class in $H^2(\operatorname{Kum}_{2n}(A),\mathbb{Z})$. There is an integer $N_0$, which does not depend on $A$, such that for all $N>N_0$, the class $\delta$ in the cohomology of the principally polarised generalised Kummer variety $(\operatorname{Kum}_{2n}(A), NH -\delta)$ is the unique class satisfying the hypotheses of Definition \ref{defn:K2nH}. 
    \end{enumerate}
\end{proposition}
\begin{proof}
    For claim (1) note the tangent space to the moduli of IHS manifolds is canonically $H^1(X,T_X)$, and the symplectic form induces an isomorphism  $ H^1(X,T_X) \cong H^{1,1}(X,\mathbb{C})$. The locus $\operatorname{NL}(N)$ consists of varieties with a polarisation and satisfying five conditions. Conditions (2-5) are open and thus \'etale locally $\operatorname{NL}(N)\subset M_\text{IHS}$ is the closed subset characterised by the existence of a polarisation $L$ and class $\delta$ described in condition (1). By \cite[Lemma 5.16]{Voisin_2003}, the tangent space to the Noether--Lefschetz locus associated to $\delta$ is $\langle \delta \rangle^\perp\subset H^{1,1}(X,\mathbb{C})$. Similarly the tangent space of the NL divisor associated to $H$ is $\langle H \rangle^\perp\subset H^{1,1}(X,\mathbb{C})$. Since $\delta$ is not ample but $L$ is, the sublattice $\langle H,\delta \rangle$ which contains $L$ has rank $2$ and so these Noether--Lefschetz loci intersect transversely, and their intersection has tangent space as described in (1).

    We now prove (2). By \cite[Section 1.10]{Huybrechts1999} the restriction of $q$ to the N\'eron--Severi group $\operatorname{NS}_\mathbb{R} = \mathbb{R}\otimes \operatorname{NS}(\operatorname{Kum}_{2n}(A))$ has signature $(1,\rho-1)$ for $\rho$ the Picard rank. We know that $q(H,H)=2>0$, and thus the restriction of $q$ to the orthogonal complement $H^\perp$ of $H$ is negative definite.

    Take any $\delta'$ in $\operatorname{NS}(X)$ which satisfies the hypotheses of Definition \ref{defn:K2nH}. Note that there is a unique way to express $\delta' = aH + y$ in $\operatorname{NS}_\mathbb{R}$ where $y \in H^\perp$ and $a = q(\delta',H)/q(H,H)\in (1/2)\mathbb{Z}$. Constraints (2) and (4) of Definition \ref{defn:K2nH} are now 
    \begin{equation}\label{eq:constr} q(\delta',\delta')= 2a^2 + q(y,y) = -2(n+1)\quad \quad q(\delta',L)  = 2Na - q(y,\delta) = 2(n+1).\end{equation}

    Both $y$ and $\delta$ lie in $H^\perp$ where $-q$ is positive definite, so we may apply Cauchy--Schwarz to $-q$ with inputs $y,\delta$. Starting by applying the triangle inequality to the second expression in equation~\eqref{eq:constr}, and then applying Cauchy--Schwarz to $-q$ to obtain a second inequality, we learn:
    $$
\begin{aligned}
2N|a|
&\leq 2(n+1)+|q(y,\delta)| \\
&\leq 2(n+1)
   + \sqrt{|q(y,y)|}\sqrt{2(n+1)} \\
&\leq 2(n+1)
   + \sqrt{2(n+1)}
     \bigl(|a| \sqrt{2}+\sqrt{2(n+1)}\bigr) \\
&= 4(n+1)
   + \sqrt{2(n+1)}(|a| \sqrt{2}) .
\end{aligned}
$$
    where line 3 followed by rearranging the first expression of equation~\eqref{eq:constr}. Rearranging the final inequality for $|a|$, we learn that for $N$ large $|a|<1/2$, and is thus zero. In particular $y=\delta'$ lies in $H^\perp$.

    Since we know $a=0$, the second formula of equation \eqref{eq:constr} ensures that $q(y,\delta) = -2(n+1)$. Thus using condition (2) of Definition~\ref{defn:K2nH} twice we find that $$q(y-\delta,y-\delta) = q(y,y) -2q(y,\delta) +q(\delta,\delta) = 0.$$ But $y,\delta\in H^\perp$ where $-q$ is positive definite, so this is only possible if $(\delta' =) y=\delta$.
\end{proof}

Continuing Remark \ref{rem:MoveClasses}, we will identify $H^2(A,\mathbb{Z})$ with the direct summand of $H^2(\operatorname{Kum}_{2n}(A),\mathbb{Z})$ which is pulled back from $A^{(n+1)}$.

\subsubsection{The generalised Kummer locus}
We now define the generalised Kummer locus. This is the central object of Theorem \ref{mainthm:Niceness properties}. We write $\delta$ for half of the class of the exceptional divisor of the Hilbert--Chow morphism.

\begin{definition}\label{defn:genKummerLocus}
    Fix once and for all an integer $N>N_0$ with $N_0$ as in Proposition \ref{prop:TangentSpaceANDUniquePointwise}. Further assume that $N_0$ is sufficiently large to ensure that $NH-\delta$ is ample for any $A$ an abelian surface. We write $\mathcal{K}_{2n}^\circ$ for the \textit{generalised Kummer locus}: the connected component in $\operatorname{NL}(N)$ whose closed points are of the form $(\operatorname{Kum}_{2n}(A),NH-\delta)$ where $A$ is a principally polarised abelian surface. 
\end{definition}

\begin{lemma}\label{lem:resolveComment}
    The locus of Definition~\ref{defn:genKummerLocus} is a connected component in $\operatorname{NL}(N)$. 
\end{lemma} 
\begin{proof}
We will show that it is both open and specialisation closed. Specialisation closed because Proposition \ref{prop:HodgeTheoryRecall} forces the Hodge structure to degenerate whenever the abelian variety degenerates. It is open because, as follows from Proposition~\ref{prop:TangentSpaceANDUniquePointwise}(1), the map from $\mathcal{A}_2$ induces an isomorphism on tangent spaces at each point.
\end{proof}
Note that $\mathcal{K}_{2n}^\circ$ a priori depends on $N$, but \textit{a posteriori} we deduce from Proposition \ref{prop:NLlocus} that its isomorphism class, and the isomorphism class of its universal (unpolarised) IHS variety, is independent of $N>N_0$. For this reason we suppress $N$ from the notation.

\subsection{The logarithmic generalised Kummer locus}\label{sec:LogGenKummerLocus}
We now turn to constructing the logarithmic generalised Kummer locus $\mathcal{K}_{2n}$ which will compactify $\mathcal{K}_{2n}^\circ$. Subsection \ref{subsec:ConstructCompact} constructs the logarithmic DM stack $\mathcal{K}_{2n}$ and its universal family, while Subsection \ref{subsec:RelateK2n0} identifies the open substack which has trivial logarithmic structure (and its universal family) with the DM stack $\mathcal{K}_{2n}^\circ$ (and its universal family).

\subsubsection{Automorphisms of logarithmic Kummer varieties}\label{sec:AutLogKummer}
We begin by describing how the gerbe structure of $\mathcal{K}_{2n}$, which will be constructed later, will act upon the universal generalised Kummer varieties. Let $A$ be a logarithmic abelian variety over $\Delta_0$ and let $T_{n+1}$ be its $(n+1)$-torsion subgroup. 
\begin{definition}
    Let $\mathsf{Aut}(A)$ be the group of automorphisms of $A$ as polarised abelian variety. Let $\mathsf{Aut}(\operatorname{Kum}_{2n}(A))$ be the group of automorphisms of $\operatorname{Kum}_{2n}(A)$ as a sheaf on the category of logarithmic schemes.
\end{definition}
    
Automorphisms of $A$ which preserve the identity induce automorphisms of $\operatorname{Kum}_{2n}(A)$. Thus we obtain an injection $$\mathsf{Aut}(A)\hookrightarrow \mathsf{Aut}(\operatorname{Kum}_{2n}(A)).$$
Moreover, translation by $x\in T_{n+1}$ preserves $\Sigma_{n+1}^{-1}(0)$ in $A^{[n+1]}$ and thus acts upon $\operatorname{Kum}_{2n}(A)$. Our two group actions combine to define a group homomorphism
$$\varphi_{2n}\colon T_{n+1}\rtimes \mathsf{Aut}(A)\rightarrow \mathsf{Aut}(\operatorname{Kum}_{2n}(A))$$

\begin{proposition}\label{prop:AMGpAnyField}
    For any integer $n\geq 2$ the morphism $\varphi_{2n}$ is an injection of groups.
\end{proposition}

\begin{proof}
    The preimage of the identity in $A$ under the summation map $A^{(n+1)}\rightarrow A$ is denoted $K_{(2n)}$. The Hilbert--Chow morphism $A^{[n+1]}\rightarrow A^{(n+1)}$ restricts to define a map $\operatorname{HC}\colon \operatorname{Kum}_{2n}(A)\rightarrow K_{(2n)}$, and this map is an isomorphism on the dense open of both source and target where the support consists of $n+1$ points. 
    
    One can construct a morphism $$\varphi_{(2n)}\colon T_{n+1}\rtimes \mathsf{Aut}(A)\rightarrow \mathsf{Aut}(\operatorname{K}_{(2n)})$$ similarly to the construction of $\varphi_{2n}$. The images of an element of $T_{n+1}\rtimes \mathsf{Aut}(A)$ under $\varphi_{2n}$ and $\varphi_{(2n)}$ agree on the common open subset defined by $\operatorname{HC}$. Since an automorphism is determined by its restriction to a scheme-theoretically dense open by \stacks{01RH}, it will suffice to show that $\varphi_{(2n)}$ is injective. 
    
    A point of $K_{(2n)}$ is the data of $n+1$ sections of $A/\Delta_0$ which sum to the identity. We write these supports as an unordered tuple $(p_0,...,p_n)$ of points on $A$ and reiterate that $\sum_i p_i$ is trivial in the group law on $A$.

    There is a map which is injective on closed points from $A$ to $\operatorname{K}_{(2n)}$ sending $a$ to $(a,...,a, -na)$. Restriction defines a map $$\operatorname{Im}(\varphi_{2n})\rightarrow T_{n+1}\rtimes\mathsf{Aut}(A)$$ whose image is precisely the automorphisms of $A$ which send the identity to a point of the $(n+1)$-torsion. This map is a splitting of $\varphi_{(2n)}$ and thus $\varphi_{(2n)}$ is injective.
\end{proof}
\begin{remark}\label{rem:ClassicalCase}
    For $A$ an abelian surface over $\mathbb{C}$ the morphism $\varphi$ is known to be an isomorphism to the collection of automorphisms which preserve the Hilbert--Chow exceptional divisor, see \cite{BOISSIERE2011553}. We expect the same statement remains true in the logarithmic category.
\end{remark}

\subsubsection{Constructing $\mathcal{K}_{2n}$.}\label{subsec:ConstructCompact}
Recall that every Noetherian logarithmic scheme admits a strict \'etale cover by atomic logarithmic schemes \cite[Lemma 2.2.5]{AWbirational}. Let $S$ be a Noetherian logarithmic scheme. Since the generalised Kummer construction of Definition~\ref{defn:logKummer} commutes with base change, for each integer $n$ this construction descends to associate a logarithmic generalised Kummer variety to any family of logarithmic abelian surfaces $A\rightarrow S$. 

\begin{definition}
    Recall $\mathcal{Y}_2(n+1)\rightarrow \mathcal{A}_2(n+1)$ is the universal logarithmic abelian surface (Section~\ref{sec:ModulilogAVLevelStr}) over the moduli space of logarithmic abelian surfaces with $(n+1)$-level structure. We write $$\pi_{n+1}\colon Y_2(n+1):=\operatorname{Kum}_{2n}(\mathcal{Y}_2(n+1)/\mathcal{A}_2(n+1)) \rightarrow \mathcal{A}_2(n+1)$$
    for the associated logarithmic Kummer variety.
\end{definition}

\begin{definition}
Define a finite group,
$$
G_{n+1}=(\mathbb{Z}/(n{+}1)\mathbb{Z})^{4}\rtimes 
\mathsf{Sp}_{4}(\mathbb{Z}/(n{+}1)\mathbb{Z}),
$$ where $\mathsf{Sp}_{4}(\mathbb{Z}/(n{+}1)\mathbb{Z})$ acts upon $(\mathbb{Z}/(n{+}1)\mathbb{Z})^{4}$ via the standard representation.
\end{definition}
We consider $\mathcal{A}_2(n+1)$ to be a $G_{n+1}$ variety through the action of $\mathsf{Sp}_{4}(\mathbb{Z}/(n{+}1)\mathbb{Z})$ on the universal family's level structure via the standard action of the symplectic group on $(\mathbb{Z}/(n+1)\mathbb{Z})^4$.
\begin{lemma}\label{lem:GroupActionGerbe}
    Let $n\geq 2$ be an integer. There is an action of $G_{n+1}$ upon ${Y}_2(n+1)$ satisfying the following hypotheses.
    \begin{enumerate}[leftmargin = *]
        \item The morphism $\pi_{n+1}$ is $G_{n+1}$-equivariant, and thus descends to a morphism of quotient stacks
        $$\mathcal{U}_{2n}\coloneqq[{Y}_2(n+1)/G_{n+1}] \rightarrow \mathcal{K}_{2n}\coloneqq [\mathcal{A}_2(n+1)/G_{n+1}].$$
        \item Let $p$ be a standard logarithmic geometric point of $\mathcal{K}_{2n}$ and let $A_p/p$ be the abelian surface specified by any point of its preimage in $\mathcal{A}_2(n+1)$. The gerbe structure at $p$ in $\mathcal{K}_{2n}$ is the group $(\mathbb{Z}/(n+1)\mathbb{Z})^4\rtimes \operatorname{Aut}(A_p)$ and the action on $\operatorname{Kum}_{2n}(A_p)$ is induced by the morphism $\varphi_{2n}$ described after Definition~\ref{sec:AutLogKummer}.
    \end{enumerate}
\end{lemma}
\begin{proof}
    The universal family over $\mathcal{A}_2(n+1)$ is, by definition of level structure, equipped with an identification of finite \'etale $\mathcal{A}_2(n+1)$ group schemes, $$h\colon (\mathbb{Z}/(n+1)\mathbb{Z})^4_{\mathcal{A}_2(n+1)} \rightarrow \mathcal{Y}_2[(n+1)],$$ where $\mathcal{Y}_2[(n+1)]$ is the $(n+1)$-torsion of $\mathcal{Y}_2(n+1)$. The group $(\mathbb{Z}/(n{+}1)\mathbb{Z})^{4}$ acts upon $\mathcal{Y}_2(n+1)$ by sending $y$ to translation through $h(y)$. The group $\mathsf{Sp}_{4}(\mathbb{Z}/(n{+}1)\mathbb{Z})$ acts as follows: an element $g$ specifies a new framing $g\circ h$ and thus, through the universal property of $\mathcal{A}_2(n+1)$, an automorphism of $\mathcal{Y}_2(n+1)$. By construction $\mathsf{Sp}_{4}(\mathbb{Z}/(n{+}1)\mathbb{Z})$ the map $\pi_{n+1}$ is $G_{n+1}$ equivariant.

    Every automorphism of a logarithmic abelian surface which fixes the $(n+1)$-torsion is the identity \cite[Lemma 5.1]{KajiwaraKatoNakayamaPartVII}, and thus $\mathcal{A}_2(n+1)$ has trivial stabiliser groups. It follows that the gerbe structure at a point $p$ in $\mathcal{K}_{2n}$ associated to abelian surface $A\rightarrow p$ is the stabiliser of any point in the preimage of $p$ within $\mathcal{A}_2(n+1)$. In particular, the stabiliser group is the subgroup $K$ of $G_{n+1}$ which fixes the abelian variety $A$. The group $K$ certainly contains the normal subgroup $(\mathbb{Z}/(n+1)\mathbb{Z})^4$ of $G_{n+1}$ and is thus determined by its image in $G_{n+1}/(\mathbb{Z}/(n+1)\mathbb{Z})^4$. But this image is exactly the stabiliser of $A$ in the moduli of logarithmic abelian surfaces without level structure, and thus coincides with $\text{Aut}(A)$.
\end{proof}

\begin{definition}
    The stack quotient $\mathcal{K}_{2n}\coloneqq [\mathcal{A}_2(n+1)/G_{n+1}]$ is called the \textit{logarithmic generalised Kummer locus}, and the morphism $\mathcal{U}_{2n}\rightarrow \mathcal{K}_{2n}$ of Lemma~\ref{lem:GroupActionGerbe}(1) is the \textit{universal logarithmic generalised Kummer variety}.
\end{definition}

\subsubsection{Relation to moduli of polarised IHS varieties}\label{subsec:RelateK2n0}
    In order to relate $\mathcal{K}_{2n}$ to the moduli space of polarised IHS varieties, we equip $\mathcal{U}_{2n}$ with a polarisation relative to $\mathcal{K}_{2n}$. Recall that in Definition~\ref{defn:genKummerLocus} we fixed once and for all an integer $N$.

\begin{definition}\label{defn:polarisation}
    Given a family of principally polarised logarithmic abelian surfaces $(A\rightarrow S,H)$, we define a polarisation upon $A^{n+1}$ by tensoring the principal polarisation pulled back along each projection to $A$. This polarisation descends to a polarisation $\mathcal{L}$ on the stack quotient (and thus also on the coarse space of) $A^{n+1}/\text{Sym}(n+1)$. The Hilbert--Chow map of the relative Hilbert scheme, $$A^{[n+1]}\rightarrow A^{n+1}/\text{Sym}(n+1)$$
    is a resolution of singularities with exceptional divisor $E$. We define a polarisation upon $A^{[n+1]}$ through the formula $\mathcal{L}^{\otimes N}\left(- \frac{1}{2}E\right)$ and note this restricts to a polarisation on $\operatorname{Kum}_{2n}(A)$.
\end{definition}

We write $\mathcal{K}_{2n}^{\operatorname{str}}$ for the open in (any model of) $\mathcal{K}_{2n}$ characterised as the locus with trivial logarithmic structure. Write $\mathcal{U}_{2n}^{\operatorname{str}} = \mathcal{U}_{2n}\times_{\mathcal{K}_{2n}}\mathcal{K}_{2n}^{\operatorname{str}}$.

\begin{proposition}\label{prop:NLlocus}
    The family $$\pi^{\operatorname{str}}\colon \mathcal{U}_{2n}^{\operatorname{str}} \rightarrow \mathcal{K}_{2n}^{\operatorname{str}},$$ equipped with the polarisation of Definition~\ref{defn:polarisation}, induces an isomorphism of conventional DM stacks $$k \colon \mathcal{K}_{2n}^{\operatorname{str}}\rightarrow \mathcal{K}_{2n}^\circ.$$
\end{proposition}

\begin{proof}
    %Noting that the polarisation of Definition~\ref{defn:polarisation} specified a class $h\in H^2(X, \mathbb{Z})$, we recall also the polarised period domain, $$\Omega_h\cong \{[\eta] \in \mathbb{P}(\Lambda \otimes \mathbb{C})| q(\eta, h)=0, q(\eta,\overline{\eta})>0, q(\eta, \eta)=0\}$$ which has dimension $4$. Write $\Omega_h(\delta)$ for the collection of points $[\eta]$ in $\Omega_h$ such that $q(\eta, \delta)=0$. 
    The morphism $\pi^{\operatorname{str}}$ together with the polarisation of Definition \ref{defn:polarisation} define a morphism from $\mathcal{K}_{2n}^{\operatorname{str}}$ to the moduli of polarised IHS varieties of generalised Kummer type. By construction the image lies within $\mathcal{K}_{2n}^\circ$ and thus we have a morphism $$k\colon \mathcal{K}_{2n}^{\operatorname{str}} \rightarrow \mathcal{K}_{2n}^\circ.$$ 
    We will now show that $k$ is an isomorphism of DM stacks.
    
    \textbf{Step I: k is \'etale.} We claim that $k$ induces an isomorphism of tangent spaces at each point. Let $p$ be any point of $\mathcal{K}_{2n}^{\operatorname{str}}$ and choose any point $x$ in the preimage of $p$ in $\mathcal{A}_2(n+1)$. Write $A$ for the abelian surface associated to $x$, and set $H\in H^2(A,\mathbb{Z})$ its polarisation class. Since $G_{n+1}$ is \'etale, the tangent space at $p$ is canonically identified with the tangent space at $x$. The period map identifies the tangent space of $\mathcal{A}_2(n+1)$ at $x$ with $H^\perp$ in $H^2(A,\mathbb{C})$. 
    On the level of period maps $k$ is the inclusion of $H^2(A,\mathbb{C})$ into $H^2(\operatorname{Kum}_{2n}(A),\mathbb{C})$ described in Proposition~\ref{prop:HodgeTheoryRecall}. Under this inclusion, $H^\perp\cap H^{1,1}(A,\mathbb{C})$ in $H^2(A,\mathbb{C})$ is mapped isomorphically to $\langle H,\delta \rangle^\perp\cap H^{1,1}(\operatorname{Kum}_{2n}(A),\mathbb{C})$ in $H^2(\operatorname{Kum}_{2n}(A),\mathbb{C})$. But by Proposition~\ref{prop:TangentSpaceANDUniquePointwise}(1) (and Lemma~\ref{lem:resolveComment}) $\langle H,\delta \rangle^\perp$ in $H^2(\operatorname{Kum}_{2n}(A),\mathbb{C})$ is the tangent space of $\mathcal{K}_{2n}^\circ$ at $k(p)$, and the identification of tangent spaces is proved.

    Any morphism of smooth complex DM stacks which induces an isomorphism on tangent spaces at every point has locally equidimensional fibres and so miracle flatness applies. Thus $k$ is flat. The isomorphism of tangent spaces further forces $k$ unramified. Flat and unramified maps are \'etale.

\textbf{Step II: k induces a bijection on $\mathbb{C}$-points.} Surjectivity is by construction: it remains to prove injectivity. By Proposition~\ref{prop:TangentSpaceANDUniquePointwise} the polarised variety $(X,L)$ specified by a geometric point of $\mathcal{K}_{2n}^\circ$ determines the class $\delta$ which is half the class of the exceptional divisor of the Hilbert--Chow map. The orthogonal complement of $\delta$ in $H^2(X,\mathbb{Z})$ under $q$ is thus determined, and by Proposition \ref{prop:HodgeTheoryRecall} we obtain $H^2(A,\mathbb{Z})$. But knowing $H^2(A,\mathbb{Z})$ determines $A$ up to dualising by the main theorem of \cite{Shioda1978ThePM}. Since $A$ admits a principal polarisation it is self dual, and $A$ is determined. It remains to show that the polarisation on $A$ is determined.

The polarisation class add $\delta$ divided by $N$ determines a class $H$ in a polarised Hodge structure isomorphic to $H^2(A,\mathbb{Z})$ with square $2$. Note that $H$ is ample on $A$ and thus effective. We need to show that this polarisation class is well defined independent of the choice of isomorphism between this abstract polarised Hodge structure $\delta^\perp$ and $H^2(A,\mathbb{Z})$. We will adopt notation of \cite{Shioda1978ThePM}, and in particular this allows us to discuss the determinant of a morphism between two lattices isomorphic to $H^2(A,\mathbb{Z})$. If there are two choices of isomorphism between $\delta^\perp$ and $H^2(A,\mathbb{Z})$ this gives an automorphism of $H^2(A,\mathbb{Z})$ as a polarised Hodge structure. If this automorphism of $H^2(A,\mathbb{Z})$ has determinant one then by \cite[Lemma 1]{Shioda1978ThePM} either it will send an ample class to minus an ample class, or it is induced by an isomorphism on $H^1(A,\mathbb{Z})$ and thus an automorphism of $A$. Since at most one of $H$ or $-H$ is effective, $H$ is determined. If instead our automorphism of $H^2(A,\mathbb{Z})$ has determinant $-1$ then we can compose with the morphism $H^2(A,\mathbb{Z})\rightarrow H^2(A^\vee,\mathbb{Z})$ of \cite[Lemma 3]{Shioda1978ThePM} which the same reference shows has determinant $-1$. We have thus reduced to the determinant $1$ case.

\textbf{Step III: $k$ is representable by algebraic spaces.} For $N$ large, and $A$ a (conventional) abelian variety, all polarised automorphisms of $\operatorname{Kum}_{2n}(A)$ preserve the Hilbert--Chow exceptional. Thus by Remark~\ref{rem:ClassicalCase}, the morphism on stabiliser groups induced by $k$ at all $\mathbb{C}$ points is an isomorphism. The morphism of inertia stacks $\mathcal{I}_{\mathcal{K}_{2n}^{\operatorname{str}}}\rightarrow \mathcal{I}_{\mathcal{K}_{2n}^\circ}\times_{\mathcal{K}_{2n}^\circ} \mathcal{K}_{2n}^{\operatorname{str}}$ induced by $k$ is an isomorphism by \stacks{0DU9}. 
Invoking \stacks{04Z6}, the relative inertia stack is isomorphic to $\mathcal{K}_{2n}^{\operatorname{str}}$. 
It now follows from \stacks{04YY} ((3) implies (1)) that $k$ is representable by algebraic spaces. 

\textbf{Step IV: Conclusion.} Since $\mathcal{K}_{2n}^\circ$ is a Deligne--Mumford stack of finite type, \'etale locally it is of the form $[V/G]$ for $G$ a finite group, and by representability of $k$ this morphism is \'etale locally a map $$[\widetilde{V}/G]\rightarrow [V/G],$$ induced by an \'etale map of algebraic spaces $k'\colon \widetilde{V}\rightarrow V$. By orbit-stabiliser theorem, bijectivity of $k$ on geometric points implies bijectivity of $k'$ on geometric points. Thus $k'$ is an \'etale map of algebraic spaces of finite type over $\mathbb{C}$ which induces a bijection on geometric points. Such spaces are Jacobson, and thus $k'$ is an \'etale monomorphism so an open immersion which surjects on closed (and thus all by Jacobson property) points. We deduce that $k'$ is an isomorphism, but being an isomorphism is \'etale local on the target so our proof is complete. 
\end{proof}

\subsection{The tropical generalised Kummer locus}\label{sec:TropGenKummerlocus}
We establish the commutative square described in Theorem~\ref{Thm:maintrop}. 
\begin{lemma}\label{lem:squarefromMainThm}
    There are compatible actions of $G_{n+1}$ on $\mathpzc{A}_2(n+1)$ and $\mathpzc{Y}_2(n+1)$, such that the commutative square of Theorem~\ref{thm:commuting square} descends to a diagram of stack quotients,
      $$\begin{tikzcd}
\mathcal{U}_{2n} \arrow[r, "\widetilde{\varpi}_\mathcal{U}"] \arrow[d, "\pi"'] 
& a^\ast [\operatorname{Kum}_{2n}(\mathpzc{Y}_2(n+1))/G_{n+1}]=\colon a^\ast\mathpzc{U}_{2n} \arrow[d, "a^\ast \pi^{\mathsf{trop}}"] \\
 \mathcal{K}_{2n} \arrow[r, "\varpi_\mathcal{K}"] 
& a^\ast  [\mathpzc{A}_2(n+1)/G_{n+1}] =\colon  a^\ast \mathpzc{K}_{2n}
\end{tikzcd}.$$  
in which the horizontal maps are strict.
\end{lemma}
\begin{proof}
    Note that $G_{n+1}$ acts upon the objects of Definition~\ref{defn:TropAVLevel} exactly as in the proof of Lemma~\ref{lem:GroupActionGerbe}. The same argument establishes a $G_{n+1}$ equivariant commutative square, which descends to a diagram of quotient stacks. Since strictness can be checked strict \'etale locally, this property descends to our quotient square.
\end{proof}

\begin{theorem}\label{thm:GettingModels} We establish the following.
\begin{enumerate}[leftmargin = *]
    \item For any diagram of stacks upon the tropical site, 
    $$
    \begin{tikzcd} 
    \mathpzc{U}_{2n}^\Lambda \arrow[r] \arrow[d,"\pi"]
  & \mathpzc{U}_{2n} \arrow[d] \\
\mathpzc K_{2n}^\Lambda \arrow[r]
  & \mathpzc{K}_{2n}.
\end{tikzcd}$$
    where the horizontal maps are tropical alterations, and $\pi$ is a saturated morphism of cone stacks, the morphism
    $$\mathcal{U}_{2n}^\Lambda = \mathcal{U}_{2n}\times_{a^\ast \mathpzc{U}_{2n}}a^\ast \mathpzc{U}^\Lambda_{2n}\rightarrow \mathcal{K}_{2n}^\Lambda = \mathcal{K}_{2n}\times_{a^\ast \mathpzc{K}_{2n}} a^\ast \mathpzc{K}^\Lambda_{2n}$$
    is a vertical, saturated, and logarithmically smooth morphism of DM stacks with logarithmic structure.
    
    \item The open immersion $\mathcal{K}_{2n}^\circ \hookrightarrow \mathcal{K}_{2n}^\Lambda$ is a toroidal compactification.
    \end{enumerate}
\end{theorem}
\begin{proof}
    Note (1) can be checked \'etale locally on $\mathcal{K}_{2n}^\Lambda$, which is \'etale locally atomic. Thus (1) follows from Proposition~\ref{prop:ProperLogSmooth}. Claim (2) can also be checked \'etale locally, and $\mathcal{K}_{2n}$ admits an \'etale cover by $\mathcal{A}_2(n+1)$ by construction, so the result follows from \cite[Theorem~1.8]{KajiwaraKatoNakayamaPartVII}.
\end{proof}

\subsection{Properties of the logarithmic generalised Kummer locus}
Write $\operatorname{KumSh}$ for the functor that sends a stack on the strict \'etale site for its stackification in the kummer logarithmic \'etale topology.
\begin{proposition}\label{prop:KnBehaved} We record the following properties of the logarithmic generalised Kummer locus.
\begin{enumerate}
    \item The morphism $\operatorname{KumSh}(\mathcal{A}_2)\rightarrow \operatorname{KumSh}(\mathcal{K}_{2n})$ induces an isomorphism at the level of coarse moduli spaces.
    \item The stack $\mathcal{K}_{2n}$ is a proper logarithmic DM stack.
    \item The stack $\mathcal{U}_{2n}$ is a proper logarithmic DM stack.
\end{enumerate}
\end{proposition}
\begin{proof}
    For (1) observe that kummer log \'etale locally every principally polarised logarithmic abelian variety admits a level (n+1) structure \cite[Proposition 18.1(3)]{KajiwaraKatoNakayamaPartIV}. Thus the fibre over a test scheme $S$ in the category fibred in groupoids $\operatorname{KumSh}(\mathcal{A}_2)$ is the groupoid which can be presented with objects the maps from $S$ to $\operatorname{KumSh}(\mathcal{A}_2(n+1))$ and morphisms induced by the action of $\operatorname{Sp}_4(\mathbb{Z}/(n+1)\mathbb{Z})$. The fibre over $S$ in the category fibred in groupoids $\operatorname{KumSh}(\mathcal{K}_{2n})$ admits a presentation with the same objects and morphisms induced by $G_{n+1}$. It suffices to show that the setoids associated to the groupoids named above coincide. Equivalently, two maps from $S$ to $\mathcal{A}_2(n+1)$ are identified by $G_{n+1}$ if and only if they are identified by $\operatorname{Sp}_4(\mathbb{Z}/(n+1)\mathbb{Z})$. This is true because the action of $\mathbb{Z}/(n+1)\mathbb{Z}$ on $\mathcal{A}_2(n+1)$ is trivial.

    For (2) we first check that $\mathcal{K}_{2n}$ is a logarithmic DM stack. Indeed $\mathcal{K}_{2n}$ is defined as the global quotient stack of $\mathcal{A}_2(n+1)$ by the action of a finite \'etale group scheme. Composing with the projection, every logarithmic \'etale cover of $\mathcal{A}_2(n+1)$ thus induces a logarithmic \'etale cover of $\mathcal{K}_{2n}$. 
    
    Properness is invariant under logarithmic modifications and Kummer root stacks, and is thus invariant under all logarithmic alterations. Indeed for passing to a root stack this is clear, and invariance under logarithmic modification holds because the valuative criterion of \cite[Theorem~2.2.5.2]{MolchoWise} tests only against valuative test schemes, and whenever $S$ is valuative the $S$-valued points of a logarithmic space $X$ are invariant under logarithmic modifications of $X$. It thus suffices to check that $\mathcal{K}_{2n}$ admits a proper logarithmic alteration. We will show that $\mathcal{K}_{2n}^\Lambda$ is proper. Indeed, it is the global quotient stack of $\mathcal{A}_2(n+1)\times_{\mathcal{K}_{2n}} \mathcal{K}_{2n}^\Lambda$ by a finite group, and $\mathcal{A}_2(n+1)$ is proper by \cite[Theorem~1.6]{KajiwaraKatoNakayamaPartVII}.

    An \'etale cover of $\mathcal{K}_{2n}$ pulls back along the universal morphism to an \'etale cover of $\mathcal{U}_{2n}$. Claim (3) follows by combining claim (2) with Proposition~\ref{prop:ProperLogSmooth}. 
\end{proof}

\section{Examples}\label{sec:examples}
In this section we study two geometries:
\begin{enumerate}[leftmargin = *]
    \item We describe the geometry of the logarithmic generalised Kummer variety, and especially the family of integral affine manifolds, associated to a specific\footnote{We fix an example for concreteness, but note our analysis works in great generality.} family of degenerating abelian surfaces over $\mathbb{A}^2$. Our example captures both type~II and type~III degenerations. 
    \item We name an explicit cone complex structure upon the family of integral affine manifolds associated to any one-parameter degeneration for both the type~II and type~III geometries. 
\end{enumerate}
\subsection{Naming a family of logarithmic abelian varieties} We will now name a logarithmic abelian surface over $S = \mathbb{A}^2_{u,v}\backslash V((u^\alpha-1)(v^\beta-1))$ with its toric logarithmic structure. Over the interior of each torus invariant divisor the degeneration is of type~II, whilst the degeneration over the torus fixed point is of type~III.

\subsubsection{Families of logarithmic abelian varieties} Recall that associated to a proper, vertical, and logarithmically smooth family of varieties $X$ over a Noetherian scheme $S$ is a \textit{logarithmic Picard group} \cite{LogPicHigherDim,MolchoWise}. The connected component of the identity $\operatorname{Pic}^0(X/S)$ is known to be a logarithmic abelian variety whenever $S$ has generically trivial logarithmic structure \cite{LogPicHigherDim,MolchoWise}.
\begin{example}\label{ex:AbSurface1}
    Fix positive integers $\alpha, \beta$. Let $S=\mathbb{A}^2_{u,v}\backslash V((u^\alpha-1)(v^\beta-1))$ be affine space over $\mathbb{C}$ less $\alpha + \beta$ lines, with toric logarithmic structure and coordinates $u$ and $v$. 
    We write $S^\circ$ for the complement of $V(uv)$ inside $S$. 
    Consider the morphism, $$W = \mathbb{P}^2_{x_1,y_1,z_1}\times \mathbb{P}^2_{x_2,y_2,z_2}\times S\rightarrow S$$ and write $W^\circ= W\times_S S^\circ$. 
    Let $\hat{A}$ be the surface over $S$ whose ideal in $\mathcal{O}_W$ is given by: $$\left(y_1^2z_1 -x_1 (x_1-z_1)(x_1-u^\alpha z_1),y_2^2 z_2 - x_2 (x_2 - z_2)(x_2 - v^\beta z_2)\right),$$ and write $A^\circ$ for the abelian scheme $\hat{A}\times_S S^\circ$ over $S^\circ$.
    Thus $A^\circ$ is the product of two degenerating elliptic curves. Note that $\hat{A}\rightarrow S$ is logarithmically smooth, and thus $$A = \operatorname{Pic}^0_S(\operatorname{Pic}^0_S(\hat{A}))\rightarrow S$$ is a (and thus the) family of logarithmic abelian varieties over $S$ whose restriction over $S^\circ$ is $A^\circ$.
\end{example}

\subsubsection{Associated family of tropical abelian varieties} Following Proposition~\ref{prop:Summarise Log AV}(1) we record the family of tropical abelian varieties (i.e. the family of integral affine manifolds) which tropicalises $A\rightarrow S$. The Artin fan of $S$ with toric logarithmic structure is $[\mathbb{A}^2/\mathbb{G}_m^2]$ (with $\mathbb{G}_m^2$ acting as the dense torus). This Artin fan may be expressed as $a^\ast \mathbb{N}^2$, and we will denote $\sigma=\mathbb{N}^2$. Thus we will construct a family of integral affine manifolds over $\sigma$. 

\begin{example}\label{ex:TropAV1}
    We record the tropicalisation associated to Example \ref{ex:AbSurface1}, whose notation we continue. The tropical abelian variety is $\mathpzc{A}(\varphi)$, where $\varphi$ is the morphism $$\mathbb{Z}^2\rightarrow \operatorname{Hom}(\mathbb{Z}^2,\operatorname{Hom}(\mathbb{N}^2,\mathbb{Z})\cong \mathbb{Z}^2), \quad (s,t) \mapsto \left((a,b)\mapsto (2a \alpha s, 2b \beta t )\right).$$
    The total space of $\mathpzc{A}(\varphi)$ is of real dimension four (being a tropical abelian surface over $\sigma$ which has dimension two), and a depiction is far beyond the author's artistic skill. Instead we depict two instructive slices. First, choose a ray $(p=\text{log}(u),q=\text{log}(v))$ in $\mathbb{N}^2$ and pull our family of tropical abelian varieties back to this ray. Below centre we assume $p,q\ne 0$, and the associated integral affine manifold is of real dimension two. Right and left we assume precisely one of these parameters is zero (if both are zero then no degeneration has occurred and the tropicalisation is a point). Algebro-geometrically left and right are type~II degenerations of abelian surfaces, while centre is type~III.
    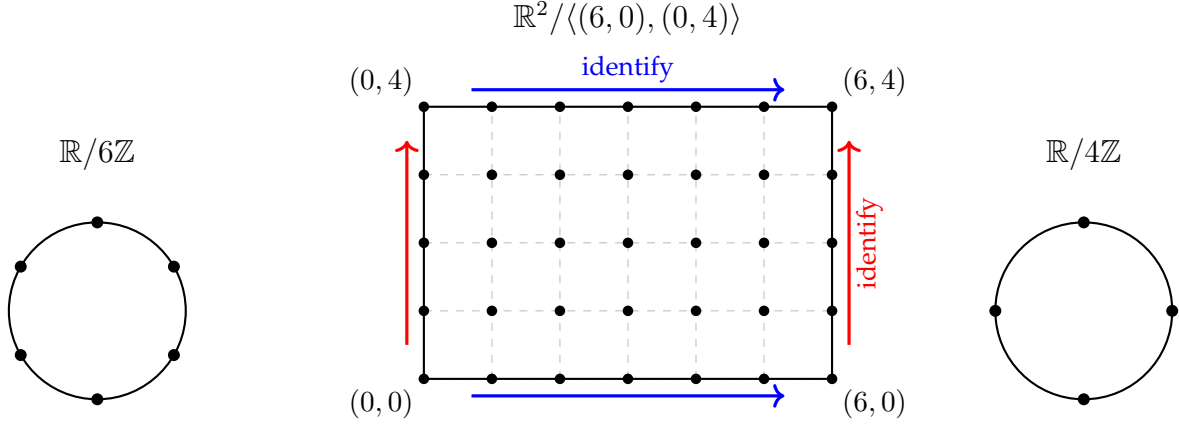
\begin{figure}
\begin{tikzpicture}[scale=0.9]
%========================
% Picture 2: circle with six lattice points
%========================
\begin{scope}[xshift=0cm]
  \node at (0,3.3) {\large $\mathbb{R}/6\mathbb{Z}$};

  \draw[thick] (0,1) circle (1.3);

  \foreach \ang in {90,150,210,270,330,30}
    \fill (0,1) ++(\ang:1.3) circle (2.5pt);
\end{scope}

%========================
% Picture 1: torus fundamental domain
%========================
\begin{scope}[xshift=4.8cm]
  \node at (3,5.3) {\large $\mathbb{R}^2/\langle (6,0),(0,4)\rangle$};

  % grid
  \draw[step=1, gray!50, dashed] (0,0) grid (6,4);

  % rectangle
  \draw[thick] (0,0) rectangle (6,4);

  % lattice points
  \foreach \x in {0,1,2,3,4,5,6}
    \foreach \y in {0,1,2,3,4}
      \fill (\x,\y) circle (2.2pt);

  % labels
  \node[below left] at (0,0) {$(0,0)$};
  \node[below right] at (6,0) {$(6,0)$};
  \node[above left] at (0,4) {$(0,4)$};
  \node[above right] at (6,4) {$(6,4)$};

  % edge identifications
  \draw[->, very thick, blue] (0.7,4.25) -- (5.3,4.25);
  \draw[->, very thick, blue] (0.7,-0.25) -- (5.3,-0.25);

  \draw[->, very thick, red] (-0.25,0.5) -- (-0.25,3.5);
  \draw[->, very thick, red] (6.25,0.5) -- (6.25,3.5);

  \node[blue] at (3,4.55) {\small identify};
  \node[red, rotate=90] at (6.55,2) {\small identify};
\end{scope}

%========================
% Picture 3: circle with four lattice points
%========================
\begin{scope}[xshift=14.5cm]
  \node at (0,3.3) {\large $\mathbb{R}/4\mathbb{Z}$};

  \draw[thick] (0,1) circle (1.3);

  \foreach \ang in {90,180,270,0}
    \fill (0,1) ++(\ang:1.3) circle (2.5pt);
\end{scope}

\end{tikzpicture}
        \caption{In this caption the phrase a ray of $\sigma$ refers to a map from $\mathbb{N}$ to $\sigma$. Rays are determined by the image of $1$. Fix $\alpha =3 , \beta = 2.$ Integral affine manifolds associated to the pullback of the tropical abelian variety of Example \ref{ex:TropAV1}, from left to right, along the inclusion of the rays $(1,0)$, $(1,1)$ and $(0,1)$. The black dots are integral points of each integral affine manifold. The number of lattice points visible in the left hand picture would double if $(1,0)$ were replaced with $(2,0)$. Replacing the ray $(1,1)$ with an arbitrary ray $(p,q)$ would rescale the number of lattice points in the $x$ and $y$ directions in the central figure by a factor of $p$ or $q$ respectively. In the limit of $q$ becoming large with fixed $p=1$ this diagram asymptotes to the circle on the left. Reversing the roles of $p$ and $q$ causes the central diagram to asymptote to the circle on the right.}
        \label{fig:TropAV1}
    \end{figure}
\end{example}

\subsection{Integral affine manifolds for generalised Kummer varieties, and their singularities} 
We now describe the integral affine manifold associated to $\operatorname{Kum}_{2n}(A)$ with $A$ as in Example \ref{ex:AbSurface1}. To do so we must apply the tropical generalised Kummer construction to the family of tropical abelian varieties in Example \ref{ex:TropAV1}.

\subsubsection{The surface $n=1$ situation} For $n=1$ the output is a family of integral affine manifolds over $\sigma$ which, as with the family of tropical abelian surfaces, is of total dimension $4$ and thus not depicted. Instead Figure \ref{fig:TropGenKummerFibres} follows the pattern of Figure \ref{fig:TropAV1} by drawing the integral affine manifolds associated via the construction of Section \ref{sec:TropGenKummerIAMSing} to pulling back $\operatorname{Kum}_2(\mathpzc{A})$ along rays $(1,0)$, $(1,1)$, and $(0,1)$ within $\sigma$.

We first explain how to compute these integral affine manifolds, starting with the type~II situation. The moduli space of two points on a circle which sum (in the group law on the circle) to the origin is homeomorphic to the interval. Indeed, identifying our circle with the quotient of the interval $[0,1]$ by identifying end points, any two points with sum zero mod 1 may be expressed as $(\alpha, 1-\alpha)$. Such an unordered tuple corresponds to a unique point $\alpha$ in $[0,0.5]$. There are two singular points to our integral affine structure: the endpoints of the interval where $\alpha=0$ and $\alpha = 0.5$. Thinking of the interval as the moduli of two points on a circle, this is the locus where both points collide. See left and right of Figure \ref{fig:TropGenKummerFibres}.

Now we describe the type~III locus. We are tasked with studying the locus of points summing to zero within the moduli space of two points in the torus as depicted in the centre of Figure \ref{fig:TropAV1}. To put coordinates on this space, we track the position of whichever point is further left, and if both points tie then the point which is lower down. Thus, a fundamental domain is given by the shaded region in Figure \ref{figure:fundamentaldomain}. The solid and dashed blue lines are identified; similarly the solid and dashed purple lines are identified. Thus we obtain a sphere. The sphere has four singular points, corresponding to where the two tropical points coincide. On our coordinate square, these are the corners of the blue rectangle (contributing two points), and the two points where a solid and dashed line meet. 

\begin{figure}
\centering
\begin{tikzpicture}

\begin{scope}[xshift=4.5cm]
  \node at (1.5,3.3) {\large $\mathbb{R}^2/\langle (3,0),(0,2)\rangle$};

  % shaded left half
  \fill[blue!25] (0,0) rectangle (1.5,2);

  % outer rectangle
  \draw[thick] (0,0) -- (3,0) -- (3,2) -- (0,2);
  \draw[purple, very thick, dashed] (0,2) -- (0,1);
  \draw[purple, very thick] (0,1) -- (0,0);

  % middle vertical line: top dashed, bottom solid
  \draw[blue, very thick, dashed] (1.5,2) -- (1.5,1);
  \draw[blue, very thick] (1.5,1) -- (1.5,0);

  % labels
  \node[below left] at (0,0) {$(0,0)$};
  \node[below right] at (3,0) {$(3,0)$};
  \node[above left] at (0,2) {$(0,2)$};
  \node[above right] at (3,2) {$(3,2)$};

  % edge identifications
  \draw[->, very thick, blue] (0.4,2.25) -- (2.6,2.25);
  \draw[->, very thick, blue] (0.4,-0.25) -- (2.6,-0.25);

  \draw[->, very thick, red] (-0.25,0.3) -- (-0.25,1.7);
  \draw[->, very thick, red] (3.25,0.3) -- (3.25,1.7);

  \node[blue] at (1.5,2.55) {\small identify};
  \node[blue] at (1.5,-0.45) {\small identify};
  \node[red, rotate=90] at (-0.55,1) {\small identify};
  \node[red, rotate=90] at (3.55,1) {\small identify};
\end{scope}
\end{tikzpicture}
\caption{In blue is the fundamental domain for $n=1$ type~III case, depicted as a subset of the associated tropical abelian variety. To form a sphere we identify dashed with solid blue (resp. purple) lines.}\label{figure:fundamentaldomain}
\end{figure}
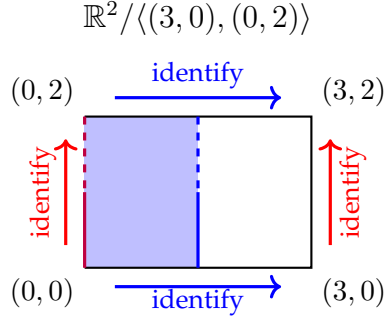

\begin{figure}
\centering
\begin{tikzpicture}[scale=1.0, line cap=round, line join=round]

  % --- LEFT: horizontal line with red endpoints ---
  \begin{scope}[xshift=-4cm]
    \draw[thick] (-1.5,0) -- (1.5,0);
    \fill[red] (-1.5,0) circle (2.5pt);
    \fill[red] (1.5,0) circle (2.5pt);
  \end{scope}

  % --- CENTER: slightly smaller sphere ---
  \begin{scope}
    % sphere
    \shade[ball color=gray!15, opacity=0.9] (0,0) circle (1.6);
    \draw[thick] (0,0) circle (1.6);

    % equator
    \draw[thick] (-1.6,0) arc[start angle=180,end angle=360,x radius=1.6,y radius=0.55];
    \draw[dashed] (1.6,0) arc[start angle=0,end angle=180,x radius=1.6,y radius=0.55];

    % four red points
    \fill[red] (-0.95,0.6) circle (2.5pt);
    \fill[red] (0.95,0.6) circle (2.5pt);
    \fill[red] (-0.95,-0.6) circle (2.5pt);
    \fill[red] (0.95,-0.6) circle (2.5pt);
  \end{scope}

  % --- RIGHT: vertical line with red endpoints ---
  \begin{scope}[xshift=4cm]
    \draw[thick] (0,-1.5) -- (0,1.5);
    \fill[red] (0,-1.5) circle (2.5pt);
    \fill[red] (0,1.5) circle (2.5pt);
  \end{scope}
\end{tikzpicture}
\caption{The integral affine manifold associated (from left to right) with the rays $(1,0)$, $(1,1)$, and $(0,1)$ of $\sigma$. As in Figure \ref{fig:TropAV1}, we imagine a family of integral affine manifolds parameterised by ray directions $(p,q)$. Fixing $q=1$ and enlarging $p$ fixes the distance between the pairs of red dots on left and right. However, the distance between the left and right pairs becomes large. Thus asymptotically the sphere collapses to the left hand diagram. Conversely, fixing $p=1$ and making $q$ large keeps horizontal distances constant and enlarges vertical distances.}\label{fig:TropGenKummerFibres}
\end{figure}
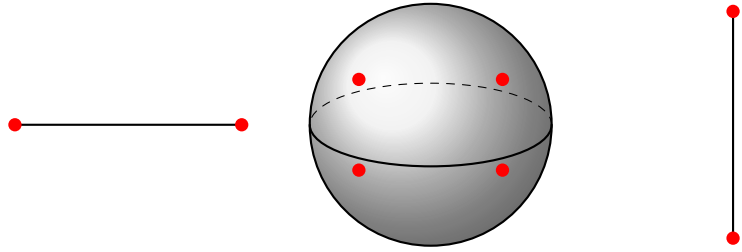
\subsubsection{The fourfold $n=2$ situation.} In the type~II situation we are led to study moduli of three points on $ S^1$ which sum to zero. The result is a 2-simplex with singularities along the entire boundary, as depicted on the left and right of Figure \ref{fig:n=3 singular loci}. The red nodes depict distinguished points within the singular locus, which are in some sense especially singular - this special singularity means there is no torus direction in the irreducible component associated to this vertex in any model of our degeneration.

In the type~III case we have the the moduli of three points on an integral affine manifold homeomorphic to a torus which sum to zero in its group law. This integral affine manifold has dimension four, and thus the total space of our family is six. We are unable to depict a four dimensional space: instead we draw the two dimensional singular locus in the centre of Figure \ref{fig:n=3 singular loci}. 

We explain how to identify the topology of this singular locus with that of $( S^1)^2$. The singular locus is the moduli space of three points in a torus which sum to zero, and which have the property that at least two points coincide. To identify with a torus note that for each singular point specifying an unordered tuple $(p_1,p_2,p_3)$ with $p_i$ in $ S^1\times  S^1$, there is exactly one point $p$ of $ S^1\times  S^1$ such that $p$ appears in the tuple $(p_1,p_2,p_3)$ at least twice.

Within the singular locus there are nine distinguished points corresponding to tuples of the form $(p,p,p)\in ( S^1\times  S^1)^3$. These tuples are associated to the nine 3-torsion points of our tropical abelian surface.

\begin{figure}
\begin{tikzpicture}[scale=1.0, line cap=round, line join=round]

  % --- LEFT TRIANGLE ---
  \begin{scope}[shift={(0,0)}]
    \coordinate (A) at (0,0);
    \coordinate (B) at (2.4,0);
    \coordinate (C) at (1.2,2.1);

    \draw[red, very thick] (A) -- (B) -- (C) -- cycle;
    \fill[red] (A) circle (2.8pt);
    \fill[red] (B) circle (2.8pt);
    \fill[red] (C) circle (2.8pt);
  \end{scope}

  % --- MIDDLE TORUS AS IDENTIFICATION SQUARE ---
  \begin{scope}[shift={(4.5,-0.2)}]
    % shaded square
    \fill[red!15] (0,0) rectangle (3,3);
    \draw[red, very thick] (0,0) rectangle (3,3);

    % 9 distinguished points
    \foreach \x in {0.5,1.5,2.5} {
      \foreach \y in {0.5,1.5,2.5} {
        \fill[red] (\x,\y) circle (2.8pt);
      }
    }

    % horizontal identifications
    \draw[->, blue, thick] (0.4,3.25) -- (2.6,3.25);
    \draw[->, blue, thick] (0.4,-0.25) -- (2.6,-0.25);

    % vertical identifications
    \draw[->, thick] (-0.25,0.4) -- (-0.25,2.6);
    \draw[->, thick] (3.25,0.4) -- (3.25,2.6);
  \end{scope}

  % --- RIGHT TRIANGLE ---
  \begin{scope}[shift={(10.0,0)}]
    \coordinate (A) at (0,0);
    \coordinate (B) at (2.4,0);
    \coordinate (C) at (1.2,2.1);

    \draw[red, very thick] (A) -- (B) -- (C) -- cycle;
    \fill[red] (A) circle (2.8pt);
    \fill[red] (B) circle (2.8pt);
    \fill[red] (C) circle (2.8pt);
  \end{scope}
\end{tikzpicture}
\caption{In red we display singular loci for the integral affine manifolds associated to the $n=2$ case over rays (from left to right) $(1,0)$, $(1,1)$ and $(0,1)$. In the type~II case (left and right) the integral affine manifold is a two simplex with boundary shown in red. In the type~III case, the integral affine manifold is homeomorphic to a four dimensional space ($\mathbb{CP}^2$), and thus only the singular locus is depicted.}\label{fig:n=3 singular loci}
\end{figure}
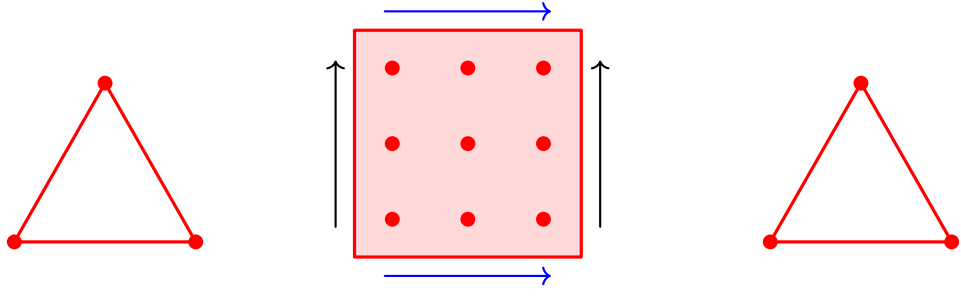

\begin{figure}
\begin{tikzcd}[column sep=3em, row sep=3em]
% top left circle
\begin{tikzpicture}[scale=0.7, baseline=-0.5ex]
  \draw (0,0) circle (0.8);
  \fill[blue] (35:0.8) circle (3pt);
  \fill[blue] (-35:0.8) circle (3pt);
  \node at (0.95,0.35) {$1$};
  \node at (0.95,-0.6) {$2$};
\end{tikzpicture}
\arrow[rd, green!60!black]
&
% top middle circle
\begin{tikzpicture}[scale=0.7, baseline=-0.5ex]
  \draw (0,0) circle (0.8);
  \fill[blue] (25:0.8) circle (3pt);
  \node at (1.0,0.2) {$3$};
\end{tikzpicture}
\arrow[d, green!60!black]
&
% top right circle
\begin{tikzpicture}[scale=0.7, baseline=-0.5ex]
  \draw (0,0) circle (0.8);
  \fill[blue] (140:0.8) circle (3pt);
  \fill[blue] (220:0.8) circle (3pt);
  \node at (-1.0,-0.45) {$1$};
  \node at (-0.8,0.7) {$2$};
\end{tikzpicture}
\arrow[ld, green!60!black]
\\
% bottom-left circle
\begin{tikzpicture}[scale=0.7, baseline=-0.5ex]
  \draw (0,0) circle (0.8);
  \fill[blue] (270:0.8) circle (3pt);
  \node at (0,-1.5) {$3$};
\end{tikzpicture}
\arrow[r, green!60!black]
&
% middle triangle
\begin{tikzpicture}[scale=1.0, baseline=-0.5ex]
  \coordinate (A) at (0,0);
  \coordinate (B) at (2.4,0);
  \coordinate (C) at (1.2,2.1);
  \draw[thick] (A)--(B)--(C)--cycle;
  \fill (A) circle (2pt);
  \fill (B) circle (2pt);
  \fill (C) circle (2pt);
\end{tikzpicture}
&
% middle-right circle
\begin{tikzpicture}[scale=0.7, baseline=-0.5ex]
  \draw (0,0) circle (0.8);
  \fill[blue] (180:0.8) circle (3pt);
  \node at (-1.2,-0.2) {$3$};
\end{tikzpicture}
\arrow[l, green!60!black]
\\
&
% bottom-middle circle
\begin{tikzpicture}[scale=0.7, baseline=-0.5ex]
  \draw (0,0) circle (0.8);
  \fill[blue] (240:0.8) circle (3pt);
  \fill[blue] (300:0.8) circle (3pt);
  \node at (-0.7,-1) {$2$};
  \node at (0.7,-1) {$1$};
\end{tikzpicture}
\arrow[u, green!60!black]
&
\end{tikzcd}
\caption{The dual intersection complex in the type~II, $n=2$ case is the moduli space of three points on $ S^1$ which sum to zero. Combinatorial flatness enforces that the collision type is fixed on each face. The open face corresponds to three non-intersecting points. Above, we show how each other face corresponds to a different collision type. The green arrow shows the face which parameterises the collision type depicted in the arrow's source.}\label{fig:DualIntCx}
\end{figure}
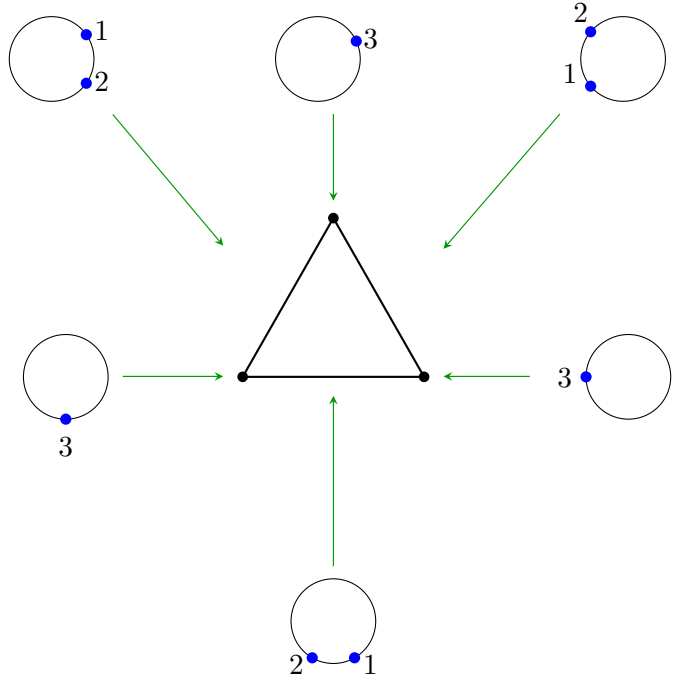

\subsection{Tropical models for one-parameter degenerations} Let $\Delta$ be a curve smooth over $\mathbb{C}$ with divisorial logarithmic structure from a closed point $\Delta_0$. Let $A\rightarrow \Delta$ be a logarithmic abelian surface. In this section we will name a choice of cone complex structure on the family of tropical abelian varieties $\mathpzc{A}\rightarrow \mathbb{N}\cong \rho$.

\subsubsection{Polyhedral structures for type~II degeneration} 
Let $A\rightarrow \Delta$ be a type~II degeneration of logarithmic abelian varieties (type II means that the tropical abelian variety $\mathpzc{A} \rightarrow \rho$ is of relative dimension one over the ray $\rho$). For example, pull back the family of Example \ref{ex:AbSurface1} along a curve $\mathbb{A}^1 = \Delta\rightarrow \mathbb{A}^2$ which meets the $x$ axis in one point and does not meet the $y$ axis. Further assume that $\mathpzc{A}\cong \mathpzc{A}(\varphi)$ where $$\varphi\colon \mathbb{Z}\longrightarrow \mathsf{Hom}(\mathbb{Z},\mathsf{Hom}(\rho,\mathbb{Z})),\qquad \varphi(m)(a)(t)=\alpha mat,\qquad \alpha\in\mathbb{Z}_{>0}.$$ Set $n$ a positive integer.

We define a morphism of monoids 
$$A_\alpha
:=
\left\{
(d,g_1,\ldots,g_{n+1})\in\mathbb N^{n+2}
\;\middle|\;
\sum_{i=1}^{n+1}g_i=\alpha d,\quad
\sum_{i=1}^{n+1}i\,g_i\equiv0\pmod{n+1}
\right\}
\rightarrow \rho:=\mathbb N,$$ by projection to the first factor.

\begin{lemma}\label{lem:IAmRep}
    The functor of points of $A_\alpha$ on $\textbf{RPC}_\rho$ represents $\operatorname{Kum}_{2n}(\mathpzc{A})$.
\end{lemma}
The proof strategy is to note that a version of the argument in \cite[Top of page 351]{Morton_1967} works in families. The geometry of this proof is reflected on the link over $\tau$, see Figure \ref{fig:DualIntCx} for the $n=2$ case.

\begin{proof}
    A map from a test cone $\tau$ to $\operatorname{Kum}_{2n}(\mathpzc{A})$ is $n+1$ unordered sections $(s_1,...,s_{n+1})$ of $$\mathpzc{A}_\tau:=\mathpzc{A}\times_\rho \tau\rightarrow \tau$$ which do not meet over the interior of $\tau$ unless they coincide. If the image of (the interior of and thus the whole of) $\tau$ is $0$ in $\rho$ then the only map from $\tau$ to both $\operatorname{Kum}_{2n}(\mathpzc{A})$ and $A_\alpha$ is the zero map. 
    
    If the interior of $\tau$ is not mapped to $0$ in $\rho$ then we proceed as follows. 
    %First we specify a cyclic ordering: over the interior $|\tau^\circ|$ of the topological realisation $\tau\otimes \mathbb{R}$ of $\tau$, the link of $|\mathpzc{A}|\rightarrow |\tau^\circ|$ is the trivial $\mathbb{S}^1$ bundle. Choosing any point in $|\tau^\circ|$, our sections specify $n+1$ points $(\overline{q}_1,...,\overline{q}_{n+1})$ on a circle, and thus a cyclic ordering on sections. There is a unique lift to points $({q}_1,...,{q}_{n+1})$ in $\mathbb{R}$ satisfying the following hypotheses.
    %\begin{enumerate}
    %    \item The difference between the smallest and largest section is less than the length $L$ of our circle.\footnote{The length of the circle at a point $p$ in $\tau^\circ$ is obtained by multiplying the image of this point in $|\rho|= \mathbb{N}\otimes \mathbb{R}$ by $\alpha$.} 
    %    \item The $q_i$ sum to $0$ in $\mathbb{R}$.
    %\end{enumerate}
    %A lift satisfying both properties is possible: choose any lift satisfying the first property. Since the $\overline{q}_i$ sum to $0$ we know that the sum of $q_i$ is $kL$ for some integer $k$. Now update the $q_i$ by subtracting $L$ from the largest $k$ elements. Since $q_i$ are canonically in bijection with the sections $s_i$, we have ordered the sections $s_i$. Since sections can not meet over interior of $\tau$ by hypothesis, this ordering is independent of point in $\tau^\circ$. The construction is clearly functorial in $\tau$.
    Each $s_i$ is an equivalence class of elements in $\mathsf{Hom}(\mathbb{Z},\mathsf{Hom}(\tau,\mathbb{Z}))= \mathsf{Hom}(\tau,\mathbb{Z}).$ 
    The equivalence relation identifies two elements where their difference lies in the image of $\varphi$. Choose lifts $\hat{s}_1,...\hat{s}_{n+1}$ of the sections $s_i$ to $\mathsf{Hom}(\mathbb{Z},\mathsf{Hom}(\tau,\mathbb{Z}))$ satisfying the following properties: 
    \begin{enumerate}
        \item Considered functions on $\tau$, for any $i,j$ we have $|\hat{s}_{i}-\hat{s}_j|\leq|\varphi(1)|$ at each interior point of $\tau$.
        \item In the group law of $\mathsf{Hom}(\tau,\mathbb{Z})$ there is an equality $\sum_i\hat{s}_i = 0$.
    \end{enumerate} 
    Fix a point $t$ in the interior of $\tau$. The first equality can be ensured at $t$ by repeatedly subtracting $\varphi(1)$ from the largest representative. Once true at $t$ it is true for all points in the interior of $\tau$ because the original sections $s_i$ were assumed not to meet over the interior of $|\tau|$. Once condition (1) is achieved, condition (2) can be made to occur at the same time: the sum of lifts satisfying the first condition can be written $m\varphi$ for some $m$ because the $s_i$ sum to zero in the group law of $\mathpzc{A}$. If $m>0$ then choose $j$ such that $\hat{s}_j(t)$ is maximal, and replace $\hat{s}_j$ by $\hat{s}_j-\varphi(1)$. This decreases $m$ by $1$: repeat until $m=0$. If $m<0$ then add $m\varphi(1)$ to each $\hat{s}_i$ and we are in the case $m>0$. In the sequel we consider $\hat{s}_i$, and thus $s_i$ to be ordered by their value on $t$, and note this ordering is independent of the choice of $t$. Relabelling, we assume the ordering is $s_1\leq s_2\leq ...\leq s_{n+1}$
    
    We now define elements of $\operatorname{Hom}(\tau,\mathbb{N})$, $$g_i = \hat{s}_{i+1}-\hat{s}_i\text{ for }1\leq i \leq n \text{ and } g_{n+1} = \varphi(1) + \hat{s}_1-\hat{s}_{n+1}.$$ The $g_i$ together specify a map $\tau \rightarrow \mathbb{N}^{n+1}$. Combining with the map from $\tau$ to $\rho$ specifies a map $$\tau\rightarrow \mathbb{N}^{n+2}.$$
    Note that summing the $g_i$ gives $\varphi(1)$, which ensures that the image satisfies the first equality in the definition of $A_\alpha$. Also we can write $$\hat{s}_i = \hat{s}_1 + \sum_{j<i} g_j \text{ for }i<n+1$$ and thus the summing to zero condition ensures that $$(n+1)\hat{s}_1 = \sum_i (i-n-1)g_i.$$ Taken modulo $n+1$ this implies the second equality. Together the last two displayed equations demonstrate that $\hat{s}_i$, and thus the map from $\tau$ to $\operatorname{Kum}_{2n}(\mathpzc{A})$, can be recovered from the map from $\tau$ to $A$. We have thus constructed our bijection. Every step was functorial in $\tau$, and thus the result is proved. 
\end{proof}

\begin{corollary}
    \begin{enumerate}[leftmargin = *]
        \item The map $\operatorname{Kum}_{2n}(\mathpzc{A})\rightarrow \rho $ is represented by an integral morphism of cone complexes whose link is, as a topological space, an n-simplex. The vertices of the polyhedral structure on this simplex are precisely vertices of the simplex. 
        \item The map $\operatorname{Kum}_{2n}(A)\rightarrow \Delta$ is a vertical, proper, and integral morphism of DM stacks with logarithmic structure. 
    \end{enumerate}
\end{corollary}
Note that $A_\alpha$ is not a free monoid, and thus Theorem \ref{mainthm:Explicit} does not apply without performing a Kummer logarithmic alteration of $\operatorname{Kum}_{2n}(A)$ i.e. performing toric semistable reduction to our morphism of cone complexes.
\begin{proof}
    Note (2) follows from (1) by Theorem~\ref{thm:GettingModels}. 
    For (1) we know that $\rho$ is represented by a cone complex (it is a ray).
    By Lemma \ref{lem:IAmRep}, $\operatorname{Kum}_{2n}(\mathpzc{A})$ is a cone complex. Integrality is the same as asking that the associated morphism of toric varieties is flat. Since all toric varieties are Cohen--Macaulay, and the toric variety $\mathbb{A}^1$ associated to $\rho$ is regular, this flatness property is immediate from equidimensional fibres by miracle flatness. 
\end{proof}

\begin{example}[Lattice point subdivision]\label{ex:LP}
    Let $\alpha = 8$. We now specialise to the $n=2$ (fourfold) case, where a degeneration was already constructed in \cite{halle2026degenerationsgeneralizedkummervarieties}. 
    The preimage of $1$ in $\rho$ inside the monoid $A_8$ consists of 15 points\footnote{Note not every ray contributes a point - we have not saturated and so this is not necessary.} on the link as depicted by red crosses in Figure \ref{fig:hhz}.

    The reader is entitled to choose a logarithmic alteration $\mathpzc{A}_{\operatorname{LP}}\rightarrow \mathpzc{A}$ which introduces rays which pass through every integral point in the preimage of $1$ in $\mathpzc{A}\rightarrow \rho$. The resulting logarithmic alteration $A_{\operatorname{LP}} = A\times_\mathpzc{A}\mathpzc{A}_{\operatorname{LP}}\rightarrow A$ is semistable over $\Delta$ and its central fibre is a cycle of $\alpha=8$ ruled surfaces. 

    The inclusion of the universal tropical subscheme into $\mathpzc{A}^{[3]}\times \mathpzc{A}$ together with the projection map define morphisms: $$\mathpzc{Z}_3\rightarrow \mathpzc{A}^{[3]}\times \mathpzc{A} \rightarrow \mathpzc{A}$$
    Taking the common refinement of the cone complex structure pulled back from $\mathpzc{A}_{\operatorname{LP}}\rightarrow \mathpzc{A}$ with the locally closed stratification of $\mathpzc{A}^{[3]}\times \mathpzc{A}$ induced by $\mathpzc{Z}_3$ defines a subdivision of piecewise linear spaces  $\Sigma'\rightarrow \mathpzc{A}^{[3]}\times \mathpzc{A}$. Pushing forward the induced locally closed stratification of $\mathpzc{A}^{[3]}\times \mathpzc{A}$ induced by $\Sigma'$ to $\mathpzc{A}^{[3]}$ specifies a tropical alteration $\Sigma\rightarrow \mathpzc{A}^{[3]}$, and thus a subdivision of $\operatorname{Kum}_{2n}(\mathpzc{A})$. The link of this subdivision is depicted in Figure \ref{fig:hhz}.
\end{example}

\begin{example}[{Recovering combinatorics of \cite{halle2026degenerationsgeneralizedkummervarieties}}]\label{ex:HHZ}
    The dual intersection complex of the Type~II degenerations constructed in \cite{halle2026degenerationsgeneralizedkummervarieties} can naturally be considered a refinement of the subdivision described in Example \ref{ex:LP}. This refinement is shown as light grey with lines dashed in Figure \ref{fig:hhz}.

    Rational points of the link in Figure \ref{fig:hhz} biject with unordered tuples of three points $k_1,k_2,k_3\in \mathbb{Q}/8\mathbb{Z}$ which sum to $0$ mod $8$. Each point $k_i$ may be written in the form $$2a_i + \epsilon_i d_i \text{ where }a_i\in\{0,1,2,3\}, \quad \epsilon_i\in\{\pm 1\}, \text{ and } d_i \in [0,1].$$ The grey lines subdivide the simplex such that on the interior of each face the order of $d_i$ is constant.
\end{example}

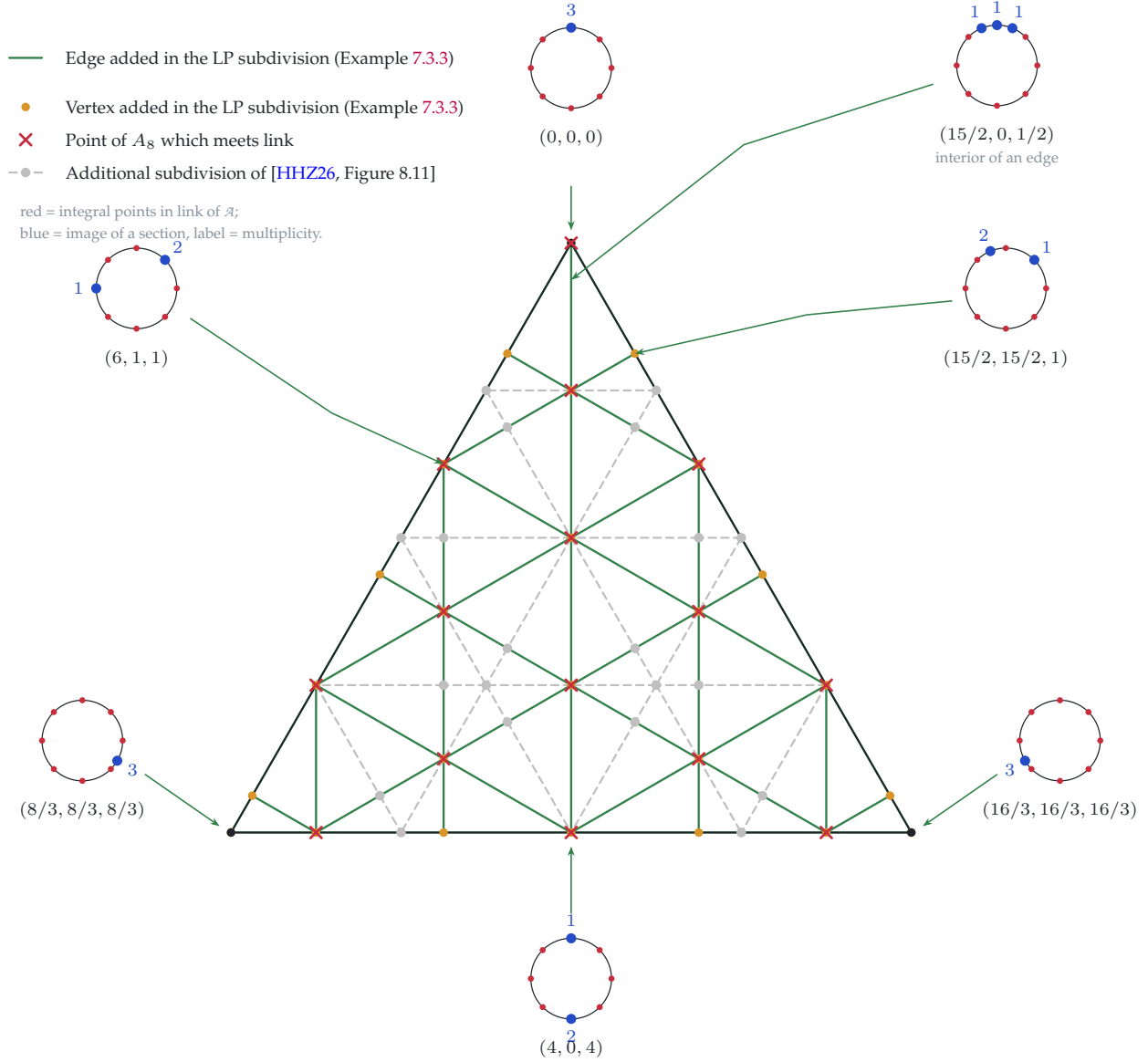
\begin{figure}[p]
\centering

\definecolor{hhzBlack}{HTML}{20242A}
\definecolor{hhzGrey}{HTML}{85909B}
\definecolor{hhzRed}{HTML}{C82D3D}
\definecolor{hhzGold}{HTML}{D99427}
\definecolor{hhzBlue}{HTML}{244BC5}
\definecolor{hhzGreen}{HTML}{32814B}
\definecolor{hhzPurple}{HTML}{8A40B5}

\tikzset{
  hhz wall/.style={draw=hhzGreen,line width=.85pt},
  hhz expansion/.style={
    draw=black!25,line width=.8pt,
    dash pattern=on 3pt off 2pt
  },
  hhz narrow/.style={draw=hhzGreen,line width=.85pt},
  hhz boundary/.style={draw=hhzBlack,line width=.7pt},
  hhz arrow/.style={
    draw=hhzGreen,line width=.575pt,
    -{Stealth[length=3.8pt,width=2.7pt]}
  },
  hhz label/.style={
    font=\fontsize{7}{8}\selectfont,text=hhzBlack
  },
  hhz small/.style={
    font=\fontsize{6}{7}\selectfont,text=hhzGrey
  }
}

% Red lattice cross at a named coordinate.
\def\LPcross#1{%
  \draw[hhzRed,line width=1pt]
    ([xshift=-2.2pt,yshift=-2.2pt]#1)
      --([xshift=2.2pt,yshift=2.2pt]#1)
    ([xshift=-2.2pt,yshift=2.2pt]#1)
      --([xshift=2.2pt,yshift=-2.2pt]#1);
}

% Circle configuration:
% #1,#2 = centre;
% #3 = position/multiplicity pairs;
% #4 = tuple; #5 = description.
% Positions lie in R/8Z, clockwise from 0 at the top.
\def\LPcircle#1#2#3#4#5{%
  \begin{scope}[shift={(#1,#2)}]
    \draw[hhzBlack,line width=.45pt]
      (0,0) circle[radius=16pt];

    \foreach \k in {0,...,7}{
      \fill[hhzRed]
        ({90-45*\k}:16pt) circle[radius=1.2pt];
    }

    \foreach \position/\multiplicity in {#3}{
      \fill[hhzBlue]
        ({90-45*(\position)}:16pt) circle[radius=2pt];
      \node[
        font=\fontsize{7}{8}\selectfont,
        text=hhzBlue,inner sep=0pt
      ] at ({90-45*(\position)}:23pt)
        {$\multiplicity$};
    }

    \node[hhz label,anchor=base]
      at (0,-59) {$#4$};
    \node[hhz small,anchor=base]
      at (0,-77) {#5};
  \end{scope}
}

\resizebox{\linewidth}{!}{%
\begin{tikzpicture}[
  x=.5pt,y=.5pt,line cap=round,line join=round
]

% Vertices in gap coordinates (g_1,g_2), with g_3=8-g_1-g_2.
\foreach \u/\v [count=\j from 0] in {
  0/0,0/1.5,0/2,0/3,0/4,0/4.5,0/6,0/7.5,0/8,
  .5/2,.5/3.5,
  1/1,1/4,1/7,
  1.5/0,1.5/4.5,1.5/6,
  2/0,2/.5,2/2,2/3.5,2/4,2/5,2/6,
  2.5/4,2.5/5.5,
  3/0,3/3,
  3.5/.5,3.5/2,
  4/0,4/1,4/2,4/2.5,4/4,
  4.5/0,4.5/1.5,
  5/2,5.5/2.5,
  6/0,6/1.5,6/2,
  7/1,7.5/0,8/0%
}{
  \coordinate (v\j) at
    ({500+270*(\v-\u)/8},{670-468*(\u+\v)/8});
}

% Solid subdivision edges: extra expansion edges are excluded.
\foreach \a/\b in {
  0/1,0/14,1/2,1/11,2/3,3/4,3/9,3/10,
  4/5,5/6,5/12,6/7,7/8,7/13,8/13,9/11,
  10/12,11/14,11/18,12/15,12/20,13/16,13/23,14/17,
  15/22,16/22,17/26,18/26,20/27,22/24,22/25,23/25,
  24/27,25/34,26/28,26/30,27/29,27/33,28/31,29/31,
  30/35,31/35,31/36,33/37,34/38,35/39,36/37,37/38,
  37/40,38/41,39/43,40/42,41/42,42/43,42/44,43/44%
}{
  \draw[hhz wall] (v\a)--(v\b);
}

% Additional equal-expansion-level walls.
\foreach \a/\b in {
  2/9,2/11,4/10,4/12,6/15,6/16,9/19,10/19,
  11/17,12/21,15/21,16/23,17/18,18/19,19/20,19/28,
  19/29,20/21,21/22,21/24,21/27,22/23,24/34,27/32,
  28/30,29/32,30/31,31/32,32/33,32/36,32/37,33/34,
  36/39,37/41,39/40,40/41%
}{
  \draw[hhz expansion] (v\a)--(v\b);
}

% Edges corresponding to narrow strata.
\foreach \a/\b in {
  0/11,3/19,6/12,6/13,6/22,11/19,12/19,19/26,
  19/27,19/31,22/34,27/34,31/39,34/37,37/39,39/42%
}{
  \draw[hhz narrow] (v\a)--(v\b);
}

% Original boundary and vertices.
\draw[hhz boundary] (v0)--(v44)--(v8)--cycle;

\foreach \i in {0,...,44}{
  \fill[hhzGold] (v\i) circle[radius=1.7pt];
}
\foreach \i in {0,44,8}{
  \fill[hhzBlack] (v\i) circle[radius=1.7pt];
}

% Vertices introduced by the additional expansion walls.
\foreach \i in {
  2,4,9,10,15,16,17,18,20,21,
  23,24,28,29,30,32,33,36,40,41%
}{
  \fill[black!25] (v\i) circle[radius=1.9pt];
}

% The 15 height-one lattice points of A_8.
\foreach \i in {
  0,3,6,11,12,13,19,22,26,27,31,34,37,39,42%
}{
  \LPcross{v\i}
}

% Circle configurations and arrows.
\LPcircle{500}{809}
  {0/3}{(0,0,0)}{}
\draw[hhz arrow] (500,715)--(500,680);

\LPcircle{155}{634}
  {6/1,1/2}{(6,1,1)}{}
\draw[hhz arrow] (198,610)--(310,535)--(v26);

\LPcircle{845}{634}
  {7.5/2,1/1}{(15/2,15/2,1)}{}
\draw[hhz arrow] (802,624)--(687,613)--(v1);

\LPcircle{112}{275}
  {{8/3}/3}{(8/3,8/3,8/3)}{}
\draw[hhz arrow] (162,248)--(221,207);

\LPcircle{888}{275}
  {{16/3}/3}{(16/3,16/3,16/3)}{}
\draw[hhz arrow] (838,248)--(779,207);

\LPcircle{500}{86}
  {4/2,0/1}{(4,0,4)}{}
\draw[hhz arrow] (500,138)--(500,191);

\LPcircle{838}{811}
  {7.5/1,0/1,.5/1}{(15/2,0,1/2)}
  {interior of an edge}
\draw[hhz arrow] (787,796)--(637,748)--(500,640.75);

\draw[hhz narrow] (54,817)--(80,817); \node[hhz label,anchor=base west] at (91,812) {Edge added in the LP subdivision (Example \ref{ex:LP})}; \fill[hhzGold] (67,778) circle[radius=1.7pt];
\node[hhz label,anchor=base west]
  at (91,773)
  {Vertex added in the LP subdivision (Example \ref{ex:LP})};

\coordinate (keycross) at (67,752);
\LPcross{keycross}
\node[hhz label,anchor=base west]
  at (91,747) {Point of $A_8$ which meets link};

\draw[hhz expansion] (54,726)--(80,726);
\fill[black!25] (67,726) circle[radius=1.9pt];
\node[hhz label,anchor=base west]
  at (91,721) {Additional subdivision of \cite[Figure 8.11]{halle2026degenerationsgeneralizedkummervarieties}};

\node[hhz small,anchor=base west]
  at (55,691) {red = integral points in link of $\mathpzc{A}$;};
\node[hhz small,anchor=base west]
  at (55,674) {blue = image of a section, label = multiplicity.};

\end{tikzpicture}%
}

\caption{Let $\alpha=8$. Ignoring grey dots and dashed grey lines, the diagram shows the type II n=2 (fourfold) dual intersection complex with the LP subdivision described in Example \ref{ex:LP}. 
The gold dots indicate vertices which the construction in Example \ref{ex:LP} introduces. 
A red cross superimposed on a gold vertex indicates a height-one lattice point of $A_8$. 
Compare to \cite[Figure 8.11]{halle2026degenerationsgeneralizedkummervarieties}.
Diagrams around the outside show how the combinatorial type of three points with sum zero (in blue) on the link of $\mathpzc{A}$ gives rise to different strata. 
In terms of these surrounding pictures, the gold vertices arise and solid edges between them occur where a blue point meets one of the red points in the link of $\mathpzc{A}/\rho$ associated to the tropical model $\mathpzc{A}_{\operatorname{LP}}\rightarrow \mathpzc{A}$;
the blue labels indicate multiplicities. 
The red crosses (integral points) occur where all blue points lie on a red dot (i.e. when introducing a primitive ray in this direction on the tropicalisation gives rise to a saturated map from the introduced ray to $\rho$).}
\label{fig:hhz}
\end{figure}
\subsubsection{Polyhedral structures for type~III degeneration} In the type~III situation $\operatorname{Kum}_{2n}(\mathpzc{A})$ is not represented by a monoid. We now develop the language required to name a choice of tropical model. The first step is constructing a subdivision of $\mathpzc{A}\times \mathpzc{A}^{[n]}$. We will push-forward this cone complex structure to define a tropical model of $\mathpzc{A}^{[n+1]}$, and thus of $\operatorname{Kum}_{2n}(\mathpzc{A})$. The polyhedral structure we name adapts the approaches of \cite{Nabijou2024ToricConfigurationSpaces} and \cite[Section~2.3]{mok2025logarithmicfultonmacphersonconfigurationspaces}. 

\begin{definition}\label{defn:Aug}
    Fix points $u,v\in \mathbb{Z}^2$ which generate a lattice $\Lambda$. Given a configuration of points $(p_1,...,p_k)$ upon $\mathbb{R}^2/\Lambda$, we define the \textit{augmented subdivision} of $\mathbb{R}^2/\Lambda$ as follows.
    \begin{enumerate}[leftmargin = *]
        \item First define a subdivision $\Sigma^\text{pre}$ of $\mathbb{R}^2/\Lambda$ which is a polyhedral complex\footnote{So the cone over this structure is a subdivision of the integral affine manifold $\mathpzc{A}$ by a cone complex.} as follows. First define a $\Lambda$ equivariant subdivision of $\mathbb{R}^2$ by placing a vertex at each point of $\Lambda$, and an edge between any two vertices $n_1,n_2$ such that $n_1-n_2 \in \{\pm u,\pm v\}$. The resulting subdivision of tropical abelian varieties specifies a logarithmic alteration of the logarithmic abelian variety $A\times_S\mathbb{A}^1$.
        \item Next for each element $p_1,...,p_k$ of our configuration we impose a vertex at $p_i+\lambda$ for all $\lambda \in \Lambda$, and further pass to the minimal subdivision such that the straight line between $p_i +\lambda_1$ and $p_i + \lambda_2$ is a union of cones whenever $\lambda_1-\lambda_2$ is either a multiple of $u$ or $v$. 
        \item The above subdivision is $\Lambda$ equivariant and thus descends to the quotient.
    \end{enumerate}
\end{definition}

Next we define a subdivision of the tropical Hilbert scheme $\mathpzc{A}^{[n+1]}$, and thus a logarithmic alteration to the logarithmic Hilbert scheme of $A\times_S\mathbb{A}^1.$

\begin{definition}\label{defn:SquareSubdiv}
    The square subdivision of $\mathpzc{A}^{[n+1]}$ is obtained by pushing forward the cone complex structure on its universal family described in Definition~\ref{defn:Aug}. The square subdivision of $\operatorname{Kum}_{2n}(\mathpzc{A})$ is the subdivision obtained by pulling back the square subdivision of $\mathpzc{A}^{[n+1]}$ along the tautological inclusion.
\end{definition}

\begin{remark}
    The procedure of Definition~\ref{defn:SquareSubdiv} defines the trivial polyhedral subdivision if performed upon the integral affine manifold associated to a type~II degeneration.
\end{remark}

\begin{example}
    In the $n=1$ case the resulting polyhedral subdivision of a sphere (and thus dual intersection complex) is obtained by connecting the red dots of the central sphere of Figure \ref{fig:TropGenKummerFibres} to form a square. 
\end{example}

\begin{lemma}\label{lem:tidyup}
    In the $n+1=s$ case, the square subdivision has $$n(n+1)+1 +\frac{n(n+1)(n^3+6n^2+3n+2)}{24}$$ vertices, and thus we have specified a type III degeneration whose special fibre this many components.
\end{lemma}

\begin{proof}
    Note that points in $\operatorname{Kum}_{2n}(\mathpzc{A})$ biject with tuples of $n+1$ points on $\mathbb{R}^2/\Lambda$ which sum to zero. We first characterise when a point in $\operatorname{Sym}^{n+1}(\mathbb{R}^2/\Lambda)$ gives rise to a vertex in the square subdivision. Let $(p_1,...,p_{n+1})\in \operatorname{Sym}^{n+1}(\mathbb{R}^2/\Lambda)$ and let $S_u$ (resp $S_v$) be the set of points $t$ in $[0,1)$ such that there exists some integer $i\in \{1,...,n+1\}$ and real number $t'$ such that $p_i$ can be represented by $tu + t'v$ (resp. $t'u + tv$) in $\mathbb{R}^2$. 
    
    We claim that if $(p_1,..,p_{n+1})$ is a vertex then both $S_u$ and $S_v$ contain at most one real number which is not $0$. Assume for contradiction that $S_u$ contains two non-zero real numbers $a,b$, and suppose that $n_a$ of the $p_j$ can be written $au + xv$ for some real $x$ and $n_b$ of the $p_j$ can be written $bu + xv$.  To reach a contradiction we will construct a path starting at $(p_1,...,p_{n+1})$ on which the combinatorial type is constant. This is done by changing points as follows:
    \begin{enumerate} 
        \item Whenever $p_j$ can be written $au + x_jv$ for some real number $x$, replace $p_j$ with $(a + \epsilon n_b)u + x_j v$. 
        \item Whenever $p_j$ can be written $bu + x_jv$ for some real number $x$, replace $p_j$ with $(b - \epsilon n_a)u + x_j v$.
    \end{enumerate}
    The condition that points sum to zero is preserved in the $u$ direction by construction, and in the $v$ direction because those coordinates were not changed. Varying $\epsilon \geq 0$ moves our point without changing its combinatorial type. Thus $(p_1,...,p_j)$ is not a vertex and we have a contradiction.

    Conversely, if $S_u$ and $S_v$ contain at most one non-zero point, then the sum zero condition means that $(p_1,...,p_j)$ cannot be deformed without changing its combinatorial type - it is a vertex.
    We now count cases.
    \begin{enumerate}[leftmargin=*]
        \item If $S_u= S_v = \{0\}$ then we know that $(p_1,...p_{n+1})$ is the origin - we have precisely one vertex.
        \item If $S_v = \{0\}$ but we are not in case 1, then suppose there are $w$ points of the form $au + xv$ with $a\ne 0$. We know that $wa=0$ mod $1$ and $wa\ne 0$, so there are $w-1$ possibilities for the value of $a$. Summing over all possible values of $w$ gives $$\sum_{w=2}^{n+1}(w-1)=\frac{(n+1)n}{2}$$ vertices.
        \item If $S_u = \{0\}$ but we are not in case (1), then as in case (2) we have ${n(n+1)}/{2}$ vertices.
        \item If neither $S_u$ nor $S_v$ is $\{0\}$ then we proceed as follows. The number of tuples $(p_1,..,p_{n+1})$ with $n_a\geq 1$ vertices with $u$ coordinate $a$, and $n_b\geq 1$ vertices with $v$ coodinate $b$ is the number of $2x2$ non-negative integer matrices whose second row sums to $n_a$, second column sums to $n_b$, and whose entries sum to ${n+1}$. Indeed, we are looking for ways to allocate ${n+1}$ vertices over the four possible positions $\{(a,0),(0,0), (0,b), (a,b)\}.$ There are $$1 + \operatorname{min}\{n_a,n_b,{n+1}-n_a,{n+1}-n_b\}$$ such matrices.\footnote{Four dimensions of 2x2 matrix, three linear relations gives one degree of freedom which is adding a multiple of $(1,-1;-1,1)$. The formula is then a direct computation.} Once again there are $(n_a-1)$ (resp. $(n_b-1)$) possible values of $a$ (resp. $b$). The total contribution is thus $$\sum_{n_a = 2}^{n+1}\sum_{n_b = 2}^{n+1} (1 + \operatorname{min}\{n_a,n_b,{n+1}-n_a,{n+1}-n_b\})(n_a-1) (n_b-1)
=
\frac{n(n+1)(n^3+6n^2+3n+2)}{24}.
$$
    \end{enumerate}
    The final result is obtained by summing the contribution from each of our four cases.
\end{proof}

\bibliographystyle{alpha}
\bibliography{bibliography}
\end{document}